\documentclass[a4paper,12pt]{report}

\usepackage{amsthm}
\usepackage{mathdots}
\usepackage{graphicx,color}

\theoremstyle{plain}
\newtheorem{theorem}{Theorem}[chapter]
\newtheorem{proposition}[theorem]{Proposition}

\newtheorem{corollary}[theorem]{Corollary}
\newtheorem{conjecture}[theorem]{Conjecture}

\theoremstyle{definition}
\newtheorem{definition}[theorem]{Definition}
\newtheorem{example}[theorem]{Example}

\theoremstyle{remark}
\newtheorem{remark}[theorem]{Remark}

\begin{document}


\begin{titlepage}
    \center
    \begin{Large}
    \begin{figure}[h]
    \centering
    \includegraphics[width=20mm]{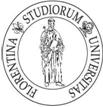}
    \end{figure}
    \textsc{ Universit\`a degli Studi di Firenze}
    \end{Large}\\
    \rule{9cm}{.4pt}\\
   \smallskip
 Dipartimento di Matematica ``Ulisse Dini''\\[1.1 cm]
\vfill
  \begin{Large} Dottorato di Ricerca in Matematica - XXV ciclo\end{Large}\\
  \bigskip
  \bigskip
  {\vfil
    \begin{Huge}
{The Torelli problem for Logarithmic bundles of hypersurface arrangements in the projective space}
    \end{Huge}%
    }
    \\[0.5 cm]
    \begin{Large}
    SSD MAT/03
    \end{Large}
    \\[1 cm]
    \begin{Large}
   {Elena Angelini}
    \end{Large}
\\[1cm]
\vfill
\begin{minipage}{0.49\textwidth}
\begin{center} \large
{\bf Coordinatore del Dottorato}\\
\smallskip
Prof. Alberto Gandolfi\\
Universit\`a degli Studi di Firenze 
\\[1.2cm]
{\bf Tutore}\\
\smallskip
Prof. Graziano Gentili\\
Universit\`a degli Studi di Firenze
\end{center}
\end{minipage}
\begin{minipage}{0.49\textwidth}
\begin{center}  \large
{\bf Cotutori}\\
\smallskip
Prof. Giorgio Ottaviani \\
Universit\`a degli Studi di Firenze 
\\[1.2cm]
Prof. Daniele Faenzi \\
Universit\'e de Pau et des Pays de l'Adour
\end{center}
\end{minipage}
\\[1 cm]

{\large Anni 2010/2012}\\
\end{titlepage}

\renewcommand{\baselinestretch}{1.66}

\pagenumbering{roman}
\tableofcontents
\listoffigures

\chapter*{Acknowledgements}

This Ph.D. thesis would never have been completed without the guidance of Professor Giorgio Ottaviani, who took me as his student since I was undergraduate and introduced me to Algebraic Geometry. I have an enormous debt of gratitude to him for all the time spent on me and all his very useful suggestions. He introduced me to the \emph{Torelli problem for logarithmic bundles of hyperplanes} and suggested to me to think about the higher degree situations. \\
In this sense, I would like to thank Professor Vincenzo Ancona, since my results are based on some handwritten unpublished notes of a talk that he gave in 1998 and that he allowed me to examine. \\
I am also in great debt to my second advisor, Professor Daniele Faenzi: in September 2010 he came to Florence to give a course on logarithmic bundles, which turned to be very helpful to me. I thank him for the hospitality that he offered to me and for the hours that he spent in spring 2011,  when I was a visiting student in Pau at the ``Laboratoire de Math\'ematiques et de leurs applications"; without its interesting hints, the descriptions of the arrangements with a sufficently large number of ``objects" would never have been possible. Moreover I am grateful to him for the opportunity of presenting in a poster some results of mine during the school ``Arrangements in Pyr\'en\'ees" which was organized in Pau in June 2012. \\
A special thank goes to Jean Vall\`es, Masahiko Yoshinaga, Paolo Aluffi, Hal Schenck and Christophe Ritzenthaler, for interesting discussions. Another particular thank goes to the director of the PhD school in Florence, Professor Alberto Gandolfi, for his support and his help.\\
Finally I would like to mention the  Scuola Normale Superiore and the Centro di Ricerca Matematica Ennio de Giorgi in Pisa, the CIRM and the Fondazione Bruno Kessler in Trento: I am grateful to these organizations for giving me the opportunity of attending many interesting courses, schools and workshops during these years.  

\newpage
\thispagestyle{empty} 
\begin{flushright}
\null\vspace{\stretch{1}}
{\sl ... E se vi capita di passare di l\`a, vi supplico, non vi affrettate,} \\
{\sl fermatevi un momento sotto le stelle! ... } \\ 
\vspace{0.3cm}
Antoine de Saint-Exup\'ery 
\null\vspace{\stretch{2}}\null
\end{flushright}

\newpage
\thispagestyle{empty} 
\begin{flushright}
\null\vspace{\stretch{1}}
{\sl To Lorenzo\\
\rightline{To my family}} 
\null\vspace{\stretch{2}}\null
\end{flushright}

\pagenumbering{arabic}
\setcounter{page}{0}

\chapter{Introduction}
\label{ch:intro}

Let $ X $ be a non singular algebraic variety of dimension $ n $ and let $ \mathcal{D} $ be a union of $ \ell $ distinct irreducible hypersurfaces on $ X $, which we call an \emph{arrangement} on $ X $.  We can associate to $ \mathcal{D} $ the \emph{sheaf of differential $ 1 $-forms with logarithmic poles on $ \mathcal{D} $}, denoted by $ \Omega^{1}_{X}(\log \mathcal{D}) $. This sheaf was introduced by Deligne in~\cite{De} for an arrangement with normal crossings. In this case, for all $ x \in X $, the space of sections of $ \Omega^{1}_{X}(\log \mathcal{D}) $ near $ x $ is defined by 
$$ <d\log z_{1}, \ldots, d\log z_{k}, dz_{k+1}, \ldots, dz_{n}>_{\mathcal{O}_{X,x}} $$
where $ z_{1}, \ldots, z_{n} $ are local coordinates such that $ \mathcal{D} = \{z_{1} \cdots z_{k} = 0\} $. Moreover $ \Omega^{1}_{X}(\log \mathcal{D}) $ turns out to be a vector bundle over $ X $, which is simply called \emph{logarithmic bundle}. If $ \mathcal{D} $ has not normal crossings, there is a more general definition of $ \Omega^{1}_{X}(\log \mathcal{D}) $ given by Saito in~\cite{Sa}. \\
Once we construct the correspondence
\begin{equation}\label{eq:Tmap}
\mathcal{D} \longrightarrow \Omega^{1}_{X}(\log \mathcal{D})
\end{equation}
a natural, interesting question is whether $ \Omega^{1}_{X}(\log \mathcal{D}) $ contains enough information to recover $ \mathcal{D} $. For this reason we can talk about the \emph{Torelli problem} for $ \Omega^{1}_{X}(\log \mathcal{D}) $, since the injectivity of the map in (\ref{eq:Tmap}) is investigated. In particular, if the isomorphism class of $ \Omega^{1}_{X}(\log \mathcal{D}) $ determines $ \mathcal{D} $, then $ \mathcal{D} $ is called a \emph{Torelli arrangement}. \newline
The first situation that has been analyzed is the case of hyperplanes in the complex projective space $ \mathbf{P}^{n} $. Hyperplane arrangements play a central role in geometry, topology and combinatorics (\cite{OT},~\cite{Arn}). In 1993 Dolgachev and Kapranov gave an answer to the Torelli problem when $ \mathcal{H} = \{H_{1}, \ldots, H_{\ell}\} $ is an arrangement of hyperplanes with normal crossings,~\cite{Do-Ka}. They proved that if $ \ell \leq n+2 $ then two different arrangements give always the same logarithmic bundle, moreover if $ \ell \geq 2n+3 $ then we can reconstruct $ \mathcal{H} $ from $ \Omega^{1}_{\mathbf{P}^n}(\log \mathcal{H}) $ unless the hyperplanes in $ \mathcal{H} $ don't osculate a rational normal curve $ \mathcal{C}_{n} $ of degree $ n $ in $ \mathbf{P}^n $, in which case $ \Omega^{1}_{\mathbf{P}^n}(\log \mathcal{H}) $ is isomorphic to $ E_{\ell - 2}(\mathcal{C}_{n}^{\vee}) $, the \emph{Schwarzenberger bundle of degree $ \ell-2 $ associated to $ \mathcal{C}_{n}^{\vee} $}. In 2000 Vall\`es extended the latter result to $ \ell \geq n+3 $,~\cite{V}: while Dolgachev and Kapranov studied the set of \emph{jumping lines} (\cite{Ba},~\cite{Hu}) of $ \Omega^{1}_{\mathbf{P}^n}(\log \mathcal{H}) $, Vall\`es characterized $ \mathcal{H} $ as the set of \emph{unstable hyperplanes} of the logarithmic bundle, i.e.
$$ \{ H \subset \mathbf{P}^{n} \, hyperplane \, | \, H^{0}(H, \Omega^{1}_{\mathbf{P}^n}(\log \mathcal{H})_{|_{H}}^{\vee}) \not = \{0\} \}. $$    
A few years ago, hyperplane arrangements without normal crossings have been investigated, in particular Faenzi-Matei-Vall\`es in~\cite{FMV} studied the Torelli problem for the subsheaf $ \widetilde{\Omega}^{1}_{\mathbf{P}^n}(\log \mathcal{H}) $ of $ \Omega^{1}_{\mathbf{P}^n}(\log \mathcal{H}) $ and proved that $ \mathcal{H} $ is a Torelli arrangement if and only if $ H_{1}, \ldots, H_{\ell} $, seen as point in the dual projective space, don't belong to a \emph{Kronecker-Weierstrass variety of type $ (d, s) $}, which is essentially the union of a smooth rational curve of degree $ d $ with s linear subspaces. \\
In this thesis, after recalling the fundamental definitions and the main classical results on the subject, we consider arrangements of higher degree hypersurfaces with normal crossings on $ \mathbf{P}^n $. \\
In chapter $ 4 $ we describe some important properties of the logarithmic bundle $ \Omega^{1}_{\mathbf{P}^n}(\log \mathcal{D}) $, in particular in theorem~\ref{T:Ancona} is proved that $ \Omega^{1}_{\mathbf{P}^n}(\log \mathcal{D}) $ admits a resolution of lenght $ 1 $ which is a very important tool for our investigations. Moreover, this resolution allows us to find again the Torelli type results in the case of hyperplanes. \\
The main results of this thesis are collected in chapters $ 5, \,6, \, 7 $ and $ 8 $. \\
Precisely, chapter $ 5 $ is devoted to arrangements of conics in $ \mathbf{P}^2 $: if $ \ell \geq 9 $, then we can recover the conics in $ \mathcal{D} $ as \emph{unstable conics} of $ \Omega^{1}_{\mathbf{P}^2}(\log \mathcal{D}) $, unless the hyperplanes in $ \mathbf{P}^5 $ corresponding to $ \mathcal{D} $ through the \emph{quadratic Veronese map} satisfy further hypothesis (theorem~\ref{t54}). The notion of unstable conic is inspired to the one of unstable hyperplane given for a Steiner bundle. \\
On the contrary, if $ \ell = 1 $ or $ \ell = 2 $ then we find arrangements which are not of Torelli type (theorems~\ref{T:1conica} and~\ref{T:coppieconiche}), in particular in the second case, by using the \emph{simultaneous diagonalization}, we prove that two pairs of conics are associated to isomorphic logarithmic bundles if and only if they have the same four tangent lines. \\
In chapter $ 6 $ we extend theorem~\ref{t54} to arrangements of a \emph{sufficiently large} number of hypersurfaces of higher degree in $ \mathbf{P}^n $ (theorem \ref{t64}) and in chapter $ 7 $ we generalize the methods used for one conic and pairs of conics to quadrics in $ \mathbf{P}^n $ (theorems \ref{T:1quadric} and \ref{T:pairquad}). \\
Finally, chapter $ 8 $ is devoted to arrangements made of lines and conics in $ \mathbf{P}^2 $, in particular the cases of a conic and a line (corollary \ref{C:line-conic}), of a conic and two lines (corollary \ref{C:2lines-conic}) and of a conic and three lines (theorem \ref{T:3lines-conic}) are investigated.

\chapter{Preliminaries}
\label{ch:preliminaries}

\section{Arrangements with normal crossings}

Let $ X $ be a smooth algebraic variety, we give the following:

\begin{definition}\label{def:logsheafde}
A \emph{reduced and effective divisor} on $ X $ is a family
$$ \mathcal{D} = \{D_{1}, \ldots, D_{\ell}\} $$
of irreducible hypersurfaces of $ X $ such that $ D_{i} \not= D_{j} $ for all $ i,j \in \{1, \ldots, \ell\} $, $ i \not= j $. $ \mathcal{D} $ is also called \emph{arrangement} on $ X $.
\end{definition}

\begin{example}
Let $ \mathcal{D} $ be an arrangement on the $ n $-dimensional complex projective space, which we simply denote by $ \mathbf{P}^n $. Each hypersurface $ D_{i} \in \mathcal{D} $ is defined as the zero locus of a homogeneous polynomial $ f_{i} $ of degree $ d_{i} $ in the variables $ x_{0}, \ldots, x_{n} $. Thus $ \mathcal{D} $ is given by the set of zeroes of $ $$\displaystyle \prod_{i=1}^{\ell}$$ f_{i} $, which is a polynomial of degree $ $$\displaystyle \sum_{i=1}^{\ell}$$ d_{i} $. In particular, if $ d_{i} = 1 $ for all $ i $, we talk about a \emph{hyperplane arrangement}, if all $ d_{i} $'s are equal to $ 2 $ we have an \emph{arrangement of quadrics} and so on. In particular, when $ n = 2 $ in the two previous cases we have to deal, respectively, with lines and conics. Also arrangements of hypersurfaces of different degrees are allowed.
\end{example}

In this setting \emph{arrangements with normal crossings} play a very special role. We have the following:
 
\begin{definition}
Let $ \mathcal{D} $ be an arrangement on $ X $. We say that $ \mathcal{D} $ has \emph{normal crossings} if it is locally isomorphic (in the sense of holomorphic local coordinates changes) to a union of coordinate hyperplanes of $ \mathbf{C}^n $.
\end{definition}

\begin{example}
Let $ \mathcal{H} = \{H_{1}, \ldots, H_{\ell}\} $ be a hyperplane arrangement in $ \mathbf{P}^n $. $ \mathcal{H} $ has normal crossings if and only if $ codim(H_{i_{1}} \cap \ldots \cap H_{i_{k}}) = k $ for any $ k \leq n+1 $ and any $ 1 \leq i_{1} < \ldots < i_{k} \leq \ell $. In particular, if $ n = 2 $ a pair of distinct lines has always normal crossings but three lines meeting in point don't. 
\end{example}

\begin{example}
In the complex projective plane let $ \mathcal{D} = \{r, C\} $, where $ r $ is a line and $ C $ is a smooth conic. Then $ \mathcal{D} $ has normal crossings if and only if $ r $ is not tangent to $ C $ (see figure $ 8.1 $). Moreover, if $ \mathcal{D} $ is made of a cubic with a node, then it has normal crossings, but if the cubic has a cusp it doesn't.
\end{example}

\begin{figure}[h]
    \centering
    \includegraphics[width=10mm]{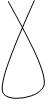}
    $\quad\quad\quad\quad\quad\quad$
    \includegraphics[width=20mm]{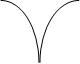}
    \caption{Nodal cubic and cusp cubic}
\end{figure}

\vfill\eject

\section{Logarithmic bundles}

Let $ X $ be a smooth algebraic variety and let $ \mathcal{D} $ be an arrangement with normal crossings on $ X $. In order to introduce the notion of \emph{sheaf of logarithmic forms} on $ \mathcal{D} $ we will refer to Deligne (\cite{De0},~\cite{De}). This is not the unique way to describe these sheaves, there are also other definitions for more general divisors (\cite{Sa},~\cite{Schenck}) that are equivalent to this one for arrangements with normal crossings. In this sense see also section $ 3.3 $.

Let's start with some notations. Let $ U = X - \mathcal{D} $ be the complement of $ \mathcal{D} $ in $ X $ and let $ j: U \hookrightarrow X $ be the embedding of $ U $ in $ X $. We denote by $ \Omega_{U}^1$ the sheaf of holomorphic differential $ 1 $-forms on $ U $ and by $ j_{\ast}\Omega_{U}^1 $ its direct image sheaf on $ X $. We remark that, since $ \mathcal{D} $ has normal crossings, then for all $ x \in X $ there exists a neighbourhood $ I_{x} \subset X $ such that $ I_{x} \cap \mathcal{D} = \{z_{1} \cdots z_{k} = 0 \} $, where $ \{z_{1}, \ldots, z_{k}\} $ is a part of a system of local coordinates. 

We have the following:

\begin{definition}\label{d:2.6}
We call \emph{sheaf of differential $ 1 $-forms on $ X $ with logarithmic poles on $ \mathcal{D} $} the subsheaf of $ j_{\ast}\Omega_{U}^1 $, denoted by $ \Omega_{\mathbf{X}}^{1}(\log \mathcal{D}) $, such that, for all $ x \in X $,
$ \Gamma(I_{x},\Omega_{\mathbf{X}}^{1}(\log \mathcal{D})) $ is given by
$$ \{s \in \Gamma(I_{x},j_{\ast}\Omega_{U}^1) \,|\, s = \displaystyle{\sum_{i=1}^k u_{i} d\log z_{i}} \, +  \displaystyle{\sum_{i=k+1}^n v_{i} dz_{i}} \} $$
where $ u_{i}, v_{i} $ are locally holomorphic functions and $ d \log z_{i} = \displaystyle{{dz_{i}} \over {z_{i}}} $.
\end{definition}

\begin{remark}
Every $ s \in \Gamma(I_{x},\Omega_{\mathbf{X}}^{1}(\log \mathcal{D})) $ is a meromorphic differential $ 1 $- form with at most simple poles on $ \mathcal{D} $, namely
$$ (z_{1} \cdots z_{k})s = (z_{1} \cdots z_{k}) \displaystyle{\sum_{i=1}^k u_{i} d\log z_{i}} + (z_{1} \cdots z_{k})\displaystyle{\sum_{i=k+1}^n v_{i} dz_{i}} $$
is holomorphic on $ I_{x} $. So we are allowed to call $ s $ a \emph{logarithmic form} on $ \mathcal{D} $. It's not hard to check that also $ ds $ has this property. As we can see in~\cite{De}, it holds also that if $ s $ is a local section of $ j_{\ast}\Omega_{U}^1 $ such that $ s $ and $ ds $ have at most simple poles on $ \mathcal{D} $, then $ s \in \Gamma(I_{x},\Omega_{\mathbf{X}}^{1}(\log \mathcal{D})) $.
\end{remark}

\begin{remark}
Let $ X $ be a smooth algebraic variety and let $ \mathcal{D} $ be an arrangement with normal crossings on $ X $. Then $ \Omega_{\mathbf{X}}^{1}(\log \mathcal{D}) $ is a locally free sheaf of rank $ n = dim X $,~\cite{De}. So, $ \Omega_{\mathbf{X}}^{1}(\log \mathcal{D}) $ can be regarded as a rank-$ n $ vector bundle on $ X $ and it is called \emph{logarithmic bundle attached to $ \mathcal{D} $}.
\end{remark}

\section{Torelli problem for logarithmic bundles}

Given a smooth algebraic variety $ X $, we are able to map an arrangement with normal crossings on $ X $ to a logarithmic vector bundle on $ X $:
\begin{equation}\label{eq:Torellimap}
\mathcal{D} \longmapsto \Omega_{\mathbf{X}}^{1}(\log \mathcal{D}). 
\end{equation}
A natural question arises from this contruction: is it true that isomorphic logarithmic bundles come from the same arrangement? If the answer is positive, then we say that $ \mathcal{D} $ is an \emph{arrangement of Torelli type}, or a \emph{Torelli arrangement}. This is the so called Torelli problem for logarithmic bundles. 

Actually this is not the ``original'' Torelli problem: we talk about \emph{problem of Torelli type} whenever we have to deal with the injectivity of certain map.  The history of this kind of problems goes back to 1913, when Torelli asked wether two curves are isomorphic if they have the same periods,~\cite{To}. 

The mathematical literature on this topic is enormous, we will focus our attention on the case of logarithmic bundles. 

In the next chapter we will discuss the main results concerning arrangements of hyperplanes in the complex projective space: a large number of mathematicians worked and are still investigating on this subject, I will mainly refer to Dolgachev, Kapranov (\cite{Do},~\cite{Do-Ka}), Ancona, Ottaviani (\cite{AppuntiAncona},~\cite{AO}), Faenzi, Matei, Vall\`es (\cite{FMV},~\cite{V}).

\chapter{The case of hyperplanes in the projective space}
\label{ch:the case of hyperplanes in the projective space}

\section{Logarithmic bundles of hyperplanes with normal crossings: Steiner and Schwarzenberger bundles}

Let $ \mathcal{H} = \{H_{1}, \ldots, H_{\ell} \} $ be an arrangement of $ \ell $ hyperplanes with normal crossings on $ \mathbf{P}^n $. Let's describe the main features of the corresponding logarithmic bundle $ \Omega_{\mathbf{\mathbf{P}^n}}^{1}(\log \mathcal{H}) $. First of all we have the following:

\begin{proposition}
Let $ x \in \mathbf{P}^n $ and let $ I_{x} \subset X $ a neighbourhood of $ x $ such that $ I_{x} \cap \mathcal{H} = \{z_{1} \cdots z_{k} = 0\} $, as in definition~\ref{d:2.6}. 
We denote by 
$$ res : \Omega_{\mathbf{P}^n}^{1}(\log \mathcal{H}) \longrightarrow \, \displaystyle\bigoplus_{i=1}^\ell \mathcal{O}_{H_{i}} $$
the \emph{Poincar\'e residue morphism}, that is the map defined locally by
$$ \displaystyle\sum_{i=1}^k u_{i} d\log z_{i}\,+ \displaystyle\sum_{i=k+1}^n v_{i} dz_{i}\,\longmapsto \, (u_{1}(x), \ldots, u_{k}(x),0, \ldots,0). $$
Then
\begin{equation}\label{eq:resiperpiani}
0 \longrightarrow \Omega_{\mathbf{P}^n}^1 \longrightarrow \Omega_{\mathbf{P}^n}^{1}(\log \mathcal{H}) \buildrel \rm res \over \longrightarrow \, \displaystyle\bigoplus_{i=1}^{\ell} \mathcal{O}_{H_{i}} \longrightarrow 0 
\end{equation}
is a short exact sequence of sheaves on $ \mathbf{P}^n $.
\end{proposition} 

\begin{proof}
Every local section $ \omega = \displaystyle\sum_{i=1}^n a_{i} dz_{i} $ of $ \Omega_{\mathbf{P}^n}^1 $ can be considered as a local section of $ \Omega_{\mathbf{\mathbf{P}^n}}^{1}(\log \mathcal{H}) $ by writing: 
$$ \omega = \displaystyle\sum_{i=1}^k a_{i}z_{i}d\log z_{i}\,+ \displaystyle\sum_{i=k+1}^n a_{i} dz_{i}. $$
So the holomorphic differential $ 1 $-forms on $ \mathbf{P}^n $ belong to the kernel of the Poincar\'e residue map. It's not hard to see that also the converse is true.
\end{proof}

\begin{remark}
In chapter $ 4 $ we will see that also logarithmic bundles attached to arrangements of smooth hypersurfaces with normal crossings on $ \mathbf{P}^n $ admit an exact sequence similar to (\ref{eq:resiperpiani}). 
\end{remark}

If we consider an arrangement $ \mathcal{H} $ made of a \emph{sufficiently large} number of hyperplanes, the corresponding logarithmic bundle satisfies a further condition: it is a \emph{Steiner bundle} on $ \mathbf{P}^n $. In this sense we have the following:

\begin{definition}\label{d:3.3}
Let $ S $ be a rank-$ n $ vector bundle over $ \mathbf{P}^n $, we say that $ S $ is a \emph{Steiner bundle} if it appears in a short exact sequence
\begin{equation}\label{eq:steinerseq}
0 \longrightarrow I \otimes \mathcal{O}_{\mathbf{P}^n}(-1) \buildrel \rm \tau \over \longrightarrow W \otimes \mathcal{O}_{\mathbf{P}^n}  \longrightarrow S \longrightarrow 0 
\end{equation}
where $ I $ and $ W $ are complex vector spaces of dimension $ k $ and $ n+k $ respectively. The map $ \tau $ is uniquely determined by a tensor $ t  \in (\mathbf{C}^{n+1})^{\vee} \otimes I^{\vee} \otimes W $ in such a way that $ \tau $ is injective on each fiber.
\end{definition}

\begin{remark}
The short exact sequence (\ref{eq:steinerseq}) associated to a Steiner bundle $ S $ can be also represented as 
\begin{equation}\label{eq:steinerfacile}
0 \longrightarrow \mathcal{O}_{\mathbf{P}^n}(-1)^{k} \longrightarrow \mathcal{O}_{\mathbf{P}^n}^{n+k} \longrightarrow S \longrightarrow 0
\end{equation}
where $ k $ is a positive integer. 
\end{remark}

We denote by $ S_{n,k} $ the family of Steiner bundles with parameters $ n $ and $ k $. The elements of $ S_{n,k} $ have a very important property: they are \emph{stable} in the sense of Mumford-Takemoto (slope-stability). Let's recall this notion of stability.

\begin{definition}
For a torsion-free coherent sheaf $ E $ over $ \mathbf{P}^n $ let
$$ \mu(E) = \displaystyle{{c_{1}(E)} \over {rk(E)}} $$
be the \emph{slope} of $ E $. \\
We say that $ E $ is \emph{stable} (resp. \emph{semistable}) if for all coherent subsheaves $ F \subset E $ such that $ 0 < rk F < rk E $ we have
\begin{equation}\label{eq:stabilitycondition}
\mu(F) < \mu(E) \quad\quad\quad (resp. \, \leq). 
\end{equation}
\end{definition}

The theorem of Bohnhorst-Spindler (\cite{BS}) gives a very useful criterion in order to check condition (\ref{eq:stabilitycondition}) in the case of rank-$ n $ vector bundles over $ \mathbf{P}^n $ with homological dimension equal to $ 1 $. We have the following:

\begin{theorem}\label{T:3.5}$($\emph{Bohnhorst-Spindler 1992,~\cite{BS}}$)$\\
Let $ E $ be a rank-$ n $ vector bundle on $ \mathbf{P}^n $ with minimal resolution
$$ 0 \longrightarrow \bigoplus_{i=1}^{k} \mathcal{O}_{\mathbf{P}^n}(a_{i}) \longrightarrow \bigoplus_{j=1}^{n+k} \mathcal{O}_{\mathbf{P}^n}(b_{j}) \longrightarrow E \longrightarrow 0  $$
where $ a_{1} \geq \ldots \geq a_{k} $, $ b_{1} \geq \ldots \geq b_{n+k} $ and $ a_{1} < b_{n+1}, \ldots, a_{k} < b_{n+k} $. \\
The following facts are equivalent:
\begin{itemize}
\item[$1)$] $ E $ is stable $($resp. semistable$)$;
\item[$2)$] $ b_{1} < $ $($resp $ \leq $$)$ $ \mu(E) = \displaystyle{{1} \over {n}}\left(\displaystyle{\sum_{j=1}^{n+k} b_{j} - \sum_{i=1}^{k} a_{i}}\right) $. 
\end{itemize}
Moreover, if $ b_{1} = \ldots = b_{n} $ then $ E $ is stable in any case.
\end{theorem}

\begin{remark}\label{r:3.7}
Theorem~\ref{T:3.5} implies that if $ S \in S_{n,k} $ then $ S $ is stable. 
\end{remark}

Now we are ready to state and to give a sketch of the proof of the result, due to Dolgachev and Kapranov (\cite{Do-Ka}), that we mentioned previously.

\begin{theorem}\label{T:3.8}$($\emph{Dolgachev-Kapranov 1993,~\cite{Do-Ka}}$)$ \\
Let $ \mathcal{H} = \{H_{1}, \ldots, H_{\ell}\} $ be an arrangement of hyperplanes with normal crossings such that $ \ell \geq n+2 $. Then $ \Omega_{\mathbf{P}^n}^{1}(\log \mathcal{H}) \in S_{n, \ell-n-1} $. 
\end{theorem}

\begin{proof}
We want to construct a map of sheaves $ \tau_{\mathcal{H}} $ as the one in (\ref{eq:steinerseq}) (with $ k = \ell -n -1 $) whose cokernel is isomorphic to $ \Omega_{\mathbf{P}^n}^{1}(\log \mathcal{H}) $. \\
In order to do that, let $ f_{1}, \ldots, f_{\ell} $ be  homogeneous forms of degree $ 1 $ in $ x_{0}, \ldots, x_{n} $ such that $ H_{i} = \{f_{i} = 0\} $ for all $ i \in \{1, \ldots, \ell \} $ and let 
\begin{equation}\label{eq:IH}
I_{\mathcal{H}} = \{ (\lambda_{1}, \ldots, \lambda_{\ell}) \in \mathbf{C}^{\ell}\, | \, \displaystyle{\sum_{i=1}^{\ell} \lambda_{i} f_{i} = 0}  \}
\end{equation}
\begin{equation}\label{eq:W}
W = \{ (\lambda_{1}, \ldots, \lambda_{\ell}) \in \mathbf{C}^{\ell}\, | \, \displaystyle{\sum_{i=1}^{\ell} \lambda_{i} = 0} \}. 
\end{equation}
Since $ \ell \geq n+2 $ we have that $ I_{\mathcal{H}} $ is non trivial. By using linear algebra computations we can see that, since $ \mathcal{H} $ is an arrangement with normal crossings, then $ dim(I_{\mathcal{H}}) = \ell - n -1 $. Moreover we have that $ dim (W) = \ell -1 $. \\
Now, let $ t_{\mathcal{H}}\in (\mathbf{C}^{n+1})^{\vee} \otimes I_{\mathcal{H}}^{\vee} \otimes W $ defined by
\begin{equation}\label{eq:fundamentaltensor}
t_{\mathcal{H}}(\lambda_{1}, \ldots, \lambda_{\ell}, v) = (\lambda_{1}f_{1}(v), \ldots, \lambda_{\ell}f_{\ell}(v))
\end{equation}
for all $ (\lambda_{1}, \ldots, \lambda_{\ell}) \in I_{\mathcal{H}} $, $ v \in \mathbf{C}^{n+1} $ and let $ \tau_{\mathcal{H}} $ the corresponding map of sheaves, $t_{\mathcal{H}}$ is called the \emph{fundamental tensor of} $ \mathcal{H} $. By using the hypothesis of normal crossings, it's not hard to prove that, for all $ v \in (\mathbf{C^{\ast}})^{n+1} $, the fiber $ t_{\mathcal{H}}(v) $ of $ \tau_{\mathcal{H}} $ over $ [v] \in \mathbf{P}^n $ is an injective map, so that $ dim (Im\,t_{\mathcal{H}}(v)) = \ell-n-1 $. We want to construct an isomorphism between the vector spaces $ Coker\,t_{\mathcal{H}}(v) $ and $ \Omega_{\mathbf{P}^n}^{1}(\log \mathcal{H})_{[v]} $. So consider the map
$$ \pi_{v} : W \longrightarrow \Omega_{\mathbf{P}^n}^{1}(\log \mathcal{H})_{[v]} $$
$$ (\lambda_{1}, \ldots, \lambda_{\ell}) \longmapsto \displaystyle{\sum_{i=1}^{\ell} \lambda_{i}(d\log f_{i})_{|_{[v]}}}. $$
Since the subspace of $ H^{0}(\mathbf{P}^n,\Omega_{\mathbf{P}^n}^{1}(\log \mathcal{H})) $ made of all sections that vanish at $ [v] $ has dimension $ \ell - n -1 $, we have that $ \pi_{v} $ is surjective; in particular $ dim (Ker\, \pi_{v}) = dim (Im\,t_{\mathcal{H}}(v)) $. In order to conclude the proof it suffices to show that $ Im\,t_{\mathcal{H}}(v) \subset Ker\, \pi_{v} $. So, let $ (\mu_{1}, \ldots, \mu_{\ell}) \in Im\,t_{\mathcal{H}}(v) $, that is there exists $ (\lambda_{1}, \ldots, \lambda_{\ell}) \in I_{\mathcal{H}} $ such that $ \mu_{i} = \lambda_{i} f_{i}(v) $. We have that
$$ \displaystyle{\sum_{i=1}^{\ell} \mu_{i} (d \log f_{i})_{|_{[v]}}} = \displaystyle{\sum_{i \,  s.t. \, [v] \notin H_{i}}} \mu_{i} \left({{df_{i}} \over {f_{i}}}\right)_{|_{[v]}}  $$
is a regular $1$-form in a neighbourhood of $ [v] $ and, for all tangent vectors $ \xi \in T_{[v]} \left(\bigcap_{i \,  s.t. \, [v] \in H_{i}} H_{i}\right) $,
$$ \displaystyle{\sum_{i=1}^{\ell} \mu_{i} (d \log f_{i})_{|_{[v]}}}(\xi) = \displaystyle{\sum_{i\,  s.t. \, [v] \notin H_{i}}}\mu_{i}{{f_{i}(\xi)} \over {f_{i}(v)}} = \displaystyle{\sum_{i\,  s.t. \, [v] \notin H_{i}}} \lambda_{i}f_{i}(\xi) = 0. $$ 
These two conditions imply that $ \displaystyle{\sum_{i=1}^{\ell}} \mu_{i} (d \log f_{i}) $ is zero at $ [v] $, as desired.
\end{proof}

Logarithmic bundles attached to arrangements of hyperplanes with normal crossings are strictly related to another class of vector bundles over $ \mathbf{P}^n $: the family of \emph{Schwarzenberger bundles},~\cite{Sch2} and~\cite{Sch1}. 

Let's introduce some preliminary notations: we denote by $ (\mathbf{P}^n)^{\vee} $ the dual variety of $ \mathbf{P}^n $ and by $ \mathbf{F} $ the incidence variety point-hyperplane of $ \mathbf{P}^n $, that is
$$ \mathbf{F} = \{ (x , \, y) \in \mathbf{P}^n \times {(\mathbf{P}}^n)^{\vee} \, | \, x \in H_{y} \} $$
$$ \quad\quad\ \buildrel \rm p \over \swarrow \quad \buildrel \rm q \over \searrow $$
$$ \quad\quad \, \, \,\,  \mathbf{P}^n \quad\quad \,\,\, {(\mathbf{P}}^n)^{\vee} $$
where $ H_{y} \subset \mathbf{P}^n $ is the hyperplane ``defined" by the point $ y \in {(\mathbf{P}}^n)^{\vee} $ and $ p,q $ are the canonical projection morphisms. \\
Let $ \mathcal{C}_{n} $ be a \emph{rational normal curve} of degree $ n $ in $ {\mathbf{P}}^n $, that is the image of the map
$$ \nu_{n} : {\mathbf{P}}^{1} \longrightarrow {\mathbf{P}}^{n} $$
$$ [x_{0}, x_{1}] \longmapsto [A_{0}, \ldots, A_{n}] $$
where $ \{A_{0}, \ldots, A_{n}\} $ is an arbitrary basis for the space of homogeneous polynomials of degree $ n $ in the variables $ x_{0},x_{1} $. To give such a curve is equivalent to make a choice of an isomorphism between the vector spaces $ {\mathbf{C}}^{n+1} $ and $ S^{n}{\mathbf{C}}^{2} $. Moreover, denote by $ {\mathcal{C}_{n}^{\vee}} \subset ({\mathbf{P}}^{n})^{\vee} $ the dual curve of $ \mathcal{C}_{n} $. \\
Consider the following diagram
$$ q^{-1}({\mathcal{C}_{n}^{\vee}}) = \{(x , \, y) \in \mathbf{F} \, | \,x \in H_{y} \wedge y \in {\mathcal{C}_{n}^{\vee}} \} $$
$$ \quad\ \buildrel \rm {\overline{p}} \over \swarrow \quad \buildrel \rm {\overline{q}} \over \searrow $$ 
$$ \quad \, \, \,  \mathbf{P}^n \quad\quad\quad {\mathcal{C}_{n}^{\vee}} $$
where $ \overline{p},\overline{q} $ are the restrictions of $ p,q $ to $ q^{-1}({\mathcal{C}_{n}^{\vee}}) \subset \mathbf{F} $. \\
We are ready to give the following:
 
\begin{definition}$($Schwarzenberger 1961, $\cite{Sch1})$ \\
Let $ m \in \mathbf{N} $, we call \emph{Schwarzenberger bundle of degree} $ m $ \emph{associated to} $ \mathcal{C}_{n}^{\vee} $ the rank-$ n $ vector bundle over $ \mathbf{P}^n $ given by
$$ E_{m}(\mathcal{C}_{n}^{\vee}) = \overline{p}_{\ast}\overline{q}^{\ast} \mathcal{O}_{\mathcal{C}_{n}^{\vee}}({{m} \over {n}}) $$
where $ \mathcal{O}_{\mathcal{C}_{n}^{\vee}}({{m} \over {n}}) $ denotes the line bundle over $ \mathcal{C}_{n}^{\vee} $ that corresponds to $ \mathcal{O}_{\mathbf{P}^{1}}(m) $ through the isomorphism $ \nu_{n} $ between $ \mathbf{P}^{1} $ and $ \mathcal{C}_{n}^{\vee} $. 
\end{definition}

If the degree of the Schwarzenberger bundle is sufficiently large then we get a Steiner bundle. In this sense we have the following: 

\begin{proposition}$($\emph{Schwarzenberger 1961,~\cite{Sch1}}$)$\\
If $ m \geq n $ then $ E_{m}(\mathcal{C}_{n}^{\vee}) \in S_{n, m-n+1} $.
\end{proposition}

\begin{proof}
It suffices to observe that, if $ m \geq n $, then $ E_{m}(\mathcal{C}_{n}^{\vee}) $ is defined by the short exact sequence
$$ 0 \longrightarrow \mathcal{O}_{\mathbf{P}^n}(-1)^{m-n+1} \buildrel \rm M \over \longrightarrow \mathcal{O}_{\mathbf{P}^n}^{m+1} \longrightarrow E_{m}(\mathcal{C}_{n}^{\vee}) \longrightarrow 0 $$
where $ M $ is the following $ (m+1) \times (m-n+1) $ matrix:
$$ M = \pmatrix{ 
0 & \cdots & x_{0} \cr
\vdots & \iddots & \vdots \cr
x_{0} & {} & x_{n} \cr
\vdots & \iddots & \vdots \cr
x_{n} & \cdots & 0 \cr}. $$
\end{proof}

\begin{remark}
We can find detailed descriptions of the previous result also in~\cite{Do-Ka} and~\cite{Wykno}. In particular, according to definition~\ref{d:3.3}, the vector spaces that characterize $ E_{m}(\mathcal{C}_{n}^{\vee}) $ are, respectively, $ S^{m-n}\mathbf{C}^2 $ and $ S^{m}\mathbf{C}^2 $ and the tensor $ t $ is the multiplication map
\begin{equation}\label{eq:Schwtensor}
t: S^{n}\mathbf{C}^2 \otimes S^{m-n}\mathbf{C}^2 \longrightarrow S^{m}\mathbf{C}^2.
\end{equation}
\end{remark}

The link between logarithmic bundles of hyperplanes with normal crossings on $ \mathbf{P}^n $ and Schwarzenberger bundles is described by the following result:

\begin{theorem}$($\emph{Dolgachev-Kapranov 1993,~\cite{Do-Ka}}$)$\\
If $ m \geq n $ then 
\begin{equation}\label{eq:isoSchw}
E_{m}(\mathcal{C}_{n}^{\vee}) \cong \Omega_{\mathbf{P}^n}^{1}(\log \mathcal{H})
\end{equation}
where $ \mathcal{H} = \{H_{1}, \ldots, H_{\ell}\} $ is an arrangement with $ \ell = m+2 $ hyperplanes with normal crossings such that $ H_{1}, \ldots, H_{\ell} $ osculate $ \mathcal{C}_{n} \subset \mathbf{P}^n $.
\end{theorem}

\begin{proof}
Let $ \mathcal{H} $ be a hyperplane arrangement satisfying the properties listed in the statement of the theorem; let $ I_{\mathcal{H}} $ and $ W $ be the vector spaces defined in (\ref{eq:IH}) and (\ref{eq:W}). In order to get (\ref{eq:isoSchw}), we have to construct two isomorphisms of vector spaces 
$$ \alpha : S^{m-n}{\mathbf{C}^2} \longrightarrow I_{\mathcal{H}} $$
$$ \beta: S^{m}{\mathbf{C}^2} \longrightarrow W $$
such that the tensor $ t $ defined in (\ref{eq:Schwtensor}) is sent to the tensor $ t_{\mathcal{H}} $ introduced in (\ref{eq:fundamentaltensor}).
Since $ H_{i} = \{ f_{i} = 0\} $ osculates $ \mathcal{C}_{n} $ for all $ i \in \{1, \ldots, m+2\} $, then there exists $ u_{i} \in ({\mathbf{C}^2})^{\ast} $ such that $ f_{i} = u_{i}^{n} $. By identifying $ S^{m}{\mathbf{C}^2} $ with $ H^{0}(\mathbf{P}^1, \Omega_{\mathbf{P}^1}^{1}([u_{1}] + \ldots + [u_{m+2}])) $, i.e. the space of forms with simple poles at $ [u_{1}], \ldots, [u_{m+2}] $, we get a well-defined map
$$ \beta: H^{0}(\mathbf{P}^1, \Omega_{\mathbf{P}^1}^{1}([u_{1}] + \ldots + [u_{m+2}])) \longrightarrow W $$
$$ \omega \longmapsto (res_{[u_{1}]}(\omega), \ldots, res_{[u_{m+2}]}(\omega)). $$ 
Now, let $ [v] \not= [u_{i}] $ and identify $ S^{m-n}{\mathbf{C}^2} $ with $ H^{0}(\mathbf{P}^1, \Omega_{\mathbf{P}^1}^{1}([u_{1}] + \ldots + [u_{m+2}] - n[v])) $, the space of forms with at most simple poles at $ [u_{i}] $ and a zero of order $ \leq n $ at $ [v] $. We define $ \alpha $ as the restriction of $ \beta $ to this space, which concludes the proof.
\end{proof}

\vfill\eject

\section{Torelli type theorems for the normal crossings case}

Let $ \mathcal{H} = \{H_{1}, \ldots, H_{\ell}\} $ be a hyperplane arrangement with normal crossings on $ \mathbf{P}^n $. If $ \mathcal{H} $ is made of \emph{few} hyperplanes then it is not a \emph{Torelli} arrangement. In this sense we have the following results:

\begin{theorem}\label{T:pochiiperpiani}$($\emph{Dolgachev-Kapranov 1993,~\cite{Do-Ka}}$)$\\
If $ 1 \leq \ell \leq n+1 $ then $ \Omega_{\mathbf{P}^n}^{1}(\log \mathcal{H}) \cong \mathcal{O}_{\mathbf{P}^n}^{\ell-1} \oplus \mathcal{O}_{\mathbf{P}^n}(-1)^{n+1-\ell}.  $
\end{theorem}

\begin{proof}
Let $ M = \displaystyle{\bigoplus_{i}} H^{0}(\mathbf{P}^n, \Omega_{\mathbf{P}^n}^{1}(\log \mathcal{H})(i)) $ be the graded $ \mathbf{C}[x_{0}, \ldots, x_{n}] $-module associated to $ \Omega_{\mathbf{P}^n}^{1}(\log \mathcal{H}) $ thanks to Serre's theorem,~\cite{Ha}. It comes out that $ M $ is the kernel of the homomorphism 
$$ \mathbf{C}[x_{0}, \ldots, x_{n}]^{\ell} \oplus \mathbf{C}[x_{0}, \ldots, x_{n}]^{n+1-\ell}(-1) \longrightarrow \mathbf{C}[x_{0}, \ldots, x_{n}] $$
$$ (g_{1}, \ldots, g_{n+1}) \longmapsto \displaystyle{\sum_{j=1}^{\ell} g_{j}} + \displaystyle{\sum_{j=\ell +1}^{n+1} g_{j}x_{j} }  $$
that is $ M \cong \mathbf{C}[x_{0}, \ldots, x_{n}]^{\ell-1} \oplus \mathbf{C}[x_{0}, \ldots, x_{n}](-1)^{n+1-\ell} $, as desired.
\end{proof}

\begin{proposition}\label{P:3.14}
If $ \ell = n+2 $ then $ \Omega_{\mathbf{P}^n}^{1}(\log \mathcal{H}) \cong \mathbf{TP}^{n}(-1). $
\end{proposition}

\begin{proof}
Theorem~\ref{T:3.8} implies that $ \Omega_{\mathbf{P}^n}^{1}(\log \mathcal{H}) \in S_{n,1} $, that is it admits an exact sequence of the form
$$ 0 \longrightarrow \mathcal{O}_{\mathbf{P}^n}(-1) \longrightarrow \mathcal{O}_{\mathbf{P}^n}^{n+1} \longrightarrow \Omega_{\mathbf{P}^n}^{1}(\log \mathcal{H}) \longrightarrow 0 $$
which is the Euler sequence for $ \mathbf{TP}^{n}(-1) $.
\end{proof}

\begin{remark}
All Steiner bundles in $ S_{n,1} $ are isomorphic to  $ \mathbf{TP}^{n}(-1) $.
\end{remark}

If we consider arrangements with a \emph{sufficiently large} number of hyperplanes, then the \emph{Torelli correspondence} defined in (\ref{eq:Torellimap})  is very closed to be an injective map. The main result on this topic is the following:

\begin{theorem}\label{T:Do-Ka}$($\emph{Vall\`es 2000,~\cite{V}}$)$\\
Let $ \mathcal{H} = \{H_{1}, \ldots, H_{\ell}\} $ and $ \mathcal{K} = \{K_{1}, \ldots, K_{\ell}\} $ be arrangements of $ \ell \geq n+3 $ hyperplanes with normal crossings on $ \mathbf{P}^n $ such that 
\begin{equation}\label{eq:iperlogiso}
\Omega_{\mathbf{P}^n}^{1}(\log \mathcal{H}) \cong \Omega_{\mathbf{P}^n}^{1}(\log \mathcal{K}).
\end{equation}
Then we have one of these possibilities:
\begin{itemize}
\item[$1)$] $ \mathcal{H} = \mathcal{K} $;
\item[$2)$] there exists a rational normal curve $ \mathcal{C}_{n} \subset \mathbf{P}^n $ such that $ H_{1}, \ldots, H_{\ell}, \\
K_{1}, \ldots, K_{\ell} $ osculate $ \mathcal{C}_{n} $ and $ \Omega_{\mathbf{P}^n}^{1}(\log \mathcal{H}) \cong \Omega_{\mathbf{P}^n}^{1}(\log \mathcal{K}) \cong E_{\ell - 2}(\mathcal{C}_{n}^{\vee}) $.
\end{itemize}
\end{theorem}

\begin{remark}
Since one can always find a rational normal curve $ \mathcal{C}_{n}^{\vee} \subset ({\mathbf{P}^n})^{\vee} $ of degree $ n $ passing through $ n+3 $ points of $ ({\mathbf{P}^n})^{\vee} $, the previous theorem becomes important when $ \ell \geq n+4 $.
\end{remark}

\begin{remark}
In 1993 Dolgachev and Kapranov in~\cite{Do-Ka} proved the same result of Vall\`es when $ \ell \geq 2n+3 $, focusing their attention on the set of \emph{jumping lines} of $ \Omega_{\mathbf{P}^n}^{1}(\log \mathcal{H}) $.
In the following we will see a sketch of the proof of theorem~\ref{T:Do-Ka}, which is based on the following idea: recover the hyperplanes of $ \mathcal{H} $ as \emph{unstable hyperplanes} of $ \Omega_{\mathbf{P}^n}^{1}(\log \mathcal{H}) $.
\end{remark}

\begin{definition}\label{d:unsthyp}$($Vall\`es 2000, $\cite{V})$ \\
Let $ S \in S_{n,k} $ be a Steiner bundle and let $ H \subset \mathbf{P}^n $ be a hyperplane. \\
We say that $ H $ is a \emph{unstable hyperplane} for $ S $ if the following condition holds:
$$ H^{0}(H, S^{\vee}_{|_{H}}) \not= \{0\}. $$
\end{definition}

\begin{remark}
The notion of unstable hyperplane for a Steiner bundle $ S \in S_{n,k} $ is justified from the fact that $ S^{\vee} $ has not global sections different from the zero one. Indeed, $ S^{\vee} $ is a stable bundle because of remark~\ref{r:3.7} and $ c_{1}(S^{\vee}) = -k < 0 $. 
\end{remark}

The \emph{Torelli theorem} for hyperplanes proved by Vall\`es is a consequence of the following result:

\begin{theorem}\label{T:V}$($\emph{Vall\`es 2000,~\cite{V}}$)$\\
Let $ \ell \geq n+2 $ and let $ S \in S_{n,\ell-n-1} $. If $ S $ has $ H_{1}, \ldots, H_{\ell+1} $ distinct unstable hyperplanes, then there exists a rational normal curve $ \mathcal{C}_{n}^{\vee} \subset (\mathbf{P}^n)^{\vee} $ such that $ H_{1}, \ldots, H_{\ell+1} $ osculate $ \mathcal{C}_{n} $ $($or, equivalently, $ H_{i} \in \mathcal{C}_{n}^{\vee} $ for all $ i \in \{1, \ldots, \ell+1\} $$)$ and $ S \cong E_{\ell - 2}(\mathcal{C}_{n}^{\vee}) $.
\end{theorem}

\begin{remark}
Theorem~\ref{T:V} asserts that a Steiner bundle in $ S_{n,\ell-n-1} $ which is not a Schwarzenberger bundle has at most $ \ell $ different unstable hyperplanes. The proof of this result is based on ``reductions" and the main steps are the following:
\begin{itemize}
\item[$1)$] the kernel $ T_{1} $ of the homomorphism $ S \longrightarrow \mathcal{O}_{H_{1}} $ induced by a non zero element of $ H^{0}(H_{1}, S^{\vee}_{|_{H_{1}}}) $ is a Steiner bundle in $ S_{n,\ell-n-2} $ and the set of unstable hyperplanes of $ S $ is contained in
$$ \{H \subset \mathbf{P}^n hyperplane \, | \, H^{0}(H, {T_{1}^{\vee}}_{|_{H}}) \not= \{0\} \} \cup H_{1}; $$ 
\item[$2)$] by iterating this method, after $ \ell -n -3 $ reductions we get a Steiner bundle $ T_{\ell -n -3} \in S_{n,2} $ which actually is isomorphic to the Schwarzenberger bundle $ E_{n+1}(\mathcal{C}_{n}^{\vee}) $, for certain rational normal curve $ \mathcal{C}_{n}^{\vee} \subset (\mathbf{P}_{n})^{\vee} $, as we can see in~\cite{Do-Ka};
\item[$3)$] since for all $ m > n $ the set of unstable hyperplanes of $ E_{m}(\mathcal{C}_{n}^{\vee}) $ coincides with $ \mathcal{C}_{n}^{\vee} $ (\cite{V}), then $ H_{\ell-n-2}, \ldots, H_{\ell+1} $ are $ n+4 $ points of $ \mathcal{C}_{n}^{\vee} $;
\item[$4)$] by changing $ H_{\ell-n-2} $ with $ H_{i} $ for all $ i \in \{1, \ldots, \ell-n-3\} $ we get that $ H_{1}, \ldots, H_{\ell-n-3} $ osculate $ \mathcal{C}_{n} $ too, which implies that $ S \cong E_{\ell - 2}(\mathcal{C}_{n}^{\vee}) $ (\cite{V}). 
\end{itemize}
\end{remark}

\begin{remark} Further interesting results about Steiner bundles and unstable hyperplanes have been proved by Ancona and Ottaviani in~\cite{AO}.
\end{remark}
\bigskip
Now we have all the tools to prove theorem~\ref{T:Do-Ka}.

\begin{proof}
If $ H \in \mathcal{H} $ then, by using the residue exact sequence (\ref{eq:resiperpiani}), it's not hard to see that $ H $ is unstable for $ \Omega_{\mathbf{P}^n}^{1}(\log \mathcal{H}) $. \\
Now, assume that $ \mathcal{H} \not= \mathcal{K} $, for example let say that $ H_{1} \not= K_{1} $; we want to prove that the statement $ 2) $ of theorem~\ref{T:Do-Ka} holds. From the isomorphism (\ref{eq:iperlogiso}) we get that also $ K_{1} $ is unstable for $ \Omega_{\mathbf{P}^n}^{1}(\log \mathcal{H}) $. This implies that $ \Omega_{\mathbf{P}^n}^{1}(\log \mathcal{H}) $ is a Steiner bundle in $ S_{n, \ell-n-1} $ with at least $ \ell + 1 $ different unstable hyperplanes. So theorem~\ref{T:V} tells us that there exists a rational normal curve $ \mathcal{C}_{n}^{\vee} $ in $ (\mathbf{P}_{n})^{\vee} $ containing all the hyperplanes in $ \mathcal{H} $ and $ \mathcal{K} $ and such that $ \Omega_{\mathbf{P}^n}^{1}(\log \mathcal{H}) \cong \Omega_{\mathbf{P}^n}^{1}(\log \mathcal{K}) \cong E_{\ell - 2}(\mathcal{C}_{n}^{\vee}) $, as desired. 
\end{proof}

\begin{remark}
The theorem proved above asserts that, if $ \ell \geq n+2 $, then the set of unstable hyperplanes of $ \Omega_{\mathbf{P}^n}^{1}(\log \mathcal{H}) $ is equal to $ \mathcal{H} = \{H_{1}, \ldots, H_{\ell}\} $, unless the hyperplanes in $ \mathcal{H} $ osculate a rational normal curve $ \mathcal{C}_{n} $ of degree $ n $ in $ \mathbf{P}^{n} $, in which case all the hyperplanes corresponding to the points of $ \mathcal{C}_{n}^{\vee} \subset (\mathbf{P}^{n})^{\vee} $ are unstable for $ \Omega_{\mathbf{P}^n}^{1}(\log \mathcal{H}) $. In the latter situation all the arrangements made of $ \ell $ hyperplanes with normal crossings that osculate $ \mathcal{C}_{n} $ yield logarithmic bundles in the same class of isomorphism.
\end{remark}

\vfill\eject

\section{Torelli type theorems for the general case}

Recently hyperplane arrangements have been investigated by removing the hypothesis of normal crossings. In particular in this section we refer to the papers of Dolgachev (\cite{Do}) and Faenzi-Matei-Vall\`es (\cite{FMV}). \\

As we can see in section $ 2.2 $, we can't introduce the sheaf of differential $ 1 $-forms on $ \mathbf{P}^{n} $, with logarithmic poles along a family of hypersurfaces that not necessarily has normal crossings by using definition~\ref{d:2.6}. So, let refer for example to~\cite{Schenck}. We have the following: 

\begin{definition}
Let $ \mathcal{D} = \{D_{1}, \ldots, D_{\ell}\} $ be an arrangement of smooth hypersurfaces on $ \mathbf{P}^{n} $ and let $ f = \displaystyle{\prod_{i=1}^{\ell} f_{i}} $ be a polynomial of degree $ q $ in $ x_{0}, \ldots, x_{n} $ defining $ \mathcal{D} $. Let $ \mathcal{T}(\log\mathcal{D}) $ the sheaf given as the kernel of the Gauss map, i.e.
$$ \mathcal{O}_{\mathbf{P}^{n}}^{n+1} \buildrel \rm (\partial_{0}{\it f}, \ldots, \partial_{{\it n}}{\it f}) \over \longrightarrow \mathcal{O}_{\mathbf{P}^{n}}(q-1). $$ 
We call \emph{sheaf of differential $ 1 $-forms on $ \mathbf{P}^{n} $ with logarithmic poles along} $ \mathcal{D} $
$$ \Omega_{\mathbf{P}^n}^{1}(\log \mathcal{D}) = \mathcal{T}(\log\mathcal{D})^{\vee}(-1). $$ 
\end{definition} 

\begin{remark}
From the previous definition we get that $ \Omega_{\mathbf{P}^n}^{1}(\log \mathcal{D}) $ is a \emph{reflexive} sheaf, that is $ \Omega_{\mathbf{P}^n}^{1}(\log \mathcal{D})^{\vee\vee} = \Omega_{\mathbf{P}^n}^{1}(\log \mathcal{D}) $. 
\end{remark}

\begin{remark}
If $ \mathcal{D} $ has normal crossings, this definition coincides with definition~\ref{d:2.6}.
\end{remark}

In this more general situation, Catanese-Hosten-Khetan-Sturmfels (\cite{CHKS}) and Dolgachev (\cite{Do}) studied a subsheaf of $ \Omega_{\mathbf{P}^n}^{1}(\log \mathcal{D}) $ instead of $ \Omega_{\mathbf{P}^n}^{1}(\log \mathcal{D}) $ itself. In this sense we have the following:

\begin{definition}
We denote by $ \widetilde{\Omega}_{\mathbf{P}^n}^{1}(\log \mathcal{D}) $ the rank $ n $ torsion free subsheaf of $ \Omega_{\mathbf{P}^n}^{1}(\log \mathcal{D}) $ that admits the short exact sequence
$$ 0 \longrightarrow \Omega_{\mathbf{P}^n} \longrightarrow \widetilde{\Omega}_{\mathbf{P}^n}^{1}(\log \mathcal{D}) \longrightarrow \displaystyle{\bigoplus_{i=1}^{\ell} \mathcal{O}_{D_{i}}} \longrightarrow 0 $$
which is called \emph{residue exact sequence}, just like the normal crossings case.
\end{definition}

\begin{remark}
If we don't assume that all the $ D_{i} $'s are smooth, then the residue exact sequence becomes
$$ 0 \longrightarrow \Omega_{\mathbf{P}^n} \longrightarrow \widetilde{\Omega}_{\mathbf{P}^n}^{1}(\log \mathcal{D}) \longrightarrow \nu_{\ast} \mathcal{O}_{\mathcal{D}'} \longrightarrow 0 $$
where $ \nu : \mathcal{D}' \longrightarrow \mathcal{D}$ is a resolution of singularities of $ \mathcal{D} $.
\end{remark}

\begin{remark}\label{r:omegatildeproperties}
A detailed description of $ \widetilde{\Omega}_{\mathbf{P}^n}^{1}(\log \mathcal{D}) $ is given in~\cite{Do}. \\
In particular we stress the following facts:
\begin{itemize}
\item[$1)$] $ \widetilde{\Omega}_{\mathbf{P}^n}^{1}(\log \mathcal{D})^{\vee\vee} \cong \Omega_{\mathbf{P}^n}^{1}(\log \mathcal{D}); $
\item[$2)$] if the codimension of the set where $ \mathcal{D} $ has not normal crossings is at least $ 3 $, then $ \widetilde{\Omega}_{\mathbf{P}^n}^{1}(\log \mathcal{D}) \cong \Omega_{\mathbf{P}^n}^{1}(\log \mathcal{D}) $ (in particular if $ \mathcal{D} $ has normal crossings, then $ \widetilde{\Omega}_{\mathbf{P}^n}^{1}(\log \mathcal{D}) $ is locally free);
\item[$3)$] if $ \mathcal{H} $ is a hyperplane arrangement such that $ \widetilde{\Omega}_{\mathbf{P}^n}^{1}(\log \mathcal{H}) $ is locally free, then $ \mathcal{H} $ has normal crossings;
\item[$4)$] if $ \mathcal{H} $ is an arrangement of $ \ell \geq n+2 $ hyperplanes, then  $ \widetilde{\Omega}_{\mathbf{P}^n}^{1}(\log \mathcal{D}) $ is a rank-$ n $ \emph{Steiner sheaf} over $ \mathbf{P}^n $ which appears in a short exact sequence like (\ref{eq:steinerfacile}) with $ k = \ell-n-1 $.
\end{itemize}
\end{remark}

In~\cite{Do} Dolgachev studied the \emph{Torelli problem} for $ \widetilde{\Omega}_{\mathbf{P}^n}^{1}(\log \mathcal{H}) $, where $ \mathcal{H} $ is a hyperplane arrangement on $ \mathbf{P}^{n} $. Statement 4) of the previous remark allowed Dolgachev to use Vall\`es' notion of \emph{unstable hyperplane} also for $ \widetilde{\Omega}_{\mathbf{P}^n}^{1}(\log \mathcal{H}) $. In order to state the conjecture that he formulated we recall the following:

\begin{definition}
Let $ E $ be a torsion-free coherent sheaf over $ \mathbf{P}^{n} $ and let $ \chi(E(k)) = \displaystyle{\sum_{i}(-1)^{i} \dim H^{i}(\mathbf{P}^{n}, E(k)) } $ the \emph{Euler characteristic} of $ E(k) $. We say that $ E $ is \emph{Gieseker-stable} (resp. \emph{Gieseker-semistable}) if for all coherent subsheaves $ F $ of $ E $ such that $ 0 < rk F < rk E $ we have that 
$$ \displaystyle{{{\chi(F(k))} \over {rk F}} < {{\chi(E(k))} \over {rk E}}} \quad\quad (resp. \,\,\displaystyle{{{\chi(F(k))} \over {rk F}} \leq {{\chi(E(k))} \over {rk E}}}) $$
for all integers $ k >>0 $.
\end{definition}

\begin{remark} 
Direct computations show that $ \displaystyle{{{\chi(E(k))} \over {rk E}} - {{\chi(F(k))} \over {rk F}}} $ is a polynomial in $ k $ with leading term given by a positive scalar multiple of $ \mu(E) -\mu(F) $ and so, for sufficiently large $ k \in \mathbf{Z} $, $ \displaystyle{{{\chi(E(k))} \over {rk E}} - {{\chi(F(k))} \over {rk F}}} $ and $ \mu(E) -\mu(F) $ have the same sign. We immediately get that the Gieseker-semistability implies the slope-semistability and the slope-stability implies the Gieseker-stability,~\cite{OSS}.
\end{remark}

\begin{definition}
Let $ \mathcal{C} $ be a connected curve of arithmetic genus $ 0 $. We say that $ \mathcal{C} $ is a \emph{stable normal rational curve of degree $ n $ in $ \mathbf{P}^n $} if $ \mathcal{C} = \displaystyle{\bigcup_{i=1}^{s} \mathcal{C}_{i}} $ where each $ \mathcal{C}_{i} $ is a smooth rational curve of degree $ d_{i} $ spanning a $ \mathbf{P}^{d_{i}} $, the degrees satisfy $ \displaystyle{\sum_{i=1}^{s} d_{i}} = n $ and $ \displaystyle{\bigcup_{i=1}^{s} \mathbf{P}^{d_{i}}} $ spans $ \mathbf{P}^{n} $.
\end{definition}

Now we are ready to state the following:

\begin{conjecture}$($\emph{Dolgachev 2007,~\cite{Do}}$)$ \\
Let $ \mathcal{H} = \{H_{1}, \ldots, H_{\ell}\} $ be an arrangement of $ \ell \geq n+2 $ hyperplanes on $ \mathbf{P}^{n} $ such that $ \widetilde{\Omega}_{\mathbf{P}^n}^{1}(\log \mathcal{H}) $ is \emph{Gieseker-semistable}. $ \mathcal{H} $ is a \emph{Torelli arrangement} if and only if $ H_{1}, \ldots, H_{\ell} $ don't osculate a \emph{stable normal rational curve of degree $ n $ in $ \mathbf{P}^n $}.
\end{conjecture}

We remark that Dolgachev proved the truth of this conjecture in the case of $ n = 2 $ and $ \ell \leq 6 $. \\

Faenzi, Matei and Vall\`es in~\cite{FMV} investigated the set of unstable hyperplanes of $ \widetilde{\Omega}_{\mathbf{P}^n}^{1}(\log \mathcal{H}) $ and proved Dolgachev's conjecture by changing \emph{stable normal rational curve of degree $ n $} with \emph{Kronecker-Weierstrass variety of type} $ (d;s) $. In this sense we have the following:

\begin{definition}
Let $ (d, n_{1}, \ldots, n_{s}) \in \mathbf{N}^{s+1} $ such that $ 1 \leq d \leq n $ and $ n = d + \displaystyle{\sum_{i=1}^{s} n_{i}} $. $ Y \subset (\mathbf{P}^n)^{\vee} $ is called a \emph{Kronecker-Weierstrass variety of type} $ (d;s) $ if $ Y = \mathcal{C} \cup L_{1} \cup \ldots \cup L_{s} $, where $ \mathcal{C} $ is a smooth rational curve of degree $ d $ that spans a linear space $ L $ of dimension $ d $ ($ \mathcal{C} $ is said to be the \emph{curve part} of $ Y $) and $ L_{i} $ is a linear subspace of dimension $ 1 \leq n_{i} \leq n-1 $ for all $ i \in \{1, \ldots, s \} $, with the following properties:
\begin{itemize}
\item[$1)$] $ L \cap L_{i} = \{p_{i}\} \in \mathcal{C} $ for all $ i $;
\item[$2)$] $ L_{i} \cap L_{j} = \emptyset $ for all $ i \not= j $.
\end{itemize}
In the case of $ d = 0 $, $ \mathcal{C} $ reduces to a single point $ \{p\} $, which is called the \emph{distinguished point} of $ Y $ and all the $ L_{i} $'s meet only at $ \{p\} $.
\end{definition}

\begin{remark}
The name given by Faenzi-Matei-Vall\`es to the varieties described above comes from the isomorphism classes of these varieties which are given by the Kronecker-Weierstrass form of a matrix of homogeneous linear forms in two variables.
\end{remark}
 
\begin{example}
If $ r_{1} $ and $ r_{2} $ are lines in $ (\mathbf{P}^2)^{\vee} $, then $ Y = r_{1} \cup r_{2} $ can be a Kronecker-Weierstrass variety of type $ (1;1) $ (in two ways, simply by interchanging the lines) or of type $ (0; 2) $ (in particular the distinguished point of $ Y $ is the point of intersection of $ r_{1} $ and $ r_{2} $). \\
If $ {\mathcal{C}}_{n}^{\vee} $ is a rational normal curve of degree $ n $ in $ (\mathbf{P}^n)^{\vee} $, then $ Y = {\mathcal{C}}_{n}^{\vee} $ is a Kronecker-Weierstrass variety of type $ (n;0) $. 
\end{example}

We have the following:

\begin{theorem}\label{T:FMV}$($\emph{Faenzi-Matei-Vall\`es 2010,~\cite{FMV}}$)$\\
Let $ \mathcal{H} = \{H_{1}, \ldots, H_{\ell}\} $ be an arrangement of hyperplanes in $ \mathbf{P}^n $ and let $ \mathcal{Z} = \{z_{1}, \ldots, z_{\ell}\} $ the corresponding set of points in $ (\mathbf{P}^n)^{\vee} $. Then $ \mathcal{H} $ is not a \emph{Torelli arrangement} if and only if $ \mathcal{Z} \subset Y $, where $ Y $ is a Kronecker-Weierstrass variety of type $ (d; s) $ in $ (\mathbf{P}^n)^{\vee} $. In particular, if $ d = 0 $, then the distinguished point of $ Y $ doesn't belong to $ \mathcal{Z} $.
\end{theorem}

\begin{remark} 
As in the normal crossing case, all the hyperplanes of $ \mathcal{H} $ are unstable for $ \widetilde{\Omega}_{\mathbf{P}^n}^{1}(\log \mathcal{H}) $. Thus, in order to get the previous theorem the authors proved that if $ H $ is a unstable hyperplane such that $ H \not= H_{i} $ for all $ i \in \{1, \ldots, \ell\} $, then there exists a Kronecker-Weierstrass variety $ Y \subset (\mathbf{P}^n)^{\vee} $ of type $ (d;s) $ containing $ \mathcal{Z} $. Y plays the role of the rational normal curve $ {\mathcal{C}}_{n}^{\vee} $ of degree $ n $ that appears in theorem~\ref{T:Do-Ka}. In particular, this theorem can be proved also with the arguments used in~\cite{FMV}. 
\end{remark}

\begin{remark}
As a direct consequence of theorem~\ref{T:FMV} we get the ``only if" implication of Dolgachev's conjecture. The ``if" implication holds only in $ \mathbf{P}^2 $, even if $ \widetilde{\Omega}_{\mathbf{P}^n}^{1}(\log \mathcal{H}) $ is not Gieseker-semistable, but in the case of $ n \geq 3 $ Faenzi-Matei-Vall\`es provide some counterexamples to it.
\end{remark}

\chapter{The higher degree case in the projective space}
\label{ch:the higher degree case in the projective space}

\section{Logarithmic bundles of hypersurfaces with normal crossings}

Let $ \mathcal{D} = \{D_{1}, \ldots, D_{\ell}\} $ be an arrangement of smooth hypersurfaces with normal crossings on $ \mathbf{P}^n $ and let $ \Omega_{\mathbf{P}^n}^{1}(\log \mathcal{D}) $ be the associated logarithmic bundle. We have the following:

\begin{proposition}
$ \Omega_{\mathbf{P}^n}^{1}(\log \mathcal{D}) $ admits the residue exact sequence
\begin{equation}\label{eq:resgeneral}
0 \longrightarrow \Omega_{\mathbf{P}^n}^{1} \longrightarrow \Omega_{\mathbf{P}^n}^{1}(\log \mathcal{D}) \buildrel \rm res \over \longrightarrow \displaystyle\bigoplus_{i=1}^\ell \mathcal{O}_{D_{i}}  \longrightarrow 0.
\end{equation} 
\end{proposition}

\begin{proof}
It is a direct consequence of statement $ 2) $ of remark~\ref{r:omegatildeproperties}.
\end{proof}

\begin{theorem}\label{T:Ancona}$($\emph{Ancona,~\cite{AppuntiAncona}}$)$\\
Let assume that $ D_{i} = \{f_{i} = 0\} $, where $ f_{i} $ is a homogeneous polynomial of degree $ d_{i} $ in the variables $ x_{0}, \ldots, x_{n} $, for all $ i \in \{1, \ldots, \ell\} $.\\
Then the logarithmic bundle $ \Omega_{\mathbf{P}^n}^{1}(\log \mathcal{D}) $ has the following resolution:
\begin{equation}\label{eq:Anconas}
0 \longrightarrow \Omega_{\mathbf{P}^n}^{1}(\log \mathcal{D})^{\vee} \longrightarrow \mathcal{O}_{\mathbf{P}^n}(1)^{n+1} \oplus \mathcal{O}_{\mathbf{P}^n}^{\ell-1} \buildrel \rm N \over \longrightarrow \displaystyle{\bigoplus_{i=1}^{\ell} \mathcal{O}_{\mathbf{P}^n}(d_{i})} \longrightarrow 0
\end{equation}
where $ N $ is the $ \ell \times (n+\ell) $ matrix 
\begin{equation}\label{eq:matrixAncona}
N = \pmatrix{ 
\partial_{0}f_{1} & \cdots & \partial_{n}f_{1} & f_{1} & 0 & \cdots & 0 \cr 
\partial_{0}f_{2} & \cdots & \partial_{n}f_{2} & 0 & f_{2} & {} & \vdots \cr
\vdots & {} & \vdots & \vdots & {} & \ddots & 0 \cr 
\partial_{0}f_{\ell-1} & \cdots & \partial_{n}f_{\ell-1} & 0 & \cdots & 0 & f_{\ell-1} \cr
\partial_{0}f_{\ell} & \cdots & \partial_{n}f_{\ell} & 0 & \cdots & \cdots & 0 \cr
}. 
\end{equation}
\end{theorem}

\begin{proof}
As in \cite{Do}, let denote by $ S $ the polynomial algebra $ \mathbf{C}[x_{0}, \ldots, x_{n}] $ and let 
$$ \Omega^{1}_{S} = <dx_{0}, \ldots, dx_{n}>_{S} \cong S(-1)^{n+1} $$
$$ Der_{S} = <{{\partial} \over {\partial x_{0}}}, \ldots, {{\partial} \over {\partial x_{n}}}>_{S} \cong S(1)^{n+1} $$
be, respectively, the graded $ S $-module of differentials and the graded $ S $-module of derivations. The Euler derivation $ \xi = \displaystyle{\sum_{i=0}^{n} x_{i} {{\partial} \over {\partial x_{i}}}} $ defines a homomorphism of graded $ S $-modules 
$$ \Omega^{1}_{S} \longrightarrow S \quad\,\,\,\, \,\, \omega \longmapsto \omega(\xi) $$
whose kernel corresponds to the sheaf $ \Omega^{1}_{\mathbf{P}^{n}} $. Moreover the cokernel of the homomorphism
$$ S \longrightarrow Der_{S} \quad\,\,\,\, p \longmapsto p \xi $$
corresponds to $ \mathbf{TP}^{n} $, which is the dual sheaf of $ \Omega^{1}_{\mathbf{P}^{n}} $.
So we have a pairing
$$ \Omega^{1}_{\mathbf{P}^{n}} \times \mathbf{TP}^{n} \buildrel \rm < \cdot, \cdot > \over \longrightarrow \mathcal{O}_{\mathbf{P}^{n}} $$
$$ \left(\displaystyle{\sum_{i=0}^{n} h_{i}dx_{i}}, \displaystyle{\sum_{i=0}^{n} b_{i}{{\partial} \over {\partial x_{i}}}}\right) \longmapsto \displaystyle{\sum_{i=0}^{n} h_{i}b_{i}} $$
and, if $ U $ is an open subset of $ \mathbf{P}^{n} $, then $ \Gamma(U, \Omega_{\mathbf{P}^n}^{1}(\log \mathcal{D})^{\vee}) $ is given by
$$  \{ v \in \Gamma(U, \mathbf{TP}^{n}) \, | \, \forall\,local\,equation\, g_{i}\,of\,D_{i}\,in\,U <d\log g_{i},v> \, is \, holomorphic \} $$
where we recall that $ <d\log g_{i},v> = <\displaystyle{{d g_{i}} \over {g_{i}}}, v> $.  \\
Assume that $ x_{0} \not = 0 $ and let $ z_{j} = \displaystyle{{x_{j}} \over {x_{0}}} $ for all $ j \in \{1, \ldots, n\} $.
Since for all $ i \in \{1, \ldots, \ell\} $ we have that $ f_{i}(x_{0}, \ldots, x_{n}) = x_{0}^{d_{i}}f_{i}(1, \displaystyle{{x_{1}} \over {x_{0}}}, \ldots, \displaystyle{{x_{n}} \over {x_{0}}}) = x_{0}^{d_{i}} g_{i}(z_{1}, \ldots, z_{n}) $, the chain rule implies that, for all $ j \in \{1, \ldots, n\} $,
$$ \displaystyle{{\partial g_{i}} \over {\partial z_{j}}} = \displaystyle{{1} \over {x_{0}^{d_{i}-1}}} {{\partial f_{i}} \over {\partial x_{j}}}. $$
So we get that
$$ dg_{i} = \displaystyle{\sum_{j=1}^{n} {{\partial g_{i}} \over {\partial z_{j}}} dz_{j}} = \displaystyle{\sum_{j=1}^{n}} {{1} \over {x_{0}^{d_{i}-1}}} {{\partial f_{i}} \over {\partial x_{j}}} \left( \displaystyle{{dx_{j}} \over {x_{0}}} - \displaystyle{{x_{j}} \over {x_{0}^{2}}} dx_{0} \right) =   $$
$$ = \displaystyle{{1} \over {x_{0}^{d_{i}}}} \displaystyle{\sum_{j=1}^{n} {{\partial f_{i}} \over {\partial x_{j}}} dx_{j} } - \displaystyle{{dx_{0}} \over {x_{0}^{d_{i}+1}}} \displaystyle{\sum_{j=1}^{n}} {{\partial f_{i}} \over {\partial x_{j}}} x_{j} = \displaystyle{{1} \over {x_{0}^{d_{i}}}} \displaystyle{\sum_{j=1}^{n} {{\partial f_{i}} \over {\partial x_{j}}} dx_{j} } - \displaystyle{{dx_{0}} \over {x_{0}^{d_{i}+1}}} (d_{i})(f_{i}). $$
Thus implies that $ v = \displaystyle{\sum_{j=0}^{n} b_{j}{{\partial} \over {\partial x_{j}}}} \in \Gamma(U, \Omega_{\mathbf{P}^n}^{1}(\log \mathcal{D})^{\vee}) $ if and only if for all $ i \in \{1, \ldots, \ell\} $ there exists a holomorphic function $ \alpha_{i} $ such that 
$$ \displaystyle{\sum_{j=1}^{n} {{\partial f_{i}} \over {\partial x_{j}}} b_{j}} = \alpha_{i} f_{i} \quad\quad modulo \,\, \xi = 0. $$
$ \Omega_{\mathbf{P}^n}^{1}(\log \mathcal{D})^{\vee} $ turns to be the cohomology of the monad given by
$$ 0 \longrightarrow \mathcal{O}_{\mathbf{P}^{n}} \buildrel \rm \mathcal{M} \over \longrightarrow \mathcal{O}_{\mathbf{P}^{n}}(1)^{n+1} \oplus \mathcal{O}_{\mathbf{P}^{n}}^{\ell} \buildrel \rm \mathcal{N} \over \longrightarrow \displaystyle{\bigoplus_{i=1}^{\ell}\mathcal{O}_{\mathbf{P}^{n}}(d_{i})} \longrightarrow 0 $$
where $ \mathcal{M} $ is the $ (n+1+\ell) \times 1 $ matrix
$$ \mathcal{M} = \,\displaystyle{^{t}}\pmatrix{ x_{0} & \cdots & x_{n} & d_{1} & \cdots & d_{\ell} \cr} $$
and $ \mathcal{N} $ is the $ \ell \times (n+1+\ell) $ matrix
$$ \mathcal{N} = \pmatrix{ 
\partial_{0}f_{1} & \cdots & \partial_{n}f_{1} & f_{1} & 0 & \cdots & \cdots & 0 \cr 
\partial_{0}f_{2} & \cdots & \partial_{n}f_{2} & 0 & f_{2} & 0 & \cdots & 0 \cr
\vdots & {} & \vdots & \vdots & {} & \ddots & {} & \vdots \cr 
\vdots & {} & \vdots & \vdots & {} & {} &  \ddots & \vdots \cr
\partial_{0}f_{\ell} & \cdots & \partial_{n}f_{\ell} & 0 & {} & \cdots & {} & f_{\ell} \cr
}. $$
We remark that if we multiply $ \mathcal{N} $ with the square matrix of order $ n+\ell+1 $ 
$$ \pmatrix{ 
{} & {} & {} & {} & {} & x_{0} \cr
{} & {} & {} & {} & {} & \vdots \cr
{} & {} & {} & {}  & {} & x_{n} \cr
{} & {} & I_{n+\ell} & {} & {} & -d_{1} \cr
{} & {} & {} & {} & {} & \vdots \cr
{} & {} & {} & {} & {} & -d_{\ell - 1} \cr
0 & {} & \cdots & {} & 0 & -d_{\ell} \cr
} $$
and we apply the Euler formula, then we can remove the last column of $ \mathcal{N} $ so that $ \mathcal{N} $ takes the form of $ N $, the matrix in (\ref{eq:matrixAncona}). In particular $ \Omega_{\mathbf{P}^n}^{1}(\log \mathcal{D})^{\vee} $ admits the short exact sequence (\ref{eq:Anconas}), which concludes the proof.   
\end{proof}

\begin{remark}
Theorem~\ref{T:Ancona} holds in particular when $ \mathcal{D} $ is a hyperplane arrangement, that is when $ d_{i} = 1 $ for all $ i $. Indeed, if $ \ell \geq n+2 $ then (\ref{eq:Anconas}) becomes the dualized sequence of the Steiner sequence (\ref{eq:steinerfacile}) and if $ \ell \leq n+1 $ then (\ref{eq:Anconas}) implies theorem~\ref{T:pochiiperpiani}.   
\end{remark}

\vfill\eject

\section{A Torelli type result for one hypersurface}
 
According to chapter $ 3 $, the \emph{Torelli problem} for hyperplane arrangements has been completely solved, but in the higher degree case it still represents an open question. \\

In~\cite{UY1} Ueda and Yoshinaga studied this problem for one smooth cubic $ D $ in $ \mathbf{P}^2 $, focusing their attention on the set of jumping lines of the corresponding logarithmic bundle. In this sense they proved the following:

\begin{theorem}\label{T:UY1}$($\emph{Ueda-Yoshinaga 2008,~\cite{UY1}}$)$ \\
Let $ D $ and $ D' $ be smooth cubics in $ \mathbf{P}^2 $ with non-vanishing $ j $-invariant. Then the \emph{Torelli map} in (\ref{eq:Torellimap}) is injective.
\end{theorem}

Afterwards, in~\cite{UY2}, Ueda and Yoshinaga extended theorem~\ref{T:UY1} for the case of one smooth hypersurface in $ \mathbf{P}^n $. In order to state this result we introduce the following:

\begin{definition}\label{d:sebthom}
Let $ D \subset \mathbf{P}^n $ be a smooth hypersurface of degree $ d $ such that $ D = \{f = 0\} $. We call $ f $ of \emph{Sebastiani-Thom type} if we can choose homogeneous coordinates $ x_{0}, \ldots, x_{n} $ of $ \mathbf{P}^n $ and $ k \in \{0, \ldots, n-1\} $ such that
$$ f(x_{0}, \ldots, x_{n}) = f(x_{0}, \ldots, x_{k}) + f(x_{k+1}, \ldots, x_{n}). $$
\end{definition}

With the notations of definition~\ref{d:sebthom} we have the following:

\begin{theorem}\label{T:UY2}$($\emph{Ueda-Yoshinaga 2009,~\cite{UY2}}$)$ \\
$ \mathcal{D} = \{D\} $ is a \emph{Torelli arrangement} if and only if $ f $ is not of \emph{Sebastiani-Thom type}. 
\end{theorem}

\begin{remark}
If $ d = 2 $ then $ f $ is always of Sebastiani-Thom type, for all $ n $.
\end{remark}

\begin{remark}
A smooth plane cubic has a vanishing $ j $-invariant if and only if it is the zero locus of the \emph{Fermat polynomial} $ x_{0}^3+x_{1}^{3}+x_{2}^{3} $ which is equivalent to say that it is defined by a $ f $ of Sebastiani-Thom type.
\end{remark}

\begin{corollary}
Let $ D $ be a general hypersurface of degree $ d $ in $ \mathbf{P}^n $. Then $ \mathcal{D} = \{D\} $ is \emph{Torelli} if and only if $ d \geq 3 $.
\end{corollary}

\chapter{Arrangements of conics in the projective plane}
\label{ch:arrangements of conics in the projective plane}

\section{Many conics}
Arrangements made of a \emph{sufficiently large} number of conics with normal crossings can be studied by using the main results concerning arrangements of hyperplanes with normal crossings (\cite{Do-Ka},\cite{V}) that are recalled in chapter $ 3 $. \\

Let $ \mathcal{D} = \{C_{1}, \ldots, C_{\ell}\} $ be an arrangement of $ \ell $ smooth conics with normal crossings on $ \mathbf{P}^2 $ and let $ \Omega_{\mathbf{P}^2}^{1}(\log \mathcal{D}) $ be the corresponding logarithmic bundle.

\begin{remark}
Let $ \nu_{2} : \mathbf{P}^2 \longrightarrow \mathbf{P}^5 $ be the quadratic Veronese map, that is
$$ \nu_{2}([x_{0}, x_{1}, x_{2}]) = [x_{0}^2, x_{1}^2, x_{2}^2, x_{0}x_{1}, x_{0}x_{2}, x_{1}x_{2}] $$
and let $ V_{2} = \nu_{2}(\mathbf{P}^2) $ be its image. As we can see also in~\cite{H}, conics are exactly hyperplane sections of $ V_{2} \subset \mathbf{P}^5 $. 
\end{remark}

So we can associate to $ \mathcal{D} = \{C_{1}, \ldots, C_{\ell}\} $ an arrangement of hyperplanes $ \mathcal{H} = \{H_{1}, \ldots, H_{\ell}\} $ on $ \mathbf{P}^5 $. Let assume that $ \mathcal{H} $ has normal crossings and let $ \Omega_{\mathbf{P}^5}^{1}(\log \mathcal{H}) $ be the logarithmic bundle attached to it. We will see in the proof of theorem $ 5.4 $ (focus on exact sequence in (\ref{eq:prop211Dolg})) that the vector bundles $ \Omega_{\mathbf{P}^2}^{1}(\log \mathcal{D}) $ and $ \Omega_{\mathbf{P}^5}^{1}(\log \mathcal{H}) $ are strictly related one to the other. 

Given the logarithmic bundle $ \Omega_{\mathbf{P}^2}^{1}(\log \mathcal{D}) $, the key idea is to reconstruct the conics in $ \mathcal{D} $ as \emph{unstable} conics of $ \Omega_{\mathbf{P}^2}^{1}(\log \mathcal{D}) $, using the fact that we are able to deal with hyperplanes. The notion of unstable conic that we introduce in the following is very close to the one of unstable hyperplane (see definition~\ref{d:unsthyp}).

\begin{definition}\label{d52}
Let $ C \subset \mathbf{P}^2 $ be a conic. We say that $ C $ is \emph{unstable} for $ \Omega_{\mathbf{P}^2}^{1}(\log \mathcal{D}) $ if the following condition holds:
\begin{equation}\label{eq:conicainstabile}
H^{0}(C, {\Omega_{\mathbf{P}^2}^{1}(\log \mathcal{D})}^{\vee}_{|_{C}}) \not= \{0\}.
\end{equation}
\end{definition}

\begin{remark}\label{r53}
The previous definition is meaningful. Indeed, we recall that $ \Omega_{\mathbf{P}^2}^{1}(\log \mathcal{D}) $ admits the short exact sequence:
\begin{equation}\label{eq:Anconapermolteconiche}
0 \longrightarrow \mathcal{O}_{\mathbf{P}^2}(-2)^{\ell} \longrightarrow \mathcal{O}_{\mathbf{P}^2}(-1)^{3} \oplus \mathcal{O}_{\mathbf{P}^2}^{\ell -1} \longrightarrow \Omega_{\mathbf{P}^2}^{1}(\log \mathcal{D})  \longrightarrow 0.
\end{equation}
In particular the slope $ \mu(\Omega_{\mathbf{P}^2}^{1}(\log \mathcal{D})) = \displaystyle{{2 \ell - 3} \over {2}} > 0 $ for all $ \ell \geq 2 $, which implies, by using Bohnhorst-Spindler criterion, that $ \Omega_{\mathbf{P}^2}^{1}(\log \mathcal{D}) $ is stable. So $ \Omega_{\mathbf{P}^2}^{1}(\log \mathcal{D})^{\vee} $ is a stable bundle too (\cite{OSS}). We claim that $ \Omega_{\mathbf{P}^2}^{1}(\log \mathcal{D})^{\vee} $ has not global sections over $ \mathbf{P}^2 $ different from the zero one: if this is not the case there is a non trivial injective map $ \mathcal{O}_{\mathbf{P}^2} \hookrightarrow \Omega_{\mathbf{P}^2}^{1}(\log \mathcal{D})^{\vee} $, $ \mathcal{O}_{\mathbf{P}^2} $ can be regarded as a subsheaf of $ \Omega_{\mathbf{P}^2}^{1}(\log \mathcal{D})^{\vee} $ with $ 0 < rk\mathcal{O}_{\mathbf{P}^2} < rk\Omega_{\mathbf{P}^2}^{1}(\log \mathcal{D})^{\vee} $ and $ \mu(\mathcal{O}_{\mathbf{P}^2}) = 0 > \mu(\Omega_{\mathbf{P}^2}^{1}(\log \mathcal{D})^{\vee}) = \displaystyle{{3 - 2 \ell} \over {2}} $, which contradicts the stability of $ \Omega_{\mathbf{P}^2}^{1}(\log \mathcal{D})^{\vee} $.
\end{remark}

Now we can state and prove the main result concerning the \emph{Torelli problem} in the case of arrangements with a \emph{large} number of conics.

\begin{theorem}\label{t54}
Let $ \mathcal{D} = \{C_{1}, \ldots, C_{\ell}\} $ be an arrangement of smooth conics with normal crossings on $ \mathbf{P}^2 $ and let $ \mathcal{H} = \{H_{1}, \ldots, H_{\ell}\} $ be the corresponding arrangement of hyperplanes on $ \mathbf{P}^5 $ in the sense of remark $ 5.1 $. Assume that:
\begin{itemize}
\item[$1)$] $ \ell \geq 9 $;
\item[$2)$] $ \mathcal{H} = \{H_{1}, \ldots, H_{\ell}\} $ is an arrangement with normal crossings;
\item[$3)$] $ H_{1}, \ldots, H_{\ell} $ don't osculate a rational normal curve of degree $ 5 $ in $ \mathbf{P}^5 $.
\end{itemize}
Then 
$$ \mathcal{D} = \{ C \subset \mathbf{P}^2 \, \emph{smooth conic} \, | \, C \, \emph{satisfies} ~(\ref{eq:conicainstabile})\}. $$
\end{theorem}

\begin{proof}
Let suppose that $ C \in \mathcal{D} $, then there exists $ j \in \{1, \ldots, \ell\} $ such that $ C = C_{j} $. If we consider the \emph{residue exact sequence} for $ \Omega_{\mathbf{P}^2}^{1}(\log \mathcal{D}) $, that is
$$ 0 \longrightarrow \Omega_{\mathbf{P}^2}^{1} \longrightarrow \Omega_{\mathbf{P}^2}^{1}(\log \mathcal{D}) \buildrel \rm res \over\longrightarrow \displaystyle\bigoplus_{i=1}^\ell \mathcal{O}_{C_{i}}  \longrightarrow 0 $$
and we restrict it to $ C_{j} $, we get the following exact sequence:
$$ 0 \rightarrow \mathcal{T}or_{1}^{\mathbf{P}^2}(\mathcal{O}_{C_{j}}, \mathcal{O}_{C_{j}}) \rightarrow \Omega_{\mathbf{P}^2|_{C_{j}}}^{1} \rightarrow \Omega_{\mathbf{P}^2}^{1}(\log \mathcal{D})_{|_{C_{j}}} \rightarrow \mathcal{O}_{C_{j}} \oplus \displaystyle\bigoplus_{i=1, i \not= j}^\ell \mathcal{O}_{C_{i} \cap C_{j}} \rightarrow 0. $$
Since the map 
$$ \Omega_{\mathbf{P}^2}^{1}(\log \mathcal{D})_{|_{C_{j}}} \longrightarrow \mathcal{O}_{C_{j}} \oplus \displaystyle\bigoplus_{i=1, i \not= j}^\ell \mathcal{O}_{C_{i} \cap C_{j}} $$
is surjective, we get a non zero map 
$$ \Omega_{\mathbf{P}^2}^{1}(\log \mathcal{D})_{|_{C_{j}}} \longrightarrow \mathcal{O}_{C_{j}} $$
and so 
$$ H^{0}(C_{j}, {\Omega_{\mathbf{P}^2}^{1}(\log \mathcal{D})}^{\vee}_{|_{C_{j}}}) = Hom(\mathcal{O}_{C_{j}}, {\Omega_{\mathbf{P}^2}^{1}(\log \mathcal{D})}^{\vee}_{|_{C_{j}}}) \not= \{0\} $$
that is $ C $ satisfies (\ref{eq:conicainstabile}). \\
Viceversa, let assume that $ C $ is a smooth unstable conic for $ \Omega_{\mathbf{P}^2}^{1}(\log \mathcal{D}) $, we want to prove that $ C \in \mathcal{D} $. It suffices to show that the hyperplane $ H \subset \mathbf{P}^5 $ associated to $ C $ by means of $ \nu_{2} $ is unstable for $ \Omega_{\mathbf{P}^5}^{1}(\log \mathcal{H}) $: namely, if this is the case, since hypothesis $ 1),\, 2), \, 3) $ hold, the \emph{Torelli type} result of Vall\`es assures us that $ H \in \mathcal{H} $, that is $ H = H_{i} $ for $ i \in \{1, \ldots, \ell\} $ and so $ C = C_{i} \in \mathcal{D} $. \\
Since $ V_{2} $ is a non singular subvariety of $ \mathbf{P}^5 $ which intersects transversally $ \mathcal{H} $, from proposition $ 2.11 $ of~\cite{Do} we get the following exact sequence:
\begin{equation}\label{eq:prop211Dolg}
0 \longrightarrow  \mathcal{N}_{V_{2}, \, \mathbf{P}^5}^{\vee} \longrightarrow \Omega_{\mathbf{P}^5}^{1}(\log \mathcal{H})_{|_{V_{2}} } \longrightarrow \Omega_{V_{2}}^{1}(\log \mathcal{H} \cap V_{2}) \longrightarrow 0 
\end{equation}
where $ \mathcal{N}_{V_{2}, \, \mathbf{P}^5}^{\vee} $ denotes the conormal sheaf of $ V_{2} $ in $ \mathbf{P}^5 $. \\
We remark that $ V_{2} \cong \mathbf{P}^2 $ and $ \mathcal{D} = \mathcal{H} \cap V_{2} $, so (\ref{eq:prop211Dolg}) becomes
\begin{equation}\label{eq:prop211Dolg2}
0 \longrightarrow  \mathcal{N}_{\mathbf{P}^2, \, \mathbf{P}^5}^{\vee} \longrightarrow \Omega_{\mathbf{P}^5}^{1}(\log \mathcal{H})_{|_{\mathbf{P}^2}} \longrightarrow \Omega_{\mathbf{P}^2}^{1}(\log \mathcal{D}) \longrightarrow 0. 
\end{equation}
Restricting (\ref{eq:prop211Dolg2}) to $ C $ and then applying $ \mathcal{H}om(\cdot, \, \mathcal{O}_{C}) $ we obtain the following short exact sequence:
\begin{equation}\label{eq:prop211Dolg3}
0 \longrightarrow {\Omega_{\mathbf{P}^2}^{1}(\log \mathcal{D})}^{\vee}_{|_{C}} \longrightarrow {\Omega_{\mathbf{P}^5}^{1}(\log \mathcal{H})}^{\vee}_{|_{C}} \longrightarrow ({\mathcal{N}_{\mathbf{P}^2, \, \mathbf{P}^5\,_{|_{C}}}^{\vee}})^{\vee} \longrightarrow 0. 
\end{equation}
Finally we apply $ \Gamma(C,\, \cdot) $ to (\ref{eq:prop211Dolg3}) and we get:
$$ 0 \longrightarrow H^{0}(C, {\Omega_{\mathbf{P}^2}^{1}(\log \mathcal{D})}^{\vee}_{|_{C}}) \longrightarrow H^{0}(C, {\Omega_{\mathbf{P}^5}^{1}(\log \mathcal{H})}^{\vee}_{|_{C}}). $$
By assumption, $ C $ is unstable for $ \Omega_{\mathbf{P}^2}^{1}(\log \mathcal{D}) $, that is condition (\ref{eq:conicainstabile}) holds. Necessarily it has to be 
\begin{equation}\label{eq:keyfact1}
H^{0}(C, {\Omega_{\mathbf{P}^5}^{1}(\log \mathcal{H})}^{\vee}_{|_{C}}) \not= 0.
\end{equation}
Now, let $ \mathcal{I}_{V_{2}, \, \mathbf{P}^5} $ be the ideal sheaf of $ V_{2} $ in $ \mathbf{P}^5 $; we have this exact sequence:
\begin{equation}\label{eq:idealsheafsequence}
0 \longrightarrow \mathcal{I}_{V_{2}, \, \mathbf{P}^5} \longrightarrow \mathcal{O}_{\mathbf{P}^5} \longrightarrow \mathcal{O}_{V_{2}} \longrightarrow 0.
\end{equation}
Since $ V_{2} \not\subset H $ we have
\begin{equation}\label{eq:idealsheafsequencebis}
0 \longrightarrow \mathcal{I}_{V_{2} \cap H, \, H} \longrightarrow \mathcal{O}_{H} \longrightarrow \mathcal{O}_{V_{2} \cap H} \longrightarrow 0.
\end{equation}
By tensor product with $ {\Omega_{\mathbf{P}^5}^{1}(\log \mathcal{H})}^{\vee} $, (\ref{eq:idealsheafsequencebis}) becomes:
\begin{equation}\label{eq:idealsheafsequence1}
0 \longrightarrow \mathcal{I}_{V_{2} \cap H, \, H} \otimes {\Omega_{\mathbf{P}^5}^{1}(\log \mathcal{H})}^{\vee}_{|_{H}} \longrightarrow {\Omega_{\mathbf{P}^5}^{1}(\log \mathcal{H})}^{\vee}_{|_{H}} \longrightarrow {\Omega_{\mathbf{P}^5}^{1}(\log \mathcal{H})}^{\vee}_{|_{C}} \longrightarrow 0.
\end{equation}
(\ref{eq:idealsheafsequence1}) induces the following long exact sequence in cohomology:
$$ 0 \longrightarrow H^{0}(H, \mathcal{I}_{V_{2} \cap H, \, H} \otimes {\Omega_{\mathbf{P}^5}^{1}(\log \mathcal{H})}^{\vee}_{|_{H}}) \longrightarrow H^{0}(H, {\Omega_{\mathbf{P}^5}^{1}(\log \mathcal{H})}^{\vee}_{|_{H}}) \longrightarrow $$
$$ \,\,\,\, \longrightarrow H^{0}(C, {\Omega_{\mathbf{P}^5}^{1}(\log \mathcal{H})}^{\vee}_{|_{C}}) \longrightarrow H^{1}(H, \mathcal{I}_{V_{2} \cap H, \, H} \otimes {\Omega_{\mathbf{P}^5}^{1}(\log \mathcal{H})}^{\vee}_{|_{H}}). \quad\quad  $$
To conclude the proof it suffices to show that 
\begin{equation}\label{eq:keyfact2}
H^{1}(H, \mathcal{I}_{V_{2} \cap H, \, H} \otimes {\Omega_{\mathbf{P}^5}^{1}(\log \mathcal{H})}^{\vee}_{|_{H}}) = \{0\}.
\end{equation}
Indeed, if (\ref{eq:keyfact2}) holds, then the map 
$$ H^{0}(H, {\Omega_{\mathbf{P}^5}^{1}(\log \mathcal{H})}^{\vee}_{|_{H}}) \longrightarrow H^{0}(C, {\Omega_{\mathbf{P}^5}^{1}(\log \mathcal{H})}^{\vee}_{|_{C}}) $$
is surjective and so, because of (\ref{eq:keyfact1}), we get 
$$ H^{0}(H, {\Omega_{\mathbf{P}^5}^{1}(\log \mathcal{H})}^{\vee}_{|_{H}}) \not= \{0\}, $$ 
that is $ H $ is unstable for $ \Omega_{\mathbf{P}^5}^{1}(\log \mathcal{H}) $.
In order to prove (\ref{eq:keyfact2}), we remark that, since $ \ell \geq 9 $, $ \Omega_{\mathbf{P}^5}^{1}(\log \mathcal{H}) $ is a \emph{Steiner} bundle over $ \mathbf{P}^5 $, i.e.
$$ 0 \longrightarrow \mathcal{O}_{\mathbf{P}^5}(-1)^{\ell-6} \longrightarrow \mathcal{O}_{\mathbf{P}^5}^{\ell-1} \longrightarrow \Omega_{\mathbf{P}^5}^{1}(\log \mathcal{H}) \longrightarrow 0 $$
is exact. Since in the previous sequence all the terms are vector bundles, applying $ \mathcal{H}om(\cdot, \, \mathcal{O}_{\mathbf{P}^5}) $ we get
$$ 0 \longrightarrow {\Omega_{\mathbf{P}^5}^{1}(\log \mathcal{H})}^{\vee} \longrightarrow \mathcal{O}_{\mathbf{P}^5}^{\ell-1} \longrightarrow \mathcal{O}_{\mathbf{P}^5}(1)^{\ell-6} \longrightarrow 0, $$
which, via tensor product with $ \mathcal{I}_{V_{2},\,\mathbf{P}^5\,|_{H}} $, becomes
\begin{equation}\label{eq:Steiner1}
0 \longrightarrow \mathcal{I}_{V_{2} \cap H, \, H} \otimes \, {\Omega_{\mathbf{P}^5}^{1}(\log \mathcal{H})}^{\vee}_{|_{H}} \longrightarrow \mathcal{I}_{V_{2} \cap H, \, H} \otimes \, \mathcal{O}_{\mathbf{P}^5 \, |_{H}}^{\ell-1} \longrightarrow
\end{equation}
$$ \longrightarrow \mathcal{I}_{V_{2} \cap H, \, H} \otimes \, \mathcal{O}_{\mathbf{P}^5}(1)^{\ell-6}_{|_{H}} \longrightarrow 0. \quad\quad\quad\quad\quad\quad\quad\quad\quad\quad\quad\quad $$ 
Applying $ \Gamma(H,\, \cdot) $ to (\ref{eq:Steiner1}) we obtain the following long exact sequence:
\begin{equation}\label{eq:coomolSteiner1}
0 \longrightarrow H^{0}(H, \mathcal{I}_{V_{2} \cap H, \, H} \otimes \, {\Omega_{\mathbf{P}^5}^{1}(\log \mathcal{H})}^{\vee}_{|_{H}}) \longrightarrow \quad\quad\quad\quad\quad\quad\quad\quad\quad\quad
\end{equation} 
$$ \longrightarrow H^{0}(H, \mathcal{I}_{V_{2} \cap H, \, H} \otimes \, \mathcal{O}_{\mathbf{P}^5 \, |_{H}}^{\ell-1}) \longrightarrow H^{0}(H, \mathcal{I}_{V_{2} \cap H, \, H} \otimes \, \mathcal{O}_{\mathbf{P}^5}(1)^{\ell-6}_{|_{H}}) \longrightarrow  $$
$$ \quad\, \longrightarrow H^{1}(H, \mathcal{I}_{V_{2} \cap H, \, H} \otimes \, {\Omega_{\mathbf{P}^5}^{1}(\log \mathcal{H})}^{\vee}_{|_{H}}) \longrightarrow H^{1}(H, \mathcal{I}_{V_{2} \cap H, \, H} \otimes \, \mathcal{O}_{\mathbf{P}^5 \, |_{H}}^{\ell-1}). \quad \,$$
We remark that
$$ H^{i}(H, \mathcal{I}_{V_{2} \cap H, \, H} \otimes \, \mathcal{O}_{\mathbf{P}^5}(t)^{s}_{|_{H}}) = H^{i}(H, \mathcal{I}_{V_{2} \cap H, \, H}(t))^{s} $$
for all $ i, s, t $ integers such that $ i, s \geq 0 $. We note also that, if $ t \geq 0 $, then $ H^{0}(H, \mathcal{I}_{V_{2} \cap H, \, H}(t)) $ is the set of all homogeneous forms of degree $ t $ over $ H $ which vanish at $ V_{2} \cap H $. So the second and the third term of (\ref{eq:coomolSteiner1}) are trivial. This implies also that
$$ H^{0}(H, \mathcal{I}_{V_{2} \cap H, \, H} \otimes \, {\Omega_{\mathbf{P}^5}^{1}(\log \mathcal{H})}^{\vee}_{|_{H}}) = \{ 0 \}. $$
In this way (\ref{eq:coomolSteiner1}) reduces to 
$$ 0 \longrightarrow H^{1}(H, \mathcal{I}_{V_{2} \cap H, \, H} \otimes \, {\Omega_{\mathbf{P}^5}^{1}(\log \mathcal{H})}^{\vee}_{|_{H}}) \longrightarrow H^{1}(H, \mathcal{I}_{V_{2} \cap H, \, H})^{\ell-1}. $$
If we restrict (\ref{eq:idealsheafsequence}) to $ H $ and we consider the induced cohomology sequence, we get that 
$$ H^{1}(H, \mathcal{I}_{V_{2} \cap H, \, H}) = {\mathbf{C}}^{k-1} $$
where $ k $ denotes the number of connected components of $ V_{2} \cap H $. Since $ V_{2} \cap H $ is connected, $ k = 1 $ and so (\ref{eq:keyfact2}) holds.
\end{proof}
Since isomorphic logarithmic bundles have the same set of unstable smooth conics, we have the following:

\begin{corollary}
If $ \ell \geq 9 $ then the map 
$$ \mathcal{D} \longmapsto \Omega_{\mathbf{P}^2}^{1}(\log \mathcal{D}) $$
is generically injective.
\end{corollary}

\begin{remark}
The hypothesis $ 1), 2), 3) $ of theorem $ 5.4 $ are necessary in order to apply Vall\`es' result for the case of $ \mathbf{P}^5 $. However we don't know what happens for arrangements made of $ \ell \in \{3,\ldots, 8\} $ conics. In the next two sections we will describe the cases of $ \ell = 1 $ and $ \ell = 2 $.
\end{remark}

\vfill\eject

\section{One conic}
Arrangements made of one smooth conic are not of \emph{Torelli type}. In this sense we have the following:

\begin{theorem}\label{T:1conica}
Let $ C \subset \mathbf{P}^2 $ be a smooth conic and let $ \mathcal{D} = \{C\} $. Then 
$$ \Omega_{\mathbf{P}^2}^{1}(\log \mathcal{D}) \cong \mathbf{TP}^2(-2). $$ 
\end{theorem}
\begin{proof}
Let consider the short exact sequence for $ \Omega_{\mathbf{P}^2}^{1}(\log \mathcal{D}) $:
\begin{equation}\label{eq:Anconaper1}
0 \longrightarrow \mathcal{O}_{\mathbf{P}^2}(-2) \buildrel \rm M \over \longrightarrow \mathcal{O}_{\mathbf{P}^2}(-1)^{3} \longrightarrow \Omega_{\mathbf{P}^2}^{1}(\log \mathcal{D})  \longrightarrow 0.
\end{equation} 
where $ M $ is the matrix associated to the injective map defined by the three partial derivatives of a quadratic polynomial defining $ C $. Without loss of generality we can assume that 
$$ M = \pmatrix{ x_{0} \cr x_{1} \cr x_{2} \cr} $$
and so, by tensor product with $ \mathcal{O}_{\mathbf{P}^2}(1) $ (\ref{eq:Anconaper1}) becomes the Euler sequence for $ \mathbf{TP}^2(-1) $, which concludes the proof.
\end{proof}

\begin{remark}
From (\ref{eq:Anconaper1}) we immediately get that our logarithmic bundle has Chern classes $ c_{1}(\Omega_{\mathbf{P}^2}^{1}(\log \mathcal{D})) = -1 $ and $ c_{2}(\Omega_{\mathbf{P}^2}^{1}(\log \mathcal{D})) = 1 $. Moreover, since the slope $ \mu(\Omega_{\mathbf{P}^2}^{1}(\log \mathcal{D})) = -\displaystyle{{1} \over {2}} > -1 $, Bohnhorst-Spindler criterion tells us that $ \Omega_{\mathbf{P}^2}^{1}(\log \mathcal{D}) $ is a stable bundle. So $ \Omega_{\mathbf{P}^2}^{1}(\log \mathcal{D}) $ belongs to the moduli space $ \mathbf{M}_{\mathbf{P}^2}(-1,1) $, which actually contains only the bundle $ \mathbf{TP}^2(-2) $, as we can see in~\cite{Wykno}. This is another way to prove theorem~\ref{T:1conica}.
\end{remark}

\begin{remark}
The previous theorem confirms the main result of~\cite{UY2} in the case of one smooth conic. Indeed, the defining equation of a conic is always of \emph{Sebastiani-Thom type} (see definition~\ref{d:sebthom}).
\end{remark}

\section{Pairs of conics}
Let's start with a classical result concerning a characterization of pairs of conics with normal crossings.
\begin{theorem}
Let $ C_{1} $ and $ C_{2} $ be smooth conics in $ \mathbf{P}^2 $. 
\\The following facts are equivalent:
\begin{itemize}
\item[$1)$] $ \mathcal{D} = \{C_{1}, C_{2}\} $ is an arrangement with normal crossings in $ \mathbf{P}^2 $;
\item[$2)$] the pencil of conics generated by $ C_{1} $ and $ C_{2} $ has four distinct base points $ \{P,Q,R,S\} $;
\item[$3)$] in the pencil of conics generated by $ C_{1} $ and $ C_{2} $ there are three distinct singular conics with singular points $ \{E,F,G\} $.
\end{itemize}
\end{theorem}

\begin{figure}[h]
    \centering
\includegraphics[width=65mm]{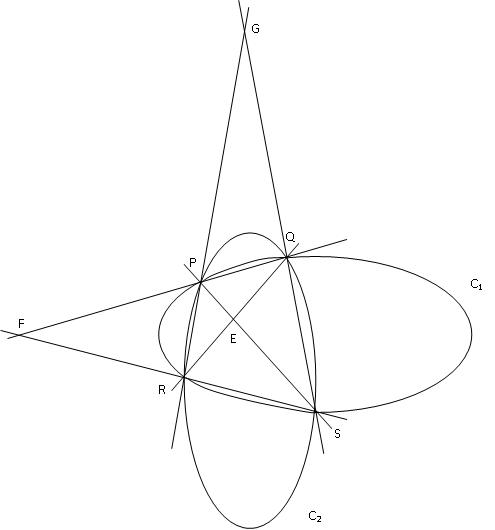}
    \caption{Two conics with normal crossings}
\label{flattening}
    \end{figure}

\begin{proof} 
The equivalence between $ 1) $ and $ 2) $ is a direct consequence of the B\'ezout's theorem, for example see ~\cite{Ki}. 
\\So, let assume that $ 2) $ holds. Then the three pairs of lines passing through disjoint pairs in $ \{P,Q,R,S\} $ are exactly the singular conics in the pencil generated by $ C_{1} $ and $ C_{2} $, which implies $ 3) $. 
\\The main part of the proof is to show that $ 3) $ implies $ 2) $. Let $ A, \, B \in GL(3,\mathbf{C}) $ be the symmetric matrices representing $ C_{1} $ and $ C_{2} $ with respect to the canonical basis $ \mathcal{C} $ of $ \mathbf{C}^3 $. By applying Sylvester's theorem to $ B $ we get that there exists $ G \in GL(3,\mathbf{C}) $ such that
\begin{equation}\label{eq:BSylvester}
^{t}GBG=I_{3}.
\end{equation}
Let $ A+tB $, with $ t \in \mathbf{C} $, be the symmetric matrix representing a generic element in the pencil generated by $ C_{1} $ and $ C_{2} $, we have that
$$ ^{t}G(A+tB)G=\,^{t}GAG+tI_{3} $$ 
and so 
$$ det(\,^{t}GAG+tI_{3})=det(G)^{2}det(A+tB). $$
Since $ G $ is not singular and $ 3) $ holds, the matrix $ A'=\,^{t}GAG $, which is clearly symmetric, has three \emph{distinct eigenvalues} (they are the opposite of the values giving singular conics in our pencil), that is $  A' $ is diagonalizable. This means that, if $ \mathcal{B}= \{v_{0},v_{1},v_{2}\}$ is a basis of $ \mathbf{C}^3 $ made of eigenvectors of $  A' $, $ \{\lambda_{0},\lambda_{1},\lambda_{2}\} $ are the corresponding eigenvalues, $ C= \mathcal{M}_{\mathcal{B}}^{\mathcal{C}}(id_{\mathbf{C}^3})$, $ \Lambda=diag(\lambda_{0},\lambda_{1},\lambda_{2}) $, then the representation of $  A' $ with respect to $ \mathcal{B} $ is 
$$ C^{-1} A'C=\Lambda. $$
We remark that we can always assume that $ C $ is an orthogonal matrix, i.e. $ ^{t}CC=C\,^{t}C=I_{3} $. Indeed, first of all we observe that, for $ i \not= j $, from 
$$ <A'v_{i},v_{j}>\,=\, <\lambda_{i}v_{i},v_{j}>\,=\,\lambda_{i}<v_{i},v_{j}> $$
and
$$ <A'v_{i},v_{j}>\,=\, ^{t}(A'v_{i})v_{j} \,=\, ^{t}v_{i}A'v_{j}\,=\, <v_{i},A'v_{j}> \,=\, \lambda_{j}<v_{i},v_{j}> $$
we get that
$$ (\lambda_{i}-\lambda_{j})<v_{i},v_{j}>\,=\,0 $$
which implies 
$$ <v_{i},v_{j}>\,=\,0. $$
This means that $ v_{0},v_{1},v_{2} $ are orthogonal with respect to the standard bilinear symmetric non degenerate form in $ \mathbf{C}^3 $ ($ <v,w> := \,^{t}vw $ for all $ v,w \in \mathbf{C}^3 $). Moreover they are orthogonal with respect to the scalar product defined by $ A' $: 
\begin{equation}\label{eq:A'ortogonali}
<v_{i},v_{j}>_{A'} \,=\, ^{t}v_{i}A'v_{j}=\lambda_{j}<v_{i},v_{j}>\,=\,0 
\end{equation}
for $ i \not= j $.
Finally these vectors satisfy $ <v_{i},v_{i}> \not= 0 $: if this is not the case, let assume for instance that $ v_{0} \not=0 $ satisfies $ <v_{0},v_{0}>\,=\,0 $. By doing the same computations as in (\ref{eq:A'ortogonali}), the previous equality implies that $ <v_{0},v_{0}>_{A'} \,=0 $. Thus $ v_{0},v_{1},v_{2} $ are three linearly independent elements of the orthogonal complement $ [v_{0}]^{\bot_{A'}} = \{v \in \mathbf{C}^3 \,| \, <v_{0},v>_{A'} \,=0\}$, that is $ dim_{\mathbf{C}}[v_{0}]^{\bot_{A'}} \geq 3 $. But we also know that 
$$ dim_{\mathbf{C}} [v_{0}]^{\bot_{A'}} = 3 - dim_{\mathbf{C}}<v_{0}> \,=\, 2  $$
which is a contradiction. \\
So, if we fix a choice of square root of $ <v_{i},v_{i}> $, we obtain that $ \mathcal{B}'= \{ {{v_{0}} \over {\sqrt{<v_{0},v_{0}>}}},{{v_{1}} \over {\sqrt{<v_{1},v_{1}>}}},{{v_{2}} \over {\sqrt{<v_{2},v_{2}>}}}\}$ is a basis of eigenvectors of $ A' $ orthonormal with respect to the standard bilinear non degenerate form in $ \mathbf{C}^3$ and $ O= \mathcal{M}_{\mathcal{B}'}^{\mathcal{C}}(id_{\mathbf{C}^3}) $ is an orthogonal matrix such that
$$ ^{t}OA'O=\Lambda $$
or equivalently
\begin{equation}\label{eq:Adiagonale}
^{t}(GO)A(GO)=\Lambda. 
\end{equation}
From (\ref{eq:BSylvester}) we get that
\begin{equation}\label{eq:Bdiagonale}
^{t}(GO)B(GO)=\,^{t}O(\,^{t}GBG)O\,=\,^{t}OO=I_{3}. 
\end{equation}
Thus, (\ref{eq:Adiagonale}) and (\ref{eq:Bdiagonale}) tell us that $ A $ and $ B $ are \emph{simultaneously diagonalizable by congruence} in the basis $ \mathcal{B}'' $ of $ \mathbf{C}^3$ made of the columns of $ GO $. In particular, in this frame the equations of $ C_{1} $ and $ C_{2} $ are, respectively,
\begin{equation}\label{eq:C1} 
\lambda_{0}x_{0}^2+\lambda_{1}x_{1}^2+\lambda_{2}x_{2}^2=0   
\end{equation}
\begin{equation}\label{eq:C2} 
x_{0}^2+x_{1}^2+x_{2}^2=0. 
\end{equation}
It's not hard to see that, if we fix a choice of square root of $ \lambda_{2}-\lambda_{1} $, then $ C_{1} $ and $ C_{2} $ intersect in four distinct points
$$ [\sqrt{\lambda_{2}-\lambda_{1}},\pm\sqrt{\lambda_{0}-\lambda_{2}},\pm\sqrt{\lambda_{1}-\lambda_{0}}] $$
which we denote by $ \{P,Q,R,S\} $. This concludes the proof of $ 2) $.
\end{proof}

\begin{remark}
As a consequence of the smoothness of $ C_{1} $ and $ C_{2} $ we immediately have that any three of $ \{P,Q,R,S\} $ are not collinear.
\end{remark}

\begin{remark}
In order to get from $ 3) $ the simultaneous diagonalization of the two conics we don't need that $ C_{1} $ is smooth. 
\end{remark}

\begin{remark}
We can check with direct computations that the elements of the basis $ \mathcal{B}'' $ with respect to which $ C_{1} $ and $ C_{2} $ have equations (\ref{eq:C1}) and (\ref{eq:C2}) are \emph{representative vectors} in $ \mathbf{C}^3 $ of the points $ \{E,F,G\} $. Indeed, let $ \{w_{0}, w_{1}, w_{2}\} $ be vectors corresponding to $ \{E,F,G\} $ and let $ \{t_{0}, t_{1}, t_{2}\} \in \mathbf{C} $ such that $ det(A+t_{i}B)=0 $. Since $ w_{i} $ represents the singular point of the conic given by $ A+t_{i}B $, we have that
$$ Aw_{i}=-t_{i}Bw_{i}; $$
but from (\ref{eq:Adiagonale}) and (\ref{eq:Bdiagonale}) we know that 
$$ A = (\,^{t}(GO))^{-1}\Lambda(GO)^{-1} $$
$$ B = (\,^{t}(GO))^{-1}(GO)^{-1}.  $$
Thus we get that
$$ (\,^{t}(GO))^{-1}\Lambda(GO)^{-1}w_{i}=-t_{i}(\,^{t}(GO))^{-1}(GO)^{-1}w_{i} $$
that is
$$ \Lambda(GO)^{-1}w_{i}=-t_{i}(GO)^{-1}w_{i}. $$
This means that $ (GO)^{-1}w_{i} $ is an eigenvector of $ \Lambda $ corresponding to the eigenvalue $ \lambda_{i}= -t_{i}$ and, being $ \Lambda $ a diagonal matrix, $ (GO)^{-1}w_{i} = e_{i}$, the $ i $-th vector of the canonical basis $ \mathcal{C} $ of $ \mathbf{C}^3 $. In this way we get that $ w_{i} = (GO)e_{i} $, the $ i $-th column of the matrix $ GO $, as desired.
\end{remark}

\begin{corollary}
Let $ \mathcal{D} = \{C_{1}, C_{2}\} $ be an arrangement of smooth conics with normal crossings and let $ \{E,F,G\} $ as in \emph{theorem $ 5.10 $}. Then the matrices associated to $ C_{1} $ and $ C_{2} $ in a frame given by  representative vectors of $ \{E,F,G\} $ are of the form $ diag(a_{1},b_{1},-1) $ and $ diag(a_{2},b_{2},-1) $, where $ a_{1},b_{1},a_{2},b_{2} \in \mathbf{C}-\{0\}.$
\end{corollary}

Now let's come back to the \emph{Torelli problem}. \\

Let $ \mathcal{D} = \{C_{1}, C_{2}\} $ be an arrangement of smooth conics with normal crossings and let $\Omega_{\mathbf{P}^2}^{1}(\log \mathcal{D})$ be the logarithmic bundle attached to it.

\begin{remark}
Theorem \ref{T:Ancona} implies that $\Omega_{\mathbf{P}^2}^{1}(\log \mathcal{D})$ is a rank $ 2 $ vector bundle over $ \mathbf{P}^2 $ with an exact sequence of the form
\begin{equation}\label{eq:Anconasequence}
0 \longrightarrow \mathcal{O}_{\mathbf{P}^2}(-2)^{2} \longrightarrow \mathcal{O}_{\mathbf{P}^2}(-1)^{3} \oplus \mathcal{O}_{\mathbf{P}^2} \longrightarrow \Omega_{\mathbf{P}^2}^{1}(\log \mathcal{D})  \longrightarrow 0.
\end{equation}
Its Chern polynomial is obtained by truncating to degree $ 2 $ the expression
$$ {{(1-t)^{3}} \over {(1-2t)^{2}}} $$ 
that is
$$ p_{\Omega_{\mathbf{P}^2}^{1}(\log \mathcal{D})}(t) = 1+t+3t^{2}. $$
In particular its Chern classes are 
\begin{equation}\label{eq:chernclasses}
c_{1}(\Omega_{\mathbf{P}^2}^{1}(\log \mathcal{D})) = 1, \,  c_{2}(\Omega_{\mathbf{P}^2}^{1}(\log \mathcal{D})) = 3.
\end{equation}
Moreover its \emph{normalized bundle} is $\Omega_{\mathbf{P}^2}^{1}(\log \mathcal{D})_{norm} = \Omega_{\mathbf{P}^2}^{1}(\log \mathcal{D})(-1) $ with $ c_{1}(\Omega_{\mathbf{P}^2}^{1}(\log \mathcal{D})_{norm}) = -1 $ and $ c_{2}(\Omega_{\mathbf{P}^2}^{1}(\log \mathcal{D})_{norm}) = 3 $. If we do the tensor product of (\ref{eq:Anconasequence}) with $ \mathcal{O}_{\mathbf{P}^2}(-1) $ and then we consider the induced long exact cohomology sequence, we get that 
$$ H^{0}(\mathbf{P}^2,\Omega_{\mathbf{P}^2}^{1}(\log \mathcal{D})_{norm}) = \{0\}, $$
that is $ \Omega_{\mathbf{P}^2}^{1}(\log \mathcal{D}) $ is stable,~\cite{OSS}. The same is true for $\Omega_{\mathbf{P}^2}^{1}(\log \mathcal{D})_{norm} $. \\
 $ \Omega_{\mathbf{P}^2}^{1}(\log \mathcal{D})_{norm} \in \mathbf{M}_{\mathbf{P}^2}(-1,3) $, the moduli space of stable rank $ 2 $ bundles on $ \mathbf{P}^2 $ such that $ c_{1} = -1 $ and $ c_{2} = 3 $. In general, as we can see in~\cite{OSS}, given a vector bundle $ E $ in the moduli space $ \mathbf{M}_{\mathbf{P}^2}(c_{1},c_{2}) $, we have that $ h^{0}(End E) = 1 $ ($ E $ is \emph{simple}) and $ h^{2}(End E) = 0 $. The Riemann-Roch theorem implies that the Euler characteristic of $ End E $ is $ \chi(EndE) = c_{1}^{2}(E)-4c_{2}(E)+4 $ and so  
\begin{equation}\label{eq:dimmoduli}
dim \mathbf{M}_{\mathbf{P}^2}(c_{1},c_{2}) = h^{1}(EndE) = -c_{1}^{2}(E)+4c_{2}(E)-3. 
\end{equation}
In our case, the previous formula gives us
\begin{equation}\label{eq:M(-1,3)}
dim \mathbf{M}_{\mathbf{P}^2}(-1,3) = 8
\end{equation}
but the number of parameters associated to a pair of conics is $ 10 $. \\
Thus we conclude that such $ \mathcal{D} $ can't be an arrangement of \emph{Torelli type}.
\end{remark}

In order to describe the pairs of conics giving isomorphic logarithmic bundles we need the following:

\begin{proposition}
Let $ \mathcal{D} = \{C_{1}, C_{2}\} $ be an arrangement of smooth conics with normal crossings and let $ \{E,F,G\} $ as in \emph{theorem $ 5.10 $}. Then:
\begin{itemize}
\item[$1)$] $ \{E,F,G\} $ is the zero locus of the non-zero section of $\Omega_{\mathbf{P}^2}^{1}(\log \mathcal{D})$;
\item[$2)$] the three lines through any two of the points in $ \{E,F,G\} $ are exactly the \emph{jumping lines} of $\Omega_{\mathbf{P}^2}^{1}(\log \mathcal{D})_{norm}$ (see figure 5.2).
\end{itemize} 
\end{proposition}

\begin{figure}[h]
    \centering
\includegraphics[width=60mm]{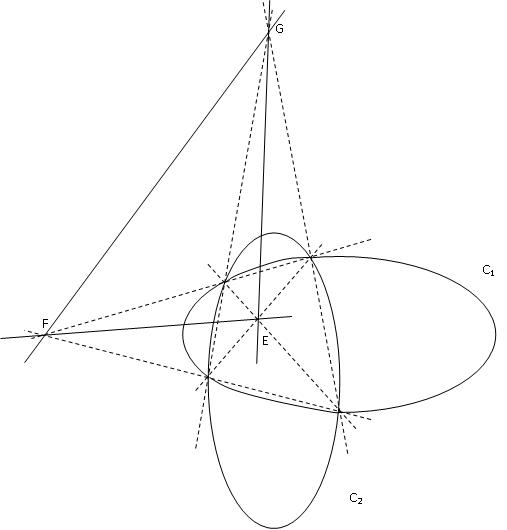}
    \caption{Jumping lines of $\Omega_{\mathbf{P}^2}^{1}(\log \mathcal{D})_{norm}$}
\label{flattening}
    \end{figure}

\begin{proof} From (\ref{eq:chernclasses}) we get that $ \Omega_{\mathbf{P}^2}^{1}(\log \mathcal{D}) $ has one non-zero section with three zeroes; we want to prove the zeroes are $ \{E,F,G\} $. So, let assume that $ A=(a_{ij}) $ and $ B=(b_{ij}) $ are the matrices representing $ C_{1} $ and $ C_{2} $ with respect to the canonical basis of $ \mathbf{C}^3 $. From theorem \ref{T:Ancona} we get that $ \Omega_{\mathbf{P}^2}^{1}(\log \mathcal{D}) $ admits the exact sequence
\begin{equation}\label{eq:anper2}
0 \longrightarrow \mathcal{O}_{\mathbf{P}^2}(-2)^{2} \buildrel \rm M \over \longrightarrow \mathcal{O}_{\mathbf{P}^2}(-1)^{3} \oplus \mathcal{O}_{\mathbf{P}^2} \longrightarrow \Omega_{\mathbf{P}^2}^{1}(\log \mathcal{D})  \longrightarrow 0
\end{equation}
where
$$ M = \pmatrix{ 
2 $$ \sum^2_{i=0}$$ a_{0i}x_{i} & 2 $$\sum^2_{i=0}$$  b_{0i}x_{i} \cr 
2 $$\sum^2_{i=0}$$ a_{1i}x_{i}  &  2 $$\sum^2_{i=0}$$  b_{1i}x_{i}  \cr 
2 $$\sum^2_{i=0}$$ a_{2i}x_{i}  &  2 $$\sum^2_{i=0}$$ b_{2i}x_{i}  \cr 
$$\sum^2_{i,j=0} $$ a_{ij}x_{i}x_{j} & 0 \cr
}. $$
We are searching for $ x = (x_{0},x_{1},x_{2}) \in \mathbf{C}^3 -\{0\} $ such that the linear part of $ M $ has rank $ 1 $, that is the solutions of 
$$ Ax = \lambda Bx $$
for certain $ \lambda \in \mathbf{C} $ ($ \lambda $ is an eigenvalue of $ AB^{-1} $ and $ x $ is the corresponding eigenvector). In other words, any such $ x $ has to be a representative vector for the singular point of the conic associated to $ A - \lambda B $ and this shows $ 1) $. \\
The \emph{jumping lines} of $\Omega_{\mathbf{P}^2}^{1}(\log \mathcal{D})_{norm} $ are the lines $ \ell $ in $ \mathbf{P}^2 $ over which this bundle doesn't split as in the Grauert-Mulich theorem (\cite{Wykno}), that is like 
$$ \mathcal{O}_{\ell} \oplus \mathcal{O}_{\ell}(-1). $$
Equivalently, a line $ \ell $ is a jumping line if it satisfies the following condition:
\begin{equation}\label{eq:jumpingline}
H^{0}(\ell,\Omega_{\mathbf{P}^2}^{1}(\log \mathcal{D})_{norm_{|\ell}}(-1)) \not= \{0\}.
\end{equation}
Since $ c_{1}(\Omega_{\mathbf{P}^2}^{1}(\log \mathcal{D})_{norm}) = -1 $, the set of jumping lines is a codimension $ 2 $ subvariety of $ Gr({\mathbf{P}^1},{\mathbf{P}^2}) $ with degree 
$$ { \displaystyle {c_{2}(\Omega_{\mathbf{P}^2}^{1}(\log \mathcal{D})_{norm})(c_{2}(\Omega_{\mathbf{P}^2}^{1}(\log \mathcal{D})_{norm}) - 1)} \over {2}} = 3 $$
that is $ \Omega_{\mathbf{P}^2}^{1}(\log \mathcal{D})_{norm} $ has three jumping lines, \cite{Wykno}. \\
We want to show that the lines through any two of the points in $ \{E,F,G\} $ verify (\ref{eq:jumpingline}). In order to do that, we can always assume that the base points of the pencil generated by $ C_{1} $ and $ C_{2} $ are
\begin{equation}\label{eq:basepoints}
P = [1,0,0], \,Q = [0,1,0], \,R = [0,0,1], \,S = [1,1,1].
\end{equation}
Indeed, there's a unique projective transformation $ \tau : \mathbf{P}^2 \rightarrow \mathbf{P}^2 $ sending four points in general position to the points in (\ref{eq:basepoints}). In this case the equations of $ C_{1} $ and $ C_{2} $ are, respectively,
\begin{equation}\label{eq:eqC1}
x_{0}x_{1}+ax_{0}x_{2}+bx_{1}x_{2}=0 
\end{equation}
\begin{equation}\label{eq:eqC2}
x_{0}x_{1}+cx_{0}x_{2}+dx_{1}x_{2}=0
\end{equation}
where $ a,b,c,d \in \mathbf{C} - \{0\} $ satisfy 
\begin{equation}\label{eq:conditions}
1+a+b = 1+c+d = 0, \, a \not= c, \,b \not= d
\end{equation}
and we have that
\begin{equation}\label{eq:EFG}
E = [1,1,0],F=[1,0,1],G=[0,1,1].
\end{equation}
The equations of the lines $ \ell_{0}, \ell_{1}, \ell_{2} $ through $ E $ and $ F $, $ E $ and $ G $, $ F $ and $ G $ are, respectively,
\begin{equation}\label{eq:eql1}
x_{0} = x_{1}+x_{2}
\end{equation}
\begin{equation}\label{eq:eql2}
x_{1} = x_{0}+x_{2}
\end{equation}
\begin{equation}\label{eq:eql3}
x_{2} = x_{0}+x_{1}.
\end{equation}
Applying $ \mathcal{H}om( \cdot,\mathcal{O}_{\mathbf{P}^{2}}) $ to (\ref{eq:anper2}) we get the following exact sequence:
\begin{equation}\label{eq:Anconadualseq}
0 \longrightarrow \Omega_{\mathbf{P}^2}^{1}(\log \mathcal{D})^{\vee} \longrightarrow \mathcal{O}_{\mathbf{P}^2}(1)^{3} \oplus \mathcal{O}_{\mathbf{P}^2} \buildrel \rm N \over \longrightarrow \mathcal{O}_{\mathbf{P}^2}(2)^{2} \longrightarrow 0
\end{equation}
where
$$ N = \,^{t}M = \pmatrix{ 
x_{1}+ax_{2} & x_{0}+bx_{2} & ax_{0}+bx_{1} & x_{0}x_{1}+ax_{0}x_{2}+bx_{1}x_{2} \cr 
x_{1}+cx_{2} & x_{0}+dx_{2} & cx_{0}+dx_{1} & 0 \cr
}. $$ 
Restricting (\ref{eq:Anconadualseq}) to $ \ell_{i} $ we obtain
\begin{equation}\label{eq:Anconadualseqi}
0 \longrightarrow \Omega_{\mathbf{P}^2}^{1}(\log \mathcal{D})^{\vee}_{|\ell_{i}} \longrightarrow \mathcal{O}_{\mathbf{P}^2}(1)^{3}_{|\ell_{i}} \oplus {\mathcal{O}_{\mathbf{P}^2}}_{|\ell_{i}} \buildrel \rm N_{i} \over \longrightarrow \mathcal{O}_{\mathbf{P}^2}(2)^{2}_{|\ell_{i}} \longrightarrow 0
\end{equation}
where
$$ N_{0} = N_{|\ell_{0}} = \pmatrix{ 
x_{1}+ax_{2} & x_{1}+(1+b)x_{2} & x_{1}(a+b)+ax_{2} & x_{1}^{2}+ax_{2}^{2} \cr 
x_{1}+cx_{2} & x_{1}+(1+d)x_{2} & x_{1}(c+d)+cx_{2} & 0 \cr
} $$ 
$$ N_{1} = N_{|\ell_{1}} = \pmatrix{ 
x_{0}+(1+a)x_{2} & x_{0}+bx_{2} & x_{0}(a+b)+bx_{2} & x_{0}^{2}+bx_{2}^{2} \cr 
x_{0}+(1+c)x_{2} & x_{0}+dx_{2} & x_{0}(c+d)+dx_{2} & 0 \cr
} $$
$$ N_{2} = N_{|\ell_{2}} = \pmatrix{ 
ax_{0}+x_{1}(1+a) & x_{0}(1+b)+bx_{1} & ax_{0}+bx_{1} & ax_{0}^{2}+bx_{1}^{2} \cr 
cx_{0}+x_{1}(1+c) & x_{0}(1+d)+dx_{1} & cx_{0}+dx_{1} & 0 \cr
}. $$
Now, the induced cohomology exact sequence of (\ref{eq:Anconadualseqi}) is
$$ 0 \rightarrow H^{0}(\ell_{i}, \Omega_{\mathbf{P}^2}^{1}(\log \mathcal{D})^{\vee}_{|\ell_{i}}) \rightarrow H^{0}(\ell_{i}, \mathcal{O}_{\mathbf{P}^2}(1)_{|\ell_{i}})^{3} \oplus H^{0}(\ell_{i}, {\mathcal{O}_{\mathbf{P}^2}}_{|\ell_{i}}) \buildrel \rm \mathcal{N}_{i} \over\rightarrow $$ 
$$ \rightarrow H^{0}(\ell_{i}, \mathcal{O}_{\mathbf{P}^2}(2)_{|\ell_{i}})^{2} \rightarrow H^{1}(\ell_{i}, \Omega_{\mathbf{P}^2}^{1}(\log \mathcal{D})^{\vee}_{|\ell_{i}}) \longrightarrow 0 \quad\quad\quad\quad\quad $$
where $ \mathcal{N}_{0},\mathcal{N}_{1},\mathcal{N}_{2} $ are the following $ 6 \times 7 $ matrices:
$$ \mathcal{N}_{0} = \pmatrix{ 1 & 0 & 1 & 0 & a+b & 0 & 1 \cr
a & 1 & 1+b & 1 & a & a+b & 0 \cr
0 & a & 0 & 1+b & 0 & a & a \cr
1 & 0 & 1 & 0 & c+d & 0 & 0 \cr
c & 1 & 1+d & 1 & c & c+d & 0 \cr 
0 & c & 0 & 1+d & 0 & c & 0 \cr
}  $$
$$ \mathcal{N}_{1} = \pmatrix{ 1 & 0 & 1 & 0 & a+b & 0 & 1 \cr
1+a & 1 & b & 1 & b & a+b & 0 \cr
0 & 1+a & 0 & b & 0 & b & b \cr
1 & 0 & 1 & 0 & c+d & 0 & 0 \cr
1+c & 1 & d & 1 & d & c+d & 0 \cr 
0 & 1+c & 0 & d & 0 & d & 0 \cr
}  $$
$$ \mathcal{N}_{2} = \pmatrix{ a & 0 & 1+b & 0 & a & 0 & a \cr
1+a & a & b & 1+b & b & a & 0 \cr
0 & 1+a & 0 & b & 0 & b & b \cr
c & 0 & 1+d & 0 & c & 0 & 0 \cr
1+c & c & d & 1+d & d & c & 0 \cr 
0 & 1+c & 0 & d & 0 & d & 0 \cr
}.  $$
By using Serre duality we have that 
$$ H^{1}(\ell_{i}, \Omega_{\mathbf{P}^2}^{1}(\log \mathcal{D})^{\vee}_{|\ell_{i}}) = H^{0}(\ell_{i}, \Omega_{\mathbf{P}^2}^{1}(\log \mathcal{D})_{|\ell_{i}}(-2))^{\vee} = $$
$$ = H^{0}(\ell_{i}, {\Omega_{\mathbf{P}^2}^{1}(\log \mathcal{D})_{norm}}_{|\ell_{i}}(-1))^{\vee}. $$
Thus it suffices to show that $ \mathcal{N}_{0},\mathcal{N}_{1},\mathcal{N}_{2} $ have not maximal rank.
Since (\ref{eq:conditions}) holds, it's not hard to see that $ \mathcal{N}_{0},\mathcal{N}_{1},\mathcal{N}_{2} $ have rank $ 5 $ and this concludes the proof.
\end{proof}

\begin{remark}
The previous proposition tells us that there are two ways to make a correspondence between the logarithmic bundle $ \Omega_{\mathbf{P}^2}^{1}(\log \mathcal{D}) $ and the points $ \{E,F,G\} $.
\end{remark}

The main theorem concerning pairs of conics is the following:
\begin{theorem}\label{T:coppieconiche}
Let $ \mathcal{D}_{1} = \{C_{1}, C_{2}\} $ and $ \mathcal{D}_{2} = \{C'_{1}, C'_{2}\} $ be arrangements of smooth conics with normal crossings. Then 
\begin{equation}\label{eq:isomorphism}
\Omega_{\mathbf{P}^2}^{1}(\log \mathcal{D}_{1}) \cong \Omega_{\mathbf{P}^2}^{1}(\log \mathcal{D}_{2}).
\end{equation}
if and only if $ \mathcal{D}_{1} $ and $ \mathcal{D}_{2} $ have the same four tangent lines.
\end{theorem}

\begin{figure}[h]
    \centering
\includegraphics[width=50mm]{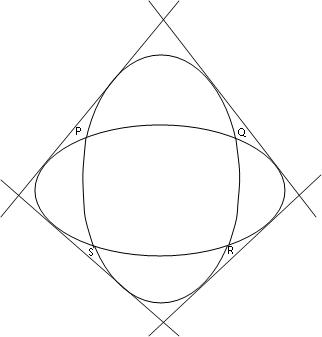}
    \caption{Four tangent lines of a pair of conics}
\label{flattening}
    \end{figure}

\begin{proof}
Let assume that (\ref{eq:isomorphism}) holds, by using proposition $ 5.16 $ we can associate to each bundle the same set of points $ \{E,F,G\} $. From corollary $ 5.14 $ we get that there's a frame of $ \mathbf{C}^3 $ made of representative vectors of $ \{E,F,G\} $ in which $ C_{1}, C_{2}, C'_{1}, C'_{2} $ have equations, respectively,
$$ a_{1}x_{0}^2+b_{1}x_{1}^2-x_{2}^2=0  $$
$$ a_{2}x_{0}^2+b_{2}x_{1}^2-x_{2}^2=0  $$
$$ c_{1}x_{0}^2+d_{1}x_{1}^2-x_{2}^2=0  $$
$$ c_{2}x_{0}^2+d_{2}x_{1}^2-x_{2}^2=0  $$
where $ a_{1}, b_{1}, a_{2}, b_{2}, c_{1}, d_{1}, c_{2}, d_{2} \in \mathbf{C}-\{0\} $ and $ a_{1} \not= a_{2} $, $ b_{1} \not= b_{2} $, $ c_{1} \not= c_{2} $, $ d_{1} \not= d_{2} $, $ \displaystyle{{b_{1}} \over {a_{1}}} \not= {{b_{2}} \over {a_{2}}} $, $ \displaystyle{{d_{1}} \over {c_{1}}} \not= {{d_{2}} \over {c_{2}}} $ (these properties assure, respectively, that we have two smooth conics with normal crossings). 
Our aim is to find relations between the coefficients of the previous equations in order to have (\ref{eq:isomorphism}). \\
We recall the two exact sequences
$$ 0 \longrightarrow \Omega_{\mathbf{P}^2}^{1}(\log \mathcal{D}_{1})^{\vee} \longrightarrow \mathcal{O}_{\mathbf{P}^2}(1)^{3} \oplus \mathcal{O}_{\mathbf{P}^2} \buildrel \rm N_{1} \over \longrightarrow \mathcal{O}_{\mathbf{P}^2}(2)^{2} \longrightarrow 0 $$
$$ 0 \longrightarrow \Omega_{\mathbf{P}^2}^{1}(\log \mathcal{D}_{2})^{\vee} \longrightarrow \mathcal{O}_{\mathbf{P}^2}(1)^{3} \oplus \mathcal{O}_{\mathbf{P}^2} \buildrel \rm N_{2} \over \longrightarrow \mathcal{O}_{\mathbf{P}^2}(2)^{2} \longrightarrow 0 $$
where
$$ N_{1} = \pmatrix{ 
2a_{1}x_{0} & 2b_{1}x_{1} & -2x_{2} & a_{1}x_{0}^2+b_{1}x_{1}^2-x_{2}^2 \cr 
2a_{2}x_{0} & 2b_{2}x_{1} & -2x_{2} & 0 \cr
} $$
$$ N_{2} = \pmatrix{ 
2c_{1}x_{0} & 2d_{1}x_{1} & -2x_{2} & c_{1}x_{0}^2+d_{1}x_{1}^2-x_{2}^2 \cr 
2c_{2}x_{0} & 2d_{2}x_{1} & -2x_{2} & 0 \cr
}. $$
(\ref{eq:isomorphism}) is equivalent to the fact that there exist two invertible matrices
\begin{equation}\label{eq:M'}
M' = \pmatrix{ 
\alpha & \beta \cr 
\gamma & \delta \cr
}
\end{equation}
\begin{equation}\label{eq:M''}
M'' = \pmatrix{ 
E & F & G & f_{1} \cr
H & I & L & f_{2} \cr
M & N & O & f_{3} \cr
0 & 0 & 0 & \theta \cr
}
\end{equation}
with $ \alpha, \dots, \delta,E, \dots, O, \theta \in \mathbf{C} $ and $ f_{j} = f^{0}_{j}x_{0}+f^{1}_{j}x_{1}+f^{2}_{j}x_{2} $, $ j \in \{1,2,3\} $, complex linear forms, such that the following diagram commutes:
\vfill\eject
$$ \mathcal{O}_{\mathbf{P}^2}(1)^{3} \oplus \mathcal{O}_{\mathbf{P}^2} \buildrel \rm N_{1} \over \longrightarrow \mathcal{O}_{\mathbf{P}^2}(2)^{2} $$ 
$$ \,\,\,\,\,\,\,\,\,\,M'' \downarrow  \,\,\,\,\,\,\,\,\,\,\,\,\,\,\,\,\,\,\,\,\,\,\,\,\,\,\,\,\,\ \downarrow M'  $$
$$ \mathcal{O}_{\mathbf{P}^2}(1)^{3} \oplus \mathcal{O}_{\mathbf{P}^2} \buildrel \rm N_{2} \over \longrightarrow \mathcal{O}_{\mathbf{P}^2}(2)^{2} $$ 
First, let's equate the coefficients of the matrices $ M'N_{1} $ and $ N_{2}M'' $ that don't belong to the fourth column. We get the following conditions:
$$ \alpha a_{1} + \beta a_{2} = c_{1} E $$
$$ H = M = 0 $$ 
$$ \alpha b_{1} + \beta b_{2} = d_{1} I $$
$$ F = N = 0 $$
$$ \alpha + \beta = O $$
$$ G = L = 0 $$
$$ \gamma a_{1} + \delta a_{2} = c_{2} E $$
$$ \gamma b_{1} + \delta b_{2} = d_{2} I $$
$$ \gamma + \delta = O. $$
These relations reduce to:
\begin{equation}\label{eq:E}
E = \displaystyle{{a_{1}} \over {c_{1}}} \alpha + {{a_{2}} \over {c_{1}}} \beta = \displaystyle{{a_{1}} \over {c_{2}}} \gamma + \displaystyle{{a_{2}} \over {c_{2}}} \delta
\end{equation}
\begin{equation}\label{eq:I}
I = \displaystyle{{b_{1}} \over {d_{1}}} \alpha + \displaystyle{{b_{2}} \over {d_{1}}} \beta = \displaystyle{{b_{1}} \over {d_{2}}} \gamma + \displaystyle{{b_{2}} \over {d_{2}}} \delta
\end{equation}
\begin{equation}\label{eq:O}
O = \alpha + \beta = \gamma + \delta.
\end{equation}
If we compute $ \gamma $ from (\ref{eq:O}) and then we substitute its expression in (\ref{eq:E}) we get
\begin{equation}\label{eq:D}
\delta = \displaystyle {{a_{1}(c_{2}-c_{1})} \over {c_{1}(a_{2}-a_{1})}} \alpha + \displaystyle {{(a_{2}c_{2}-a_{1}c_{1})} \over {c_{1}(a_{2}-a_{1})}} \beta. 
\end{equation}
Thus, (\ref{eq:O}) and (\ref{eq:D}) imply that 
\begin{equation}\label{eq:C}
\gamma = \displaystyle {{(a_{2}c_{1}-a_{1}c_{2})} \over {c_{1}(a_{2}-a_{1})}} \alpha + \displaystyle {{a_{2}(c_{1}-c_{2})} \over {c_{1}(a_{2}-a_{1})}} \beta.
\end{equation}
Now, by equating the corresponding coefficients of the last column of $ M'N_{1} $ and $ N_{2}M'' $ we obtain:
\begin{equation}\label{eq:X1400}
\alpha a_{1} = 2 c_{1} f^{0}_{1} + c_{1} \theta
\end{equation}
\begin{equation}\label{eq:X1411}
\alpha b_{1} = 2 d_{1} f^{1}_{2} + d_{1} \theta
\end{equation}
\begin{equation}\label{eq:X1422}
\alpha = 2 f^{2}_{3} + \theta
\end{equation}
\begin{equation}\label{eq:X1401}
c_{1} f^{1}_{1} + d_{1} f^{0}_{2} = 0
\end{equation}
\begin{equation}\label{eq:X1402}
c_{1} f^{2}_{1} - f^{0}_{3} = 0
\end{equation}
\begin{equation}\label{eq:X1412}
d_{1} f^{2}_{2} - f^{1}_{3} = 0
\end{equation}
\begin{equation}\label{eq:X2400}
\gamma a_{1} = 2 c_{2} f^{0}_{1}
\end{equation}
\begin{equation}\label{eq:X2411}
\gamma b_{1} = 2 d_{2} f^{1}_{2}
\end{equation}
\begin{equation}\label{eq:X2422}
\gamma = 2 f^{2}_{3} 
\end{equation}
\begin{equation}\label{eq:X2401}
c_{2} f^{1}_{1} + d_{2} f^{0}_{2} = 0
\end{equation}
\begin{equation}\label{eq:X2402}
c_{2} f^{2}_{1} - f^{0}_{3} = 0
\end{equation}
\begin{equation}\label{eq:X2412}
d_{2} f^{2}_{2} - f^{1}_{3} = 0.
\end{equation}
First of all, since $ \displaystyle{{d_{1}} \over {c_{1}}} \not= {{d_{2}} \over {c_{2}}} $, by using (\ref{eq:X1401}) and (\ref{eq:X2401}) we immediately get that 
$$ f^{0}_{2} = f^{1}_{1} = 0. $$
Similarly, equations (\ref{eq:X1402}), (\ref{eq:X2402}) and (\ref{eq:X1412}), (\ref{eq:X2412}) imply, respectively, that 
$$ f^{2}_{1} = f^{0}_{3} = 0 $$
$$ f^{2}_{2} = f^{1}_{3} = 0. $$
By computing $ \theta $ from (\ref{eq:X1422}), equations (\ref{eq:X1400}) and (\ref{eq:X1411}) become, respectively
$$ f^{2}_{3} = \displaystyle{{c_{1}-a_{1}} \over {2c_{1}}} \alpha + f^{0}_{1} $$ 
\begin{equation}\label{eq:f23a}
f^{2}_{3} = \displaystyle{{d_{1}-b_{1}} \over {2d_{1}}} \alpha + f^{1}_{2}.
\end{equation}
In particular we have
\begin{equation}\label{eq:f01a}
f^{0}_{1} = \displaystyle \left({{d_{1}-b_{1}} \over {2d_{1}}} + {{a_{1}-c_{1}} \over {2c_{1}}} \right) \alpha + f^{1}_{2}.
\end{equation}
Moreover, from (\ref{eq:X2400}), (\ref{eq:X2411}), (\ref{eq:X2422}) we get, respectively,
\begin{equation}\label{eq:f01b}
f^{0}_{1} = \displaystyle {{a_{1}} \over {2c_{2}}}\gamma
\end{equation}
\begin{equation}\label{eq:f12a}
f^{1}_{2} = \displaystyle {{b_{1}} \over {2d_{2}}}\gamma
\end{equation}
\begin{equation}\label{eq:f23b}
f^{2}_{3} = \displaystyle {{\gamma} \over {2}}.
\end{equation}
If we put together (\ref{eq:f23a}), (\ref{eq:f23b}), (\ref{eq:f12a}) and we use (\ref{eq:C}), we obtain the following well defined expression for $ \beta $:
\begin{equation}\label{eq:Bfinale}
\beta = \displaystyle {{a_{2}b_{1}c_{1}(d_{2}-d_{1})+a_{1}b_{1}(c_{2}d_{1}-c_{1}d_{2})+a_{1}d_{1}d_{2}(c_{1}-c_{2})} \over {a_{2}d_{1}(b_{1}-d_{2})(c_{1}-c_{2})}} \alpha.
\end{equation}
In this way (\ref{eq:C}) and (\ref{eq:D}) become, respectively,
\begin{equation}\label{eq:Cfinale}
\gamma = \displaystyle {{d_{2}(b_{1}-d_{1})} \over {d_{1}(b_{1}-d_{2})}} \alpha
\end{equation}
\begin{equation}\label{eq:Dfinale}
\delta = \displaystyle {{b_{1}d_{2}(a_{2}c_{2}-a_{1}c_{1})+b_{1}c_{2}d_{1}(a_{1}-a_{2})-a_{1}d_{1}d_{2}(c_{2}-c_{1})} \over {a_{2}d_{1}(b_{1}-d_{2})(c_{1}-c_{2})}} \alpha.
\end{equation}
Thus, if we choose $ \alpha \in \mathbf{C}-\{0\} $ and if the condition 
\begin{equation}\label{eq:M'invertibile}
(b_{1}-d_{2})c_{2}d_{1}+(d_{1}-b_{1})c_{1}d_{2} \not= 0
\end{equation}
is satisfied, then the matrix $ M' $ in (\ref{eq:M'}) is invertible. \\
By using (\ref{eq:f01a}), (\ref{eq:f01b}), (\ref{eq:f12a}) and (\ref{eq:Cfinale}) we get the first resolubility condition for our system of equations:
\begin{equation}\label{eq:1condition}
a_{1}b_{1}(c_{2}d_{1}-c_{1}d_{2})+b_{1}c_{1}c_{2}(d_{2}-d_{1})+a_{1}d_{1}d_{2}(c_{1}-c_{2}) = 0.
\end{equation}
With (\ref{eq:Cfinale}) we are able to compute final expressions for $ f^{0}_{1},f^{1}_{2},f^{2}_{3},\theta $; in particular we have
\begin{equation}\label{eq:tetafinale}
\theta = \displaystyle {{b_{1}(d_{1}-d_{2})} \over {d_{1}(b_{1}-d_{2})}} \alpha.
\end{equation}
Moreover, (\ref{eq:E}) and (\ref{eq:O}) become, respectively,
\begin{equation}\label{eq:Efinale}
E = \displaystyle {{b_{1}(d_{1}-d_{2})(a_{1}-a_{2})} \over {d_{1}(b_{1}-d_{2})(c_{1}-c_{2})}} \alpha
\end{equation}
\begin{equation}\label{eq:Ofinale}
O = \displaystyle{{(a_{1}-a_{2})[c_{2}d_{1}(b_{1}-d_{2})+c_{1}d_{2}(d_{1}-b_{1})]} \over {a_{2}d_{1}(b_{1}-d_{2})(c_{1}-c_{2})}} \alpha.
\end{equation}
We observe that $ \theta,E,O \in \mathbf{C}-\{0\} $ (for $ O $ see condition (\ref{eq:M'invertibile})). \\
By using (\ref{eq:I}) with (\ref{eq:Bfinale}), (\ref{eq:Cfinale}), (\ref{eq:Dfinale}) we get the second resolubility condition:
\begin{equation}\label{eq:2condition}
b_{1}b_{2}(a_{2}-a_{1})(c_{1}d_{2}-c_{2}d_{1})+d_{1}d_{2}(c_{1}-c_{2})(a_{1}b_{2}-a_{2}b_{1}) = 0.
\end{equation}
Finally, from (\ref{eq:I}) we get
\begin{equation}\label{eq:Ifinale}
I = \displaystyle{{d_{1}(a_{2}b_{1}-a_{1}b_{2})[d_{2}(c_{2}-c_{1})-b_{1}c_{2}]+b_{1}c_{1}[a_{2}d_{1}(b_{1}-b_{2})+b_{2}d_{2}(a_{2}-a_{1})]} \over {a_{2}d_{1}^{2}(b_{1}-d_{2})(c_{1}-c_{2})}} \alpha.
\end{equation}
If also $ I $ is different from $ 0 $, that is if
\begin{equation}\label{eq:M''invertibile}
d_{1}(a_{2}b_{1}-a_{1}b_{2})[d_{2}(c_{2}-c_{1})-b_{1}c_{2}]+b_{1}c_{1}[a_{2}d_{1}(b_{1}-b_{2})+b_{2}d_{2}(a_{2}-a_{1})] \not=0
\end{equation}
is satisfied, then the matrix $ M'' $ in (\ref{eq:M''}) is invertible. \\
Thus, $ \Omega_{\mathbf{P}^2}^{1}(\log \mathcal{D}_{1}) \cong \Omega_{\mathbf{P}^2}^{1}(\log \mathcal{D}_{2}) $ if and only if (\ref{eq:1condition}), (\ref{eq:2condition}), (\ref{eq:M'invertibile}), (\ref{eq:M''invertibile}) are verified.
Let's start by solving (\ref{eq:1condition}) and (\ref{eq:2condition}): if we fix $ a_{1},b_{1},a_{2},b_{2},c_{1},c_{2} $, for the remaining coefficients we get
\begin{equation}\label{eq:d_{1}}
d_{1} = \displaystyle{{b_{1}b_{2}c_{1}(a_{2}-a_{1})} \over {a_{1}b_{2}(a_{2}-c_{1})+a_{2}b_{1}(c_{1}-a_{1})}}
\end{equation}
\begin{equation}\label{eq:d_{2}}
d_{2} = \displaystyle{{b_{1}b_{2}c_{2}(a_{2}-a_{1})} \over {a_{1}b_{2}(a_{2}-c_{2})+a_{2}b_{1}(c_{2}-a_{1})}}.
\end{equation}
So, the the matrix associated to $ C'_{i} $, $ i \in \{1,2\} $, is of the form
\begin{equation}\label{eq:matricesD2}
\pmatrix{ 
c_{i} & 0 & 0 \cr
0 & \displaystyle{{b_{1}b_{2}c_{i}(a_{2}-a_{1})} \over {a_{1}b_{2}(a_{2}-c_{i})+a_{2}b_{1}(c_{i}-a_{1})}} & 0 \cr
0 & 0 & -1 \cr
}. 
\end{equation}
If we put $ t_{i} = \displaystyle {{a_{2}(a_{1}-c_{i})} \over {a_{1}(c_{i}-a_{2})}} $, then (\ref{eq:matricesD2}) becomes
\begin{equation}\label{eq:matricesintermedieD2}
\pmatrix{ 
\displaystyle{{a_{1}a_{2}(1+t_{i})} \over {a_{2}+t_{i}a_{1}}} & 0 & 0 \cr
0 & \displaystyle{{b_{1}b_{2}(1+t_{i})} \over {b_{2}+t_{i}b_{1}}} & 0 \cr
0 & 0 & -1 \cr
}. 
\end{equation}
(\ref{eq:matricesintermedieD2}) is equivalent (up to scalar multiplication) to the diagonal matrix
\begin{equation}\label{eq:solutionsD2}
C(t_{i}) = (A^{-1}+t_{i}B^{-1})^{-1}
\end{equation}
where $ A = diag(a_{1},b_{1},-1),\, B = diag(a_{2},b_{2},-1) $ are the matrices associated to $ C_{1} $ and $ C_{2} $. As we can see in \cite{GKZ}, if $ C \subset \mathbf{P}^2 $ is a smooth conic represented by a matrix $ M $, then the \emph{dual conic} $ C^{\vee} \subset {\mathbf{P}^2}^{\vee} $ is defined by the inverse matrix $ M^{-1} $. So, the four tangent lines to $ C_{1} $ and $ C_{2} $ become, in $ ({\mathbf{P}^2})^{\vee} $, the base points for the pencil of conics generated by $ A^{-1} $ and $ B^{-1} $. Coming back to $ {\mathbf{P}^2}$, these points correspond to the four tangent lines to $ C'_{1} $ and $ C'_{2} $, as desired. We remark that this implication is true when the elements of (\ref{eq:matricesintermedieD2}) satisfy the open condition (\ref{eq:M''invertibile}) ((\ref{eq:M'invertibile}) is always true). \\
Viceversa, let assume that $ \mathcal{D}_{1} $ and $ \mathcal{D}_{2} $ have the same tangent lines, we want to prove that the corresponding logarithmic bundles are isomorphic. Since $ C_{1} $ and $ C_{2} $ have normal crossings, we can suppose that they are represented, respectively, by $ A = diag(a_{1},b_{1},-1) $ and $ B = diag(a_{2},b_{2},-1) $, as above. If the two pairs of conics have the same tangent lines, $ {C'_{1}}^{\vee} $ and $ {C'_{2}}^{\vee} $ live in the pencil generated by $ {C_{1}}^{\vee} $ and $ {C_{2}}^{\vee} $, that is $ C'_{1} $ and $ C'_{2} $ are represented by matrices as in (\ref{eq:solutionsD2}) (or, equivalently, as in (\ref{eq:matricesintermedieD2})). Clearly these matrices satisfy (\ref{eq:1condition}), (\ref{eq:2condition}), (\ref{eq:M'invertibile}). If also (\ref{eq:M''invertibile}) holds, then $ \Omega_{\mathbf{P}^2}^{1}(\log \mathcal{D}_{1}) \cong \Omega_{\mathbf{P}^2}^{1}(\log \mathcal{D}_{2}) $, which concludes the proof.
\end{proof}

\begin{remark}
The previous theorem asserts that the isomorphism class of $ \Omega_{\mathbf{P}^2}^{1}(\log \mathcal{D}) $ is determined by the four tangent lines to $ \mathcal{D} $. It is confirmed also by dimensional computations: indeed, as we can see in (\ref{eq:M(-1,3)}), $ \Omega_{\mathbf{P}^2}^{1}(\log \mathcal{D})_{norm} $ lives in the $ 8 $-dimensional moduli space $ \mathbf{M}_{\mathbf{P}^2}(-1,3) $ and four lines in $ \mathbf{P}^2 $ are determined exactly by $ 8 $ parameters. In particular, all the vector bundles in $ \mathbf{M}_{\mathbf{P}^2}(-1,3) $ are logarithmic.
\end{remark}

\begin{remark}
In the proof of theorem $ 5.18 $ we use the fact that isomorphic logarithmic bundles correspond to the same set of points $ \{E,F,G\} $. This condition is necessary but not sufficient. Indeed, if $ \mathcal{D}_{1} $ and $ \mathcal{D}_{2} $ are made of conics in the same pencil, then the zero locus of the section of $ \Omega_{\mathbf{P}^2}^{1}(\log \mathcal{D}_{1}) $ coincides with the zero locus of the section of $ \Omega_{\mathbf{P}^2}^{1}(\log \mathcal{D}_{2}) $ but these bundles are not isomorphic, since $ \mathcal{D}_{1} $ and $ \mathcal{D}_{2} $ have not the same tangent lines.
\end{remark}

\chapter{Many higher degree hypersurfaces in the projective space}
\label{Many higher degree hypersurfaces in the projective space}

\section{A generalization of conic arrangements case}

The arguments used for arrangements of at least $ 9 $ conics can be extended in a natural way to families with a \emph{large} number of higher degree smooth codimension $ 1 $ objects with normal crossings on the complex projective space.\\
Let $ \mathcal{D} = \{D_{1}, \ldots, D_{\ell}\} $ be an arrangement of smooth hypersurfaces of the same degree $ d \geq 2 $ with normal crossings on $ \mathbf{P}^n $, with $ n \geq 2 $. If $ n = 2 $ each $ D_{i} $ reduces to a curve and we can assume that $ d \geq 3 $. We denote by $ \Omega_{\mathbf{P}^n}^{1}(\log \mathcal{D}) $ the corresponding logarithmic bundle.

\begin{remark}
According to~\cite{H}, the Veronese map of degree $ d $, that is  
$$ \nu_{d}: \mathbf{P}^n \longrightarrow \mathbf{P}^{N-1} $$
$$ \, [x_{0}, \ldots, x_{n}] \longmapsto [\ldots \, x^{I} \ldots] $$
where $ N = {{n+d} \choose {d}} $ and $ x^{I} $ ranges over all monomials of degree $ d $ in $ x_{0},\ldots, x_{n} $, allows us to associate to $ \mathcal{D} $ a hyperplane arrangement $ \mathcal{H} = \{H_{1}, \ldots, H_{\ell}\} $ on $ \mathbf{P}^{N-1} $. As in the case of conics, we want to recover the elements of $ D $ through this link with hyperplanes.
\end{remark}

\begin{remark}
Theorem~\ref{T:Ancona} implies that $ \Omega_{\mathbf{P}^n}^{1}(\log \mathcal{D}) $ admits this short exact sequence
$$ 0 \longrightarrow \mathcal{O}_{\mathbf{P}^n}(-d)^{\ell} \longrightarrow \mathcal{O}_{\mathbf{P}^n}(-1)^{n+1} \oplus \mathcal{O}_{\mathbf{P}^n}^{\ell -1} \longrightarrow \Omega_{\mathbf{P}^n}^{1}(\log \mathcal{D})  \longrightarrow 0; $$
we get that $ {\Omega_{\mathbf{P}^n}^{1}(\log \mathcal{D})}^{\vee} $ is stable if and only if $ c_{1}({\Omega_{\mathbf{P}^n}^{1}(\log \mathcal{D})}^{\vee}) < 0 $ which is equivalent to say that 
\begin{equation}\label{eq:fibstab}
\ell \geq {{n+1} \over {d}}.
\end{equation} 
Thus, as in remark~\ref{r53}, if $ \ell $ satisfies the previous inequality then $ {\Omega_{\mathbf{P}^n}^{1}(\log \mathcal{D})}^{\vee} $ has no global sections on $ \mathbf{P}^n $ different from the zero one. For this reason we are allowed to introduce the notion of \emph{unstable} hypersurface as in definition~\ref{d52}. 
\end{remark}

\begin{definition}
Let $ D \subset \mathbf{P}^n $ be a hypersurface of degree $ d $. \\
$ D $ is said to be \emph{unstable} for $ \Omega_{\mathbf{P}^n}^{1}(\log \mathcal{D}) $ if 
\begin{equation}\label{eq:ipersupinstabile}
H^{0}(D, {\Omega_{\mathbf{P}^n}^{1}(\log \mathcal{D})}^{\vee}_{|_{D}}) \not= \{0\}.
\end{equation}
\end{definition}

We have the following:

\begin{theorem}\label{t64}
Let $ \mathcal{D} = \{D_{1}, \ldots, D_{\ell}\} $ be an arrangement of smooth hypersurfaces of degree $ d \geq 2 $ with normal crossings on $ \mathbf{P}^n $, with $ n \geq 2 $. Let $ \mathcal{H} = \{H_{1}, \ldots, H_{\ell}\} $ be the corresponding hyperplane arrangement in $ \mathbf{P}^{N-1} $, with $ N = {{n+d} \choose {d}} $. Assume that:
\begin{itemize}
\item[$1)$] $ \ell \geq N+3 $;
\item[$2)$] $ \mathcal{H} $ is an arrangement of hyperplanes with normal crossings;
\item[$3)$] $ H_{1}, \ldots, H_{\ell} $ don't osculate a rational normal curve of degree $ N-1 $ in $ \mathbf{P}^{N-1} $.
\end{itemize}
Then $ \mathcal{D} $ is equal to the following set:
$$ \{ D \subset \mathbf{P}^n \, \emph{smooth irreducible hypersurface of degree d} \,|\, D\, \emph{satisfies (\ref{eq:ipersupinstabile})}\}.  $$
\end{theorem}

\begin{proof} We can apply the same double-inclusion argument of theorem~\ref{t54}. In particular, the first part follows from the \emph{residue exact sequence} (\ref{eq:resgeneral}) for $ \Omega_{\mathbf{P}^n}^{1}(\log \mathcal{D}) $ and the second part is a consequence of the short exact sequence given in proposition $ 2.11 $ of \cite{Do}
$$ 0 \longrightarrow  \mathcal{N}^{\vee}_{\nu_{d}(\mathbf{P}^n), \, \mathbf{P}^{N-1}} \longrightarrow \Omega_{\mathbf{P}^n}^{1}(\log \mathcal{H})_{|_{\nu_{d}(\mathbf{P}^n)}} \longrightarrow \Omega_{\nu_{d}(\mathbf{P}^n)}^{1}(\log \mathcal{H} \cap \nu_{d}(\mathbf{P}^n)) \longrightarrow 0. $$
\end{proof}

\begin{remark}
It's not hard to see that the first hypothesis of theorem~\ref{t64} implies condition (\ref{eq:fibstab}).
\end{remark}

\begin{corollary}
If $ \ell \geq {{n+d} \choose {d}}+3 $ then the map
$$ \mathcal{D} \longmapsto \Omega_{\mathbf{P}^n}^{1}(\log \mathcal{D}) $$
is generically injective.
\end{corollary}

\chapter{Arrangements of quadrics in the projective space}
\label{ch:arrangements of quadrics in the projective space}

\section{One quadric}

Theorem $ 5.7 $ that holds for one smooth conic in $ \mathbf{P}^2 $ can be generalized to $ n \geq 3 $. In this sense we have the following result:
\begin{theorem}\label{T:1quadric}
Let $ Q \subset \mathbf{P}^n $ be a smooth quadric and let $ \mathcal{D} = \{Q\} $. Then 
\begin{equation}\label{eq:isom1quad}
\Omega_{\mathbf{P}^n}^{1}(\log \mathcal{D}) \cong \mathbf{TP}^n(-2).
\end{equation} 
\end{theorem}

\begin{proof}
The isomorphism (\ref{eq:isom1quad}) is a direct consequence of the exact sequence
$$ 0 \longrightarrow \mathcal{O}_{\mathbf{P}^n}(-2) \longrightarrow \mathcal{O}_{\mathbf{P}^2}(-1)^{n+1} \longrightarrow \Omega_{\mathbf{P}^n}^{1}(\log \mathcal{D})  \longrightarrow 0. $$
\end{proof}

\begin{remark}
Theorem \ref{T:1quadric} points out that $ \mathcal{D} $ and an arrangement $ \mathcal{H} $ made of $ n+2 $ hyperplanes with normal crossings on $ \mathbf{P}^n $ behave in a similar way. Indeed, as we can see in proposition~\ref{P:3.14}, $ \Omega_{\mathbf{P}^n}^{1}(\log \mathcal{D}) \cong \Omega_{\mathbf{P}^n}^{1}(\log \mathcal{H})(-1) $.
\end{remark}

\section{Pairs of quadrics}
The case of pair of conics in $ \mathbf{P}^2 $ can be extended to the case of pairs of quadrics in $ \mathbf{P}^n $, with $ n \geq 3 $. 
\begin{theorem}
Let $ Q_{1} $ and $ Q_{2} $ be smooth quadrics in $ \mathbf{P}^n $. \\
The following facts are equivalent:
\begin{itemize}
\item[$1)$] $ \mathcal{D} = \{ Q_{1},Q_{2}\} $ is an arrangement with normal crossings in $ \mathbf{P}^n $, that is $ Q_{1} \cap Q_{2} $ is a smooth codimension two subvariety;
\item[$2)$] in the pencil of quadrics generated by $ Q_{1} $ and $ Q_{2} $ there are $ n+1 $ distinct singular quadrics with singular points $ \{v_{0}, \ldots , v_{n}\} $.
\end{itemize}
\end{theorem}
\begin{proof}
Let assume that $ 2) $ holds. Then, by using the same arguments of the proof of theorem $ 5.10 $ we get that there's a basis $ \mathcal{B}'' $ of $ \mathbf{C}^{n+1} $ made of representative vectors of the points $ \{v_{0}, \ldots , v_{n}\} $ with respect to which $ Q_{1} $ and $ Q_{2} $ have equations
\begin{equation}\label{eq:Q1}
\lambda_{0}x_{0}^{2} + \ldots + \lambda_{n}x_{n}^{2} = 0
\end{equation}
\begin{equation}\label{eq:Q2}
x_{0}^{2} + \ldots + x_{n}^{2} = 0.
\end{equation}
where $ \lambda_{i} \in \mathbf{C} - \{0\} $, $ \lambda_{i} \not= \lambda_{j} $, are the opposite values of the parameters giving the singular quadrics in the pencil of $ Q_{1} $ and $ Q_{2} $.
Now, let $ P = (\overline{x}_{0}, \ldots, \overline{x}_{n}) \in Q_{1} \cap Q_{2} $, let say $ \overline{x}_{0} \not = 0 $; we want to prove that $ Q_{1} $ and $ Q_{2} $ have normal crossings at P. It's not hard to see that the tangent spaces $ T_{P}Q_{1} $ and $ T_{P}Q_{2} $ are given by, respectively, 
\begin{equation}\label{eq:TQ1}
\lambda_{0}x_{0}\overline{x}_{0} + \ldots + \lambda_{n}x_{n}\overline{x}_{n} = 0
\end{equation}
\begin{equation}\label{eq:TQ2}
x_{0}\overline{x}_{0} + \ldots + x_{n}\overline{x}_{n} = 0.
\end{equation}
The matrix associated to the system of equations (\ref{eq:TQ1}) and (\ref{eq:TQ2}) is 
$$ \pmatrix{ 
\lambda_{0}\overline{x}_{0} & \ldots & \lambda_{n}\overline{x}_{n} \cr 
\overline{x}_{0} & \ldots & \overline{x}_{n} \cr 
} $$
and it is clearly of rank $ 2 $. Indeed, we can always find $ i \in \{1, \ldots, n\} $ such that 
$$ \left | \, \matrix{
\lambda_{0}\overline{x}_{0} & \lambda_{i}\overline{x}_{i} \cr
\overline{x}_{0} & \overline{x}_{i} \cr } \right | = \overline{x}_{0}\overline{x}_{i}(\lambda_{0}-\lambda_{i}) \not= 0. $$
If this is not the case, since by hypothesis $ \lambda_{0} \not= \lambda_{i} $ and $ \overline{x}_{0} \not= 0 $, we get $ \overline{x}_{i} = 0 $ for all $ i \in \{1, \ldots, n\} $, that is $ P = (\overline{x}_{0}, 0, \ldots, 0) $. This leads to a contradiction because the coordinates of $ P $ have to satisfy (\ref{eq:Q1}) and (\ref{eq:Q2}). 
So $ dim_{\mathbf{C}} (T_{P}Q_{1} \cap T_{P}Q_{2}) = n+1-2 = n-1 $, that is $ Q_{1} $ and $ Q_{2} $ have normal crossings in P. \\
Now we want to prove $ 2) $ from $ 1) $. Let $ A = \{a_{ij}\} $ and $ B = \{b_{ij}\} $ be symmetric elements of $ GL(n+1,\mathbf{C}) $ representing $ Q_{1} $ and $ Q_{2} $ and let $ A+tB $ be the matrix of a generic quadric in the pencil generated by $ Q_{1} $ and $ Q_{2} $. Suppose that $ 2) $ is not true, that is the equation
\begin{equation}\label{eq:quadrichesingolari}
det(A+tB) = 0 
\end{equation}
has a root with multiplicity at least $ 2 $. Let consider the multilinear map
$$ \phi : \mathbf{C}^{2} \times \mathbf{C}^{n+1} \times \mathbf{C}^{n+1} \longrightarrow \mathbf{C} $$
defined by
$$ \phi ((t_{0}, t_{1}),(x_{0}, \ldots, x_{n}), (y_{0}, \ldots, y_{n})) = \displaystyle \sum_{i,j = 0}^n y_{i}(t_{0}a_{ij} + t_{1}b_{ij})x_{j}. $$
$ \phi $ corresponds to the 3-dimensional matrix $ \mathcal{A} = \{\alpha_{i_{0} i_{1} i_{2}}\}$ of format $ (2,n+1,n+1) $ with $ A $ in the first vertical slice and $ B $ in the second one. \\

\begin{figure}[h]
    \centering
\includegraphics[width=50mm]{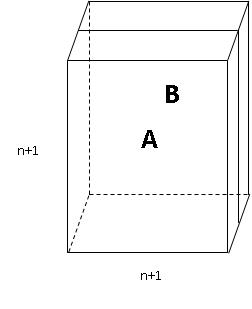}
    \caption{The $3$-dimensional matrix $\mathcal{A}$}
\end{figure}

By using \emph{Schl$\ddot a $fli's method} (for more details see \cite{GKZ}, \cite{O1} and \cite{O2}), we can associate to $ \mathcal{A} $ a family of $ (n+1) \times (n+1) $ ordinary matrices $ \overline{\mathcal{A}}(\tau_{0},\tau_{1}) $ with entries
$$ \overline{\mathcal{A}}(\tau_{0},\tau_{1})_{i_{1} i_{2}} =  \alpha_{0 i_{1} i_{2}} \tau_{0} + \alpha_{1 i_{1} i_{2}} \tau_{1} = a_{i_{1} i_{2}} \tau_{0} + b_{i_{1} i_{2}} \tau_{1} $$
that is we have a linear operator 
$$ \overline{\mathcal{A}} : \mathbf{C}^{2} \longrightarrow \mathbf{C}^{n+1} \times \mathbf{C}^{n+1}. $$
Since the 3-dimensional \emph{hyperdeterminant} of format $ (2,n+1,n+1) $ is non trivial, we can associate to $ \mathcal{A} $ a polynomial function defined by
$$ F_{\mathcal{A}}(\tau_{0},\tau_{1}) = det\,\overline{\mathcal{A}}(\tau_{0},\tau_{1}) $$ 
which is a homogenous form in $ \tau_{0},\tau_{1} $ of degree $ n+1 $. Denote by $ \Delta(F_{\mathcal{A}}) $ the \emph{discriminant} of $ F_{\mathcal{A}} $: it is a polynomial in $ \alpha_{i_{0} i_{1} i_{2}} $ of degree
$$ deg(\Delta(F_{\mathcal{A}})) = 2n(n+1) $$ 
and it is divisible by the hyperdeterminant $ Det(\mathcal{A}), $ which has the same degree. So there exists $ k \in \mathbf{C} - \{0\} $ such that
\begin{equation}\label{eq:discr-hyper}
\Delta(F_{\mathcal{A}}) = k Det(\mathcal{A}).
\end{equation}
We remark that 
$$ \Delta(F_{\mathcal{A}}) = (-1)^{{(n+1)n} \over {2}} \gamma_{n+1}^{2n} \displaystyle \prod _{i < j} (t_{i}-t_{j})^{2} $$
where $ \gamma_{n+1} $ is the leading coefficient of (\ref{eq:quadrichesingolari}) and $ t_{0}, \ldots, t_{n} $ are its roots. By assumption, $ \Delta(F_{\mathcal{A}}) $ reduces to $ 0 $ and, since (\ref{eq:discr-hyper}) holds, the same is true for $ Det(\mathcal{A}) $. This implies that the 3-dimensional matrix $ \mathcal{A} $ is degenerate, i.e. there exists a non zero $ (\overline{t}_{0}, \overline{t}_{1}) \otimes (\overline{x}_{0}, \ldots, \overline{x}_{n}) \otimes (\overline{y}_{0}, \ldots, \overline{y}_{n}) \in \mathbf{C}^{2} \otimes \mathbf{C}^{n+1} \otimes \mathbf{C}^{n+1} $ such that 
$$ \phi(\mathbf{C}^{2},(\overline{x}_{0}, \ldots, \overline{x}_{n}),(\overline{y}_{0}, \ldots, \overline{y}_{n})) = 0 $$
$$ \phi((\overline{t}_{0}, \overline{t}_{1}), \mathbf{C}^{n+1},(\overline{y}_{0}, \ldots, \overline{y}_{n})) = 0 $$
$$ \phi((\overline{t}_{0}, \overline{t}_{1}), (\overline{x}_{0}, \ldots, \overline{x}_{n}),\mathbf{C}^{n+1}) = 0. $$
After a linear change of coordinates we may assume that $ (\overline{t}_{0}, \overline{t}_{1}) = (1,0)$, $ (\overline{x}_{0}, \ldots, \overline{x}_{n}) = (1,0, \ldots, 0) $ and $ (\overline{y}_{0}, \ldots, \overline{y}_{n}) = (1,0, \ldots, 0) $.
We immediately get that $ a_{0 j} = 0 $ for all $ j \in \{0, \ldots, n\} $ and $ b_{00} = 0 $, i.e. $ \mathcal{A} $ becomes as in figure $ 7.2 $. Thus $ Q_{1} $ is not smooth and $ \mathcal{D} = \{Q_{1},Q_{2}\} $ has not normal crossings at the point $ P = (1,0, \ldots, 0) $, which is a contradiction.
\begin{figure}[h]
    \centering
    \includegraphics[width=50mm]{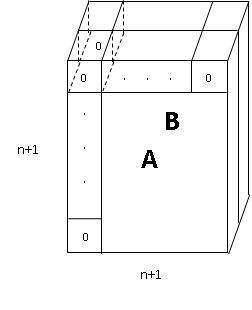}
    \caption{$\mathcal{A}$ after the linear change of coordinates}
\end{figure}

\end{proof}

\begin{remark}
If in the pencil of quadrics generated by $ Q_{1} $ and $ Q_{2} $ there are $ n+1 $ distinct singular quadrics $ \overline{Q}_{0}, \ldots, \overline{Q}_{n} $, then $ rank (\overline{Q}_{i}) = n $, that is $ \overline{Q}_{i} $ is a cone, say with vertex $ v_{i} $, for all $ i \in \{0, \ldots,n\} $. 
Indeed, let $ A $ and $ B $ be symmetric matrices in $ GL(n+1,\mathbf{C}) $ representing $ Q_{1} $ and $ Q_{2} $. By hypothesis, the equation
$$ det(AB^{-1}+tI_{n+1}) = 0 $$
has $ n+1 $ non zero distinct solutions, which implies that the matrix $ - AB^{-1} $ has $ n+1 $ distinct eigenvalues $ t_{0}, \ldots, t_{n} $. So
$$ rank(\overline{Q}_{i}) = rank(A+t_{i}B) = rank(- AB^{-1} - t_{i}I_{n+1}) = n. $$
In particular, the singular point $ v_{i} $ of $ \overline{Q}_{i} $ is an eigenvector of $ - AB^{-1} $ corresponding to the eigenvalue $ t_{i}. $
\end{remark}

\begin{remark}
Let $ Q_{1} $ and $ Q_{2} $ be smooth quadrics with normal crossings and let $ \{v_{0}, \ldots , v_{n}\} $ as in theorem $ 7.3 $. Then the matrices associated to $ Q_{1} $ and $ Q_{2} $ with respect to a basis of $ \mathbf{C}^{n+1} $ made of representative vectors of the points $ \{v_{0}, \ldots , v_{n}\} $ are of the form $ diag(a_{0}, a_{1}, \ldots, a_{n-1},-1) $ and $ diag(b_{0},b_{1}, \ldots, b_{n-1}, -1) $, where $ a_{i},b_{i} \in \mathbf{C} - \{0\} $, $ a_{i} \not= b_{i} $ and $ \displaystyle{{a_{i}} \over {a_{j}}} \not= {{b_{i}} \over {b_{j}}} $, for all $ i,j \in \{0, \ldots, n-1 \} $ (we remark that our quadrics are smooth and in the pencil generated by them there are $ n+1 $ singular quadrics). 
\end{remark}

\begin{remark}
Let $ \Omega_{\mathbf{P}^n}^{1}(\log \mathcal{D}) $ the logarithmic bundle attached to an arrangement of smooth quadrics with normal crossings $ \mathcal{D} = \{Q_{1},Q_{2}\} $. Theorem \ref{T:Ancona} asserts that it is a rank $ n $ vector bundle over $ \mathbf{P}^n $ such that
\begin{equation}\label{eq:Ancona2quadriche}
0 \longrightarrow \mathcal{O}_{\mathbf{P}^n}(-2)^{2} \longrightarrow \mathcal{O}_{\mathbf{P}^n}(-1)^{n+1} \oplus \mathcal{O}_{\mathbf{P}^n} \longrightarrow \Omega_{\mathbf{P}^n}^{1}(\log \mathcal{D})  \longrightarrow 0
\end{equation}
is exact. So the Chern polynomial $ p(t) $ of $ \Omega_{\mathbf{P}^n}^{1}(\log \mathcal{D}) $ is obtained by truncating to degree $ n $ the expression
$$ {{(1-t)^{n+1}} \over {(1-2t)^{2}}} = \left \lbrack \displaystyle \sum_{i = 0}^{n+1} {{n+1} \choose {i }} (-1)^{i} t^{i} \right \rbrack \left \lbrack \displaystyle \sum_{k \geq 0} 2^{k} (k+1) t^{k} \right \rbrack  $$ that is 
$$ p_{\Omega_{\mathbf{P}^n}^{1}(\log \mathcal{D})}(t) = \displaystyle \sum_{m = 0}^{n} \left \lbrack \displaystyle \sum_{h = 0}^{m} {{n+1} \choose {h}}(-1)^{h}2^{m-h}(m-h+1) \right \rbrack t^{m}. $$
In particular the n-th Chern class of $ \Omega_{\mathbf{P}^n}^{1}(\log \mathcal{D}) $ is
$$ c_{n}(\Omega_{\mathbf{P}^n}^{1}(\log \mathcal{D})) = \displaystyle \sum_{h = 0}^{n} {{n+1} \choose {h}}(-1)^{h}2^{n-h}(n-h+1) = $$
$$ \,\quad\quad\quad\quad\quad\quad = (n+1) \displaystyle \sum_{h = 0}^{n} {n \choose {h}} (-1)^{h}2^{n-h} = n+1. $$
Moreover, (\ref{eq:Ancona2quadriche}) tells us that 
$$ H^{0}(\mathbf{P}^n,\Omega_{\mathbf{P}^n}^{1}(\log \mathcal{D})) = \mathbf{C}. $$
Thus $ \Omega_{\mathbf{P}^n}^{1}(\log \mathcal{D}) $ has one non-zero section with $ n+1 $ zeroes. 
\end{remark}

The arguments used for the proof of proposition $ 5.16 $ part 1) naturally extend to the case of $ \mathbf{P}^{n} $, $ n \geq 3 $. The key idea is that the singular points $ \{v_{0}, \ldots, v_{n}\} $ of the cones $ \overline{Q}_{0}, \ldots, \overline{Q}_{n} $ are the eigenvectors of $ AB^{-1} $. In this sense we have the following:

\begin{proposition}
Let $ \mathcal{D} = \{Q_{1},Q_{2}\} $ be an arrangement of smooth quadrics in $ \mathbf{P}^{n} $ with normal crossings. Then $ \{v_{0}, \ldots, v_{n}\} $ is the zero locus of the non-zero section of $ \Omega_{\mathbf{P}^n}^{1}(\log \mathcal{D}) $.
\end{proposition}

In order to state and prove the main result concerning pairs of quadrics in the complex projective space, we recall that, given a smooth quadric $ Q \subset \mathbf{P}^{n} $, the \emph{dual quadric} of $ Q $ is $ Q^{\vee}\subset (\mathbf{P}^{n})^{\vee} $ given by the tangent hyperplanes to $ Q $. In particular, if $ Q $ is represented by a symmetric $ n \times n $ matrix $ G $, then $ Q^{\vee} $ is associated to $ G^{-1} $. The set of tangent hyperplanes to two smooth quadrics with normal crossings in $ \mathbf{P}^{n} $, $ Q_{1} $ and $ Q_{2} $, is the base locus of the pencil of quadrics in $ (\mathbf{P}^{n})^{\vee} $ generated by $ Q_{1}^{\vee} $ and $ Q_{2}^{\vee} $, that is $ Q_{1}^{\vee} \cap Q_{2}^{\vee} $. \\
We have the following:

\begin{theorem}\label{T:pairquad}
Let $ \mathcal{D}_{1} = \{Q_{1},Q_{2}\} $ and $ \mathcal{D}_{2} = \{Q'_{1},Q'_{2}\} $ be arrangements of smooth quadrics with normal crossings in $ \mathbf{P}^{n} $, with $ n \geq 3 $. Then 
\begin{equation}\label{eq:fibriso}
\Omega_{\mathbf{P}^n}^{1}(\log \mathcal{D}_{1}) \cong \Omega_{\mathbf{P}^n}^{1}(\log \mathcal{D}_{2})
\end{equation}
if and only if $ \mathcal{D}_{1} $ and $ \mathcal{D}_{2} $ have the same $ n+1 $ tangent hyperplanes, that is
$$ Q_{1}^{\vee} \cap Q_{2}^{\vee} =  Q_{1}^{'\vee} \cap Q_{2}^{'\vee}. $$
\end{theorem}

\begin{proof}
Suppose that (\ref{eq:fibriso}) holds. Then, by using proposition $ 7.7 $ and remark $ 7.5 $, we can assume that $ Q_{1} $, $ Q_{2} $, $ Q'_{1} $, $ Q'_{2} $ have equations, respectively,
$$ a_{0}x_{0}^2+a_{1}x_{1}^2+ \ldots + a_{n-1}x_{n-1}^2-x_{n}^2 = 0 $$
$$ b_{0}x_{0}^2+b_{1}x_{1}^2+ \ldots + b_{n-1}x_{n-1}^2-x_{n}^2 = 0 $$
$$ c_{0}x_{0}^2+c_{1}x_{1}^2+ \ldots + c_{n-1}x_{n-1}^2-x_{n}^2 = 0 $$
$$ d_{0}x_{0}^2+d_{1}x_{1}^2+ \ldots + d_{n-1}x_{n-1}^2-x_{n}^2 = 0 $$
where the coefficients satisfy the properties stated in remark $ 7.5 $.
Saying that the two logarithmic bundles are isomorphic is equivalent to the fact that we can find two invertible matrices 
\begin{equation}\label{eq:M'quadrics}
M' = \pmatrix{ 
\alpha & \beta \cr 
\gamma & \delta \cr
}
\end{equation}
\begin{equation}\label{eq:M''quadrics}
M'' = \pmatrix{ 
E_{1,1} & \ldots & E_{1,n+1} & f_{1} \cr
E_{2,1} & \ldots & E_{2,n+1} & f_{2} \cr
\vdots & {} & \vdots & \vdots \cr
E_{n+1,1} & \ldots & E_{n+1,n+1} & f_{n+1} \cr
0 & \ldots & 0 & \theta \cr
}
\end{equation}
with $ \alpha, \beta, \gamma, \delta, E_{i,j}, \theta \in \mathbf{C} $ and $ f_{j} = \displaystyle \sum_{j = 0}^{n} f^{i}_{j}x_{i} $ complex linear forms, such that the following diagram commutes:
$$ \mathcal{O}_{\mathbf{P}^n}(1)^{n+1} \oplus \mathcal{O}_{\mathbf{P}^n} \buildrel \rm N_{1} \over \longrightarrow \mathcal{O}_{\mathbf{P}^n}(2)^{2} $$ 
$$ \,\,\,\,\,\,\,\,\,\,M'' \downarrow  \,\,\,\,\,\,\,\,\,\,\,\,\,\,\,\,\,\,\,\,\,\,\,\,\,\,\,\,\,\ \downarrow M'  $$
$$ \mathcal{O}_{\mathbf{P}^n}(1)^{n+1} \oplus \mathcal{O}_{\mathbf{P}^n} \buildrel \rm N_{2} \over \longrightarrow \mathcal{O}_{\mathbf{P}^n}(2)^{2} $$ 
where
$$ N_{1} = \pmatrix{ 
2a_{0}x_{0} & \dots  & 2a_{n-1}x_{n-1} & -2x_{n} &  a_{0}x_{0}^2+ \ldots + a_{n-1}x_{n-1}^2-x_{n}^2 \cr 
2b_{0}x_{0} & \ldots & 2b_{n-1}x_{n-1} & -2x_{n} & 0 \cr
} $$
$$ N_{2} = \pmatrix{ 
2c_{0}x_{0} & \dots  & 2c_{n-1}x_{n-1} & -2x_{n} &  c_{0}x_{0}^2+ \ldots + c_{n-1}x_{n-1}^2-x_{n}^2 \cr 
2d_{0}x_{0} & \ldots & 2d_{n-1}x_{n-1} & -2x_{n} & 0 \cr
}. $$
Let's equate the entries of the $ 2 \times (n+2) $ matrices $ M'N_{1} $ and $ N_{2}M'' $; we immediately get that if $ i \not= j $ then $ E_{i,j} = 0 $. The first non trivial conditions are:
\begin{equation}\label{eq:X11X21}
E_{1,1} = \displaystyle{{a_{0}} \over {c_{0}}} \alpha + {{b_{0}} \over {c_{0}}} \beta = \displaystyle{{a_{0}} \over {d_{0}}} \gamma + \displaystyle{{b_{0}} \over {d_{0}}} \delta
\end{equation}
\begin{equation}\label{eq:X12X22}
E_{2,2} = \displaystyle{{a_{1}} \over {c_{1}}} \alpha + {{b_{1}} \over {c_{1}}} \beta = \displaystyle{{a_{1}} \over {d_{1}}} \gamma + \displaystyle{{b_{1}} \over {d_{1}}} \delta
\end{equation}
\begin{equation}\label{eq:X13X23}
E_{3,3} = \displaystyle{{a_{2}} \over {c_{2}}} \alpha + {{b_{2}} \over {c_{2}}} \beta = \displaystyle{{a_{2}} \over {d_{2}}} \gamma + \displaystyle{{b_{2}} \over {d_{2}}} \delta
\end{equation}
$$ \vdots $$
\begin{equation}\label{eq:X1nX2n}
E_{n,n} = \displaystyle{{a_{n-1}} \over {c_{n-1}}} \alpha + {{b_{n-1}} \over {c_{n-1}}} \beta = \displaystyle{{a_{n-1}} \over {d_{n-1}}} \gamma + \displaystyle{{b_{n-1}} \over {d_{n-1}}} \delta
\end{equation}
\begin{equation}\label{eq:X1n+1X2n+1}
E_{n+1,n+1} = \alpha + \beta = \gamma + \delta
\end{equation}
For the moment we don't care about equations from (\ref{eq:X12X22}) to (\ref{eq:X1nX2n}), we will use them afterwards in order to get $ n-1 $ resolubility conditions for our system.
From equations (\ref{eq:X11X21}) and (\ref{eq:X1n+1X2n+1}) we find
\begin{equation}\label{eq:deltaintermedio}
\delta = \displaystyle {{a_{0}(d_{0}-c_{0})} \over {c_{0}(b_{0}-a_{0})}} \alpha + \displaystyle {{(b_{0}d_{0}-a_{0}c_{0})} \over {c_{0}(b_{0}-a_{0})}} \beta
\end{equation}
and
\begin{equation}\label{eq:gammaintermedio}
\gamma = \displaystyle {{(b_{0}c_{0}-a_{0}d_{0})} \over {c_{0}(b_{0}-a_{0})}} \alpha + \displaystyle {{b_{0}(c_{0}-d_{0})} \over {c_{0}(b_{0}-a_{0})}} \beta.
\end{equation}
If we consider the last column of $ M'N_{1} $ and $ N_{2}M'' $ we obtain, for the first entry
\begin{equation}\label{eq:X1n+2,0}
\alpha a_{0} = 2 c_{0} f^{0}_{1} + \theta c_{0}
\end{equation}
\begin{equation}\label{eq:X1n+2,1}
\alpha a_{1} = 2 c_{1} f^{1}_{2} + \theta c_{1}
\end{equation}
\begin{equation}\label{eq:X1n+2,2}
\alpha a_{2} = 2 c_{2} f^{2}_{3} + \theta c_{2}
\end{equation}
$$ \vdots $$
\begin{equation}\label{eq:X1n+2,n-1}
\alpha a_{n-1} = 2 c_{n-1} f^{n-1}_{n} + \theta c_{n-1}
\end{equation}
\begin{equation}\label{eq:X1n+2,n}
\alpha = 2f^{n}_{n+1} + \theta
\end{equation}
\begin{equation}\label{eq:X1n+2,01}
c_{0}f^{1}_{1}+c_{1}f^{0}_{2} = 0
\end{equation}
$$ \vdots $$
\begin{equation}\label{eq:X1n+2,0n-1}
c_{0}f^{n-1}_{1}+c_{n-1}f^{0}_{n} = 0
\end{equation}
\begin{equation}\label{eq:X1n+2,0n}
c_{0}f^{n}_{1}-f^{0}_{n+1} = 0
\end{equation}
\begin{equation}\label{eq:X1n+2,12}
c_{1}f^{2}_{2}+c_{2}f^{1}_{3} = 0
\end{equation}
$$ \vdots $$
\begin{equation}\label{eq:X1n+2,1n-1}
c_{1}f^{n-1}_{2}+c_{n-1}f^{1}_{n} = 0
\end{equation}
\begin{equation}\label{eq:X1n+2,1n}
c_{1}f^{n}_{2}-f^{1}_{n+1} = 0
\end{equation}
$$ \vdots $$
\begin{equation}\label{eq:X1n+2,n-1n}
c_{n-1}f^{n}_{n}-f^{n-1}_{n+1} = 0
\end{equation}
and for the second entry
\begin{equation}\label{eq:X2n+2,0}
\gamma a_{0} = 2 d_{0} f^{0}_{1}
\end{equation}
\begin{equation}\label{eq:X2n+2,1}
\gamma a_{1} = 2 d_{1} f^{1}_{2}
\end{equation}
\begin{equation}\label{eq:X2n+2,2}
\gamma a_{2} = 2 d_{2} f^{2}_{3}
\end{equation}
$$ \vdots $$
\begin{equation}\label{eq:X2n+2,n-1}
\gamma a_{n-1} = 2 d_{n-1} f^{n-1}_{n}
\end{equation}
\begin{equation}\label{eq:X2n+2,n}
\gamma = 2f^{n}_{n+1}
\end{equation}
\begin{equation}\label{eq:X2n+2,01}
d_{0}f^{1}_{1}+d_{1}f^{0}_{2} = 0
\end{equation}
$$ \vdots $$
\begin{equation}\label{eq:X2n+2,0n-1}
d_{0}f^{n-1}_{1}+d_{n-1}f^{0}_{n} = 0
\end{equation}
\begin{equation}\label{eq:X2n+2,0n}
d_{0}f^{n}_{1}-f^{0}_{n+1} = 0
\end{equation}
\begin{equation}\label{eq:X2n+2,12}
d_{1}f^{2}_{2}+d_{2}f^{1}_{3} = 0
\end{equation}
$$ \vdots $$
\begin{equation}\label{eq:X2n+2,1n-1}
d_{1}f^{n-1}_{2}+d_{n-1}f^{1}_{n} = 0
\end{equation}
\begin{equation}\label{eq:X2n+2,1n}
d_{1}f^{n}_{2}-f^{1}_{n+1} = 0
\end{equation}
$$ \vdots $$
\begin{equation}\label{eq:X2n+2,n-1n}
d_{n-1}f^{n}_{n}-f^{n-1}_{n+1} = 0.
\end{equation}
By using equations (\ref{eq:X1n+2,01}),$ \ldots $, (\ref{eq:X1n+2,n-1n}) and (\ref{eq:X2n+2,01}), $ \ldots $, (\ref{eq:X2n+2,n-1n}) (that is conditions coming from coefficients of $ x_{i}x_{j} $ with $ i \not= j $) and remembering the properties of $ c_{0},\ldots,c_{n-1},d_{0},\ldots,d_{n-1} $, we get that if $ j-i \not= 1 $ then $ f^{i}_{j} = 0 $. So each linear form reduces to $ f_{j} = f^{j-1}_{j}x_{j -1} $. In order to determine these coefficients we consider equations from (\ref{eq:X1n+2,0}) to (\ref{eq:X1n+2,n}) (actually these are $ n+1 $ relations) and we get
\begin{equation}\label{eq:f01intermedia1}
f^{0}_{1} =  \displaystyle {{\alpha} \over {2}} \left ({{a_{0}} \over {c_{0}}} - {{a_{1}} \over {c_{1}}} \right ) + f^{1}_{2} 
\end{equation}
\begin{equation}\label{eq:f23intermedia1}
f^{2}_{3} =  \displaystyle {{\alpha} \over {2}} \left ({{a_{2}} \over {c_{2}}} - {{a_{1}} \over {c_{1}}} \right ) + f^{1}_{2} 
\end{equation}
$$ \vdots $$
\begin{equation}\label{eq:fn-1nintermedia1}
f^{n-1}_{n} =  \displaystyle {{\alpha} \over {2}} \left ({{a_{n-1}} \over {c_{n-1}}} - {{a_{1}} \over {c_{1}}} \right ) + f^{1}_{2} 
\end{equation}
\begin{equation}\label{eq:fnn+1intermedia1}
f^{n}_{n+1} =  \displaystyle {{\alpha} \over {2}} \left (1 - {{a_{1}} \over {c_{1}}} \right ) + f^{1}_{2} 
\end{equation}
\begin{equation}\label{eq:thetaintermedia}
\theta =  \displaystyle {{a_{1}} \over {c_{1}}} \alpha - 2 f^{1}_{2}.
\end{equation}
Moreover equations from (\ref{eq:X2n+2,0}) to (\ref{eq:X2n+2,n}) tell us that
\begin{equation}\label{eq:f01intermedia2}
f^{0}_{1} =  \displaystyle {{a_{0}} \over {2d_{0}}} \gamma 
\end{equation}
\begin{equation}\label{eq:f12intermedia}
f^{1}_{2} =  \displaystyle {{a_{1}} \over {2d_{1}}} \gamma
\end{equation}
\begin{equation}\label{eq:f23intermedia2}
f^{2}_{3} =  \displaystyle {{a_{2}} \over {2d_{2}}} \gamma
\end{equation}
$$ \vdots $$
\begin{equation}\label{eq:fn-1nintermedia2}
f^{n-1}_{n} =  \displaystyle {{a_{n-1}} \over {2d_{n-1}}} \gamma
\end{equation}
\begin{equation}\label{eq:fnn+1intermedia2}
f^{n}_{n+1} =  \displaystyle {{\gamma} \over {2}} 
\end{equation}
By using (\ref{eq:fnn+1intermedia1}), (\ref{eq:fnn+1intermedia2}), (\ref{eq:f12intermedia}) and the expression for $ \gamma $ given by (\ref{eq:gammaintermedio}) we get
\begin{equation}\label{eq:beta}
\beta = \displaystyle {{a_{1}b_{0}c_{0}(c_{1}-d_{1})+a_{0}c_{1}d_{1}(d_{0}-c_{0})+a_{0}a_{1}(c_{0}d_{1}-c_{1}d_{0})} \over {b_{0}c_{1}(c_{0}-d_{0})(d_{1}-a_{1})}} \alpha.
\end{equation}
This implies that (\ref{eq:gammaintermedio}) and (\ref{eq:deltaintermedio}) become, respectively,
\begin{equation}\label{eq:gamma}
\gamma = \displaystyle {{d_{1}(a_{1}-c_{1})} \over {c_{1}(a_{1}-d_{1})}} \alpha
\end{equation}
and
\begin{equation}\label{eq:delta}
\delta = \displaystyle {{a_{1}d_{1}(b_{0}d_{0}-a_{0}c_{0})+a_{1}c_{1}d_{0}(a_{0}-b_{0})-a_{0}c_{1}d_{1}(d_{0}-c_{0})} \over {b_{0}c_{1}(c_{0}-d_{0})(a_{1}-d_{1})}} \alpha.
\end{equation}
We remark that (\ref{eq:beta}), (\ref{eq:gamma}) and (\ref{eq:delta}) are the same formulas that we found in the case of two conics in $ \mathbf{P}^2 $ (indeed only the coefficients of $ x_{0}^2 $ and $ x_{1}^2 $ of the four quadrics are involved).
As in that case, if we choose $ \alpha \in \mathbf{C} - \{0\} $ and the following property holds
\begin{equation}\label{eq:M'quadricsinvertibile}
a_{1}(c_{1}d_{0}-c_{0}d_{1})+c_{1}d_{1}(c_{0}-d_{0}) \not= 0
\end{equation}
then the matrix $ M' $ introduced in (\ref{eq:M'quadrics}) is invertible.
Moreover, if we consider (\ref{eq:gamma}) together with equations from (\ref{eq:thetaintermedia}) to (\ref{eq:fnn+1intermedia2}) we get 
\begin{equation}\label{eq:fj-1j}
f^{j-1}_{j} = \displaystyle {{a_{j-1}d_{1}(a_{1}-c_{1})} \over {2d_{j-1}c_{1}(a_{1}-d_{1})}} \alpha \,\,\, \forall \, j \in \{1, \ldots, n\}
\end{equation}
\begin{equation}\label{eq:fnn+1}
f^{n}_{n+1} = \displaystyle {{d_{1}(a_{1}-c_{1})} \over {2c_{1}(a_{1}-d_{1})}} \alpha
\end{equation}
\begin{equation}\label{eq:theta}
\theta = \displaystyle {{a_{1}(c_{1}-d_{1})} \over {c_{1}(a_{1}-d_{1})}} \alpha.
\end{equation}
Since we have two ways to compute $ f^{0}_{1},f^{2}_{3},\ldots,f^{n-1}_{n}$, we get $ n-1 $ resolubility conditions for our system involving the coefficients of the quadrics. In this sense, let consider equations (\ref{eq:f01intermedia1}) and (\ref{eq:fj-1j}) for $ j=1$, $ j=2 $, we have
\begin{equation}\label{eq:condition1}
a_{0}a_{1}(c_{1}d_{0}-c_{0}d_{1})+a_{1}c_{0}d_{0}(d_{1}-c_{1})+a_{0}c_{1}d_{1}(c_{0}-d_{0}) = 0.
\end{equation}
Similarly, if we consider equations (\ref{eq:f23intermedia1}) with (\ref{eq:fj-1j}) for $ j=2$, $ j=3 $ we get 
\begin{equation}\label{eq:condition2}
a_{1}a_{2}(c_{1}d_{2}-c_{2}d_{1})+a_{1}c_{2}d_{2}(d_{1}-c_{1})+a_{2}c_{1}d_{1}(c_{2}-d_{2}) = 0
\end{equation}
and so on.
Finally, equations (\ref{eq:fn-1nintermedia1}) and (\ref{eq:fj-1j}) for $ j=2$, $ j=n $ give us
\begin{equation}\label{eq:conditionn-1}
a_{1}a_{n-1}(c_{1}d_{n-1}-c_{n-1}d_{1})+a_{1}c_{n-1}d_{n-1}(d_{1}-c_{1})+a_{n-1}c_{1}d_{1}(c_{n-1}-d_{n-1}) = 0
\end{equation}
(in order to get all these relations it suffices to find (\ref{eq:condition1}) and then to change the index $ 0 $ with $ j \in \{2,3,\ldots,n-1\} $).
Now, let come back to equations from (\ref{eq:X12X22}) to (\ref{eq:X1nX2n}). If we substitute in these equations final expressions for $ \beta,\gamma,\delta $ we get, respectively,
\begin{equation}\label{eq:conditionn}
a_{1}b_{1}(b_{0}-a_{0})(c_{0}d_{1}-c_{1}d_{0})+c_{1}d_{1}(c_{0}-d_{0})(a_{0}b_{1}-a_{1}b_{0}) = 0
\end{equation}
\begin{equation}\label{eq:conditionn+1}
a_{1}a_{2}b_{0}(c_{0}-d_{0})(c_{2}d_{1}-c_{1}d_{2})+a_{1}b_{0}b_{2}(d_{1}-c_{1})(c_{2}d_{0}-c_{0}d_{2})+
\end{equation}
$$ \quad\quad +a_{0}a_{1}b_{2}(c_{2}-d_{2})(c_{1}d_{0}-c_{0}d_{1})+ c_{1}d_{1}(c_{2}-d_{2})(c_{0}-d_{0})(a_{0}b_{2}-a_{2}b_{0}) = 0 $$
\begin{equation}\label{eq:conditionn+2}
a_{1}a_{3}b_{0}(c_{0}-d_{0})(c_{3}d_{1}-c_{1}d_{3})+a_{1}b_{0}b_{3}(d_{1}-c_{1})(c_{3}d_{0}-c_{0}d_{3})+
\end{equation}
$$ \quad\quad +a_{0}a_{1}b_{3}(c_{3}-d_{3})(c_{1}d_{0}-c_{0}d_{1})+ c_{1}d_{1}(c_{3}-d_{3})(c_{0}-d_{0})(a_{0}b_{3}-a_{3}b_{0}) = 0 $$
$$ \vdots $$
\vfill\eject
\begin{equation}\label{eq:condition2n-2}
a_{1}a_{n-1}b_{0}(c_{0}-d_{0})(c_{n-1}d_{1}-c_{1}d_{n-1})+
\end{equation}
$$ + a_{1}b_{0}b_{n-1}(d_{1}-c_{1})(c_{n-1}d_{0}-c_{0}d_{n-1})+ $$
$$ +a_{0}a_{1}b_{n-1}(c_{n-1}-d_{n-1})(c_{1}d_{0}-c_{0}d_{1})+ $$
$$ \quad\quad\quad\quad\quad +c_{1}d_{1}(c_{n-1}-d_{n-1})(c_{0}-d_{0})(a_{0}b_{n-1}-a_{n-1}b_{0}) = 0. $$
We remark that (\ref{eq:conditionn}) is exactly condition (\ref{eq:2condition}) that we got in the case of conics. Moreover, in order to get the other relations it suffices to find (\ref{eq:conditionn+1}) and then to write $ 3,4,\ldots,n-1 $ instead of the index $ 2 $ (namely $ \beta,\gamma $ and $ \delta $ depend only on the coefficients of the quadrics indexed by $ 0 $ and $ 1 $). 
By using equations from (\ref{eq:X12X22}) to (\ref{eq:X1nX2n}) we can find expressions for $ E_{i,i} $:
\begin{equation}\label{eq:E11}
E_{1,1} = \displaystyle {{a_{1}(c_{1}-d_{1})(a_{0}-b_{0})} \over {c_{1}(a_{1}-d_{1})(c_{0}-d_{0})}} \alpha
\end{equation}
\leftline{$ E_{i,i} = $} 
\begin{equation}\label{eq:Eii}
= \displaystyle {{c_{1}(a_{i}b_{0}-a_{0}b_{i-1})[d_{1}(d_{0}-c_{0})-a_{1}d_{0}]+a_{1}c_{0}[b_{0}c_{1}(a_{i-1}-b_{i-1})+b_{i-1}d_{1}(b_{0}-a_{0})]} \over {b_{0}c_{1}c_{i -1}(c_{0}-d_{0})(a_{1}-d_{1})}} \alpha
\end{equation}
for all $ i \in \{2,3,\ldots,n \} $ and
\begin{equation}\label{eq:En+1n+1}
E_{n+1,n+1}= \displaystyle {{(a_{0}-b_{0})[a_{1}(c_{1}d_{0}-c_{0}d_{1})+c_{1}d_{1}(c_{0}-d_{0})]} \over {b_{0}c_{1}(c_{0}-d_{0})(a_{1}-d_{1})}} \alpha.
\end{equation}
If $ n=2 $ these coefficients reduce to $ E,I,O $ of the proof of theorem $ 5.18 $. 
Thus the matrix $ M'' $ introduced in (\ref{eq:M''quadrics}) is invertible if and only if, fixed $ \alpha \in \mathbf{C} - \{0\} $, besides (\ref{eq:M'quadricsinvertibile}), for all $ i \in \{2,3,\ldots,n\} $ hold
\begin{equation}\label{eq:M''quadricsinvertibile}
c_{1}(a_{i-1}b_{0}-a_{0}b_{i-1})[d_{1}(d_{0}-c_{0})-a_{1}d_{0}]+a_{1}c_{0}[b_{0}c_{1}(a_{i-1}-b_{i-1})+b_{i-1}d_{1}(b_{0}-a_{0})] \not= 0.
\end{equation}
In this way we have that $ \Omega_{\mathbf{P}^n}^{1}(\log \mathcal{D}_{1}) \cong \Omega_{\mathbf{P}^n}^{1}(\log \mathcal{D}_{2}) $ if and only if relations from (\ref{eq:condition1}) to (\ref{eq:condition2n-2}) hold (they are $ 2n-2 $), with the open conditions (\ref{eq:M'quadricsinvertibile}) and (\ref{eq:M''quadricsinvertibile}) (they are n). So, let fix $ a_{0}, \ldots, a_{n-1}, b_{0}, \ldots, b_{n-1}, c_{0}, d_{0} $ and let consider equations (\ref{eq:condition1}) and (\ref{eq:conditionn}): with the same computations of the case of conics we immediately get that
\begin{equation}\label{eq:c1}
c_{1} = \displaystyle{{a_{1}b_{1}c_{0}(b_{0}-a_{0})} \over {a_{0}b_{1}(b_{0}-c_{0})+a_{1}b_{0}(c_{0}-a_{0})}}
\end{equation}
\begin{equation}\label{eq:d1}
d_{1} = \displaystyle{{a_{1}b_{1}d_{0}(b_{0}-a_{0})} \over {a_{0}b_{1}(b_{0}-d_{0})+a_{1}b_{0}(d_{0}-a_{0})}}.
\end{equation}
Now, let substitute these expressions of $ c_{1} $ and $ d_{1} $ in (\ref{eq:condition2}) and (\ref{eq:conditionn+1}): these equations become, respectively, quadratic and linear with respect to $ c_{2} $ and $ d_{2} $. With some computations we get
\begin{equation}\label{eq:c2}
c_{2} = \displaystyle{{a_{2}b_{2}c_{0}(b_{0}-a_{0})} \over {a_{0}b_{2}(b_{0}-c_{0})+a_{2}b_{0}(c_{0}-a_{0})}}
\end{equation}
\begin{equation}\label{eq:d2}
d_{2} = \displaystyle{{a_{2}b_{2}d_{0}(b_{0}-a_{0})} \over {a_{0}b_{2}(b_{0}-d_{0})+a_{2}b_{0}(d_{0}-a_{0})}}
\end{equation}
and so on. Finally the pair of equations (\ref{eq:conditionn-1}) and (\ref{eq:condition2n-2}) gives us
\begin{equation}\label{eq:cn-1}
c_{n-1} = \displaystyle{{a_{n-1}b_{n-1}c_{0}(b_{0}-a_{0})} \over {a_{0}b_{n-1}(b_{0}-c_{0})+a_{n-1}b_{0}(c_{0}-a_{0})}}
\end{equation}
\begin{equation}\label{eq:dn-1}
d_{n-1} = \displaystyle{{a_{n-1}b_{n-1}d_{0}(b_{0}-a_{0})} \over {a_{0}b_{n-1}(b_{0}-d_{0})+a_{n-1}b_{0}(d_{0}-a_{0})}}.
\end{equation}
Thus the matrices representing $ Q'_{1} $ and $ Q'_{2} $ are, respectively,
$$ \pmatrix{ 
c_{0} & 0 & {} & \ldots & 0 \cr
0 & \displaystyle{{a_{1}b_{1}c_{0}(b_{0}-a_{0})} \over {a_{0}b_{1}(b_{0}-c_{0})+a_{1}b_{0}(c_{0}-a_{0})}} & 0 & \ldots & 0 \cr
\vdots & {} & \ddots & {} & \vdots \cr
0 & \ldots & 0 & \displaystyle{{a_{n-1}b_{n-1}c_{0}(b_{0}-a_{0})} \over {a_{0}b_{n-1}(b_{0}-c_{0})+a_{n-1}b_{0}(c_{0}-a_{0})}} & 0 \cr 
0 & {} & \ldots & 0 & -1 \cr
} $$
$$ \pmatrix{ 
d_{0} & 0 & {} & \ldots & 0 \cr
0 & \displaystyle{{a_{1}b_{1}d_{0}(b_{0}-a_{0})} \over {a_{0}b_{1}(b_{0}-d_{0})+a_{1}b_{0}(d_{0}-a_{0})}} & 0 & \ldots & 0 \cr
\vdots & {} & \ddots & {} & \vdots \cr
0 & \ldots & 0 & \displaystyle{{a_{n-1}b_{n-1}d_{0}(b_{0}-a_{0})} \over {a_{0}b_{n-1}(b_{0}-d_{0})+a_{n-1}b_{0}(d_{0}-a_{0})}} & 0 \cr 
0 & {} & \ldots & 0 & -1 \cr
}. $$
If we introduce $ t = \displaystyle {{b_{0}(a_{0}-c_{0})} \over {a_{0}(c_{0}-b_{0})}} $ and $ s = \displaystyle {{b_{0}(a_{0}-d_{0})} \over {a_{0}(d_{0}-b_{0})}} $ the previous matrices become, respectively, 
\begin{equation}\label{eq:matricefinale1}
\pmatrix{ 
\displaystyle{{a_{0}b_{0}(1+t)} \over {b_{0}+ta_{0}}} & 0 & {} & \ldots & 0 \cr
0 & \displaystyle{{a_{1}b_{1}(1+t)} \over {b_{1}+ta_{1}}} & 0 & \ldots & 0 \cr
\vdots & {} & \ddots & {} & \vdots \cr
0 & \ldots & 0 & \displaystyle{{a_{n-1}b_{n-1}(1+t)} \over {b_{n-1}+ta_{n-1}}} & 0 \cr 
0 & {} & \ldots & 0 & -1 \cr
} 
\end{equation}
\begin{equation}\label{eq:matricefinale2}
\pmatrix{ 
\displaystyle{{a_{0}b_{0}(1+s)} \over {b_{0}+sa_{0}}} & 0 & {} & \ldots & 0 \cr
0 & \displaystyle{{a_{1}b_{1}(1+s)} \over {b_{1}+sa_{1}}} & 0 & \ldots & 0 \cr
\vdots & {} & \ddots & {} & \vdots \cr
0 & \ldots & 0 & \displaystyle{{a_{n-1}b_{n-1}(1+s)} \over {b_{n-1}+sa_{n-1}}} & 0 \cr
0 & {} & \ldots & 0 & -1 \cr
}
\end{equation}
that is they are equivalent (up to scalar multiplication) to $ (A^{-1}+tB^{-1})^{-1} $ and $ (A^{-1}+sB^{-1})^{-1} $, where $ A $ and $ B $ are the diagonal matrices associated to $ Q_{1} $ and $ Q_{2} $. By applying the same \emph{duality} argument of the case of pairs of conics we get that $ \mathcal{D}_{1} $ and $ \mathcal{D}_{2} $ have the same tangent hyperplanes. We remark that the entries of (\ref{eq:matricefinale1}) and (\ref{eq:matricefinale2}) have to verify the open conditions given in (\ref{eq:M''quadricsinvertibile}) ((\ref{eq:M'quadricsinvertibile}) is always satisfied). 
For the other direction, we can work as in the last part of the proof of theorem $ 5.18 $.
\end{proof}

\chapter{Arrangements of lines and conics in the projective plane}
\label{ch:arrangements of lines and conics in the projective plane}

\section{The case of a conic and a line}
Let $ \mathcal{D} = \{r, C\} $ be an arrangement with normal crossings in $ \mathbf{P}^2 $ made of a line $ r $ and a smooth conic $ C $. Assume that $ r = \{\alpha x_{0}+\beta x_1 + \gamma x_{2} = 0\}$ and $ C = \{f = 0 \} $ where $ f = \displaystyle{ \sum_{i, j = 0}^{2}a_{i j} x_{i} x_{j}} $. As usual, we denote by $ \Omega_{\mathbf{P}^2}^{1}(\log \mathcal{D}) $ the associated logarithmic bundle. 

\begin{figure}[h]
    \centering
    \includegraphics[width=50mm]{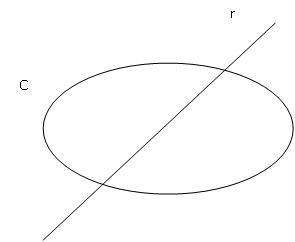}
    \caption{A line and a smooth conic with normal crossings}
\end{figure}

Theorem~\ref{T:Ancona} assures us that $ \Omega_{\mathbf{P}^2}^{1}(\log \mathcal{D}) $ admits the short exact sequence
\begin{equation}\label{eq:resline-conic}
0 \longrightarrow \mathcal{O}_{\mathbf{P}^2}(-1) \oplus \mathcal{O}_{\mathbf{P}^2}(-2)  \buildrel \rm M \over \longrightarrow \mathcal{O}_{\mathbf{P}^2}(-1)^{3} \oplus \mathcal{O}_{\mathbf{P}^2} \longrightarrow \Omega_{\mathbf{P}^2}^{1}(\log \mathcal{D}) \longrightarrow 0
\end{equation}
where $ M $ is the $ 4 \times 2 $ matrix given by
$$ M = \pmatrix{ 
\alpha & 2 \displaystyle{ \sum_{j = 0}^{2}a_{0 j} x_{j} }  \cr 
\beta & 2 \displaystyle{ \sum_{j = 0}^{2}a_{1 j} x_{j} }  \cr 
\gamma & 2 \displaystyle{ \sum_{j = 0}^{2}a_{2 j} x_{j} }  \cr 
\alpha x_{0}+\beta x_1 + \gamma x_{2} & 0 \cr
}. $$
We remark that (\ref{eq:resline-conic}) is not a minimal resolution since $ \alpha, \beta, \gamma $ are three elements of degree $ 0 $ such that $ (\alpha, \beta, \gamma) \in \mathbf{C}^3-(0,0,0) $. Without loss of generality we can assume that $ r = \{x_{0} = 0\}$, in particular we have 
$$  M = \pmatrix{ 
1 & 2 \displaystyle{ \sum_{j = 0}^{2}a_{0 j} x_{j} }  \cr 
0 & 2 \displaystyle{ \sum_{j = 0}^{2}a_{1 j} x_{j} }  \cr 
0 & 2 \displaystyle{ \sum_{j = 0}^{2}a_{2 j} x_{j} }  \cr 
x_{0} & 0 \cr
}.  $$
By applying Gaussian elimination to $ M $ we obtain the minimal resolution for $ \Omega_{\mathbf{P}^2}^{1}(\log \mathcal{D}) $:
\begin{equation}\label{eq:minresline-conic}
0 \longrightarrow \mathcal{O}_{\mathbf{P}^2}(-2)  \buildrel \rm \overline{M} \over \longrightarrow \mathcal{O}_{\mathbf{P}^2}(-1)^{2} \oplus \mathcal{O}_{\mathbf{P}^2} \longrightarrow \Omega_{\mathbf{P}^2}^{1}(\log \mathcal{D}) \longrightarrow 0
\end{equation}
with 
$$ \overline{M} = \pmatrix{ 
2 \displaystyle{ \sum_{j = 0}^{2}a_{1 j} x_{j} }  \cr 
2 \displaystyle{ \sum_{j = 0}^{2}a_{2 j} x_{j} }  \cr 
-2 \displaystyle{ \sum_{j = 0}^{2}a_{0 j} x_{j} x_{0} } \cr
}. $$

We observe that the entries of $ \overline{M} $ are, respectively, two linear forms ($ \partial_{x_{1}} f $ and $ \partial_{x_{2}} f $) and one quadratic form (the product of $ - x_{0} $ with $ \partial_{x_{0}} f $).
From (\ref{eq:minresline-conic}) we get that $ c_{1}(\Omega_{\mathbf{P}^2}^{1}(\log \mathcal{D})) = 0 $ and $ c_{2}(\Omega_{\mathbf{P}^2}^{1}(\log \mathcal{D})) = 1 $. Moreover theorem~\ref{T:3.5} implies that $ \Omega_{\mathbf{P}^2}^{1}(\log \mathcal{D}) $ is a semistable vector bundle over $ \mathbf{P}^2 $. As we can see also in~\cite{Wykno}, the following holds:

\begin{theorem}\label{T:parM201}
Let $ \mathbf{M}_{\mathbf{P}^2}^{ss}(0,1) $ be the family of semistable rank-$ 2 $ vector bundles $ E $  over $ \mathbf{P}^2 $ with minimal resolution 
$$ 0 \longrightarrow \mathcal{O}_{\mathbf{P}^2}(-2)  \buildrel \rm ^{t}\pmatrix{ f_{1} & f_{2} & f_{3} \cr} \over \longrightarrow \mathcal{O}_{\mathbf{P}^2}(-1)^{2} \oplus \mathcal{O}_{\mathbf{P}^2} \longrightarrow E \longrightarrow 0 $$
where $ f_{1}, f_{2} $ are linear forms and $ f_{3} $ is a quadratic form. Then the map
$$ \mathbf{M}_{\mathbf{P}^2}^{ss}(0,1) \buildrel \rm \pi_{2} \over \longrightarrow \mathbf{P}^2 $$
$$ E \longmapsto \{f_{1} = 0\} \cap \{f_{2} = 0\} $$
is an isomorphism.
\end{theorem}
\begin{proof}
Let $ E $ and $ E' $ be two elements of  $ \mathbf{M}_{\mathbf{P}^2}^{ss}(0,1) $. We want to prove that the intersection point of $ f_{1} $ and $ f_{2} $ coincides with the one of $ f'_{1} $ and $ f'_{2} $ if and only if $ E \cong E' $. If the intersection point is the same, without loss of generality we can assume that $ f_{1} = f'_{1} = x_{0} $ and $ f_{2} = f'_{2} = x_{1} $. Moreover, $ E_{x} $ and $ E'_{x} $ are the cokernels of two rank-$ 1 $ maps for all $ x \in\mathbf{P}^2 $, in particular if $ x = [0,0,1] $ we get that $ f_{3} $ and $  f'_{3} $ have to contain the term $ x_{2}^{2} $. In order to prove that $ E \cong E' $ it suffices to find $ g_{1}, g_{2} $ linear forms such that the following diagram commutes:
$$ \mathcal{O}_{\mathbf{P}^2}(-2) \buildrel \rm ^{t}\pmatrix{ x_{0} & x_{1} & f_{3} \cr} \over \longrightarrow \mathcal{O}_{\mathbf{P}^2}(-1)^{2} \oplus \mathcal{O}_{\mathbf{P}^2}  $$ 
$$ 1 \downarrow  \,\,\,\,\,\,\,\,\,\,\,\,\,\,\,\,\,\,\,\,\,\,\,\,\,\,\,\,\,\,\,\,\,\,\,\,\,\,\,\ \,\,\,\,\, \,\,\,\,\, \,\,\,\,\, \,\,\,\,\, \,\,\,\,\, \downarrow L $$
$$ \mathcal{O}_{\mathbf{P}^2}(-2) \buildrel \rm ^{t}\pmatrix{ x_{0} & x_{1} & f'_{3} \cr} \over \longrightarrow \mathcal{O}_{\mathbf{P}^2}(-1)^{2} \oplus \mathcal{O}_{\mathbf{P}^2}  $$ 
where
$$  L = \pmatrix{ 1 & 0 & 0  \cr 0 & 1 & 0  \cr g_{1} & g_{2} & 1 \cr }. $$
In particular we have to solve 
\begin{equation}\label{eq:equation}
g_{1}x_{0}+g_{2}x_{1}+f_{3} = f'_{3}. 
\end{equation}
Assume that
$$ f_{3} = b_{00}x_{0}^{2}+b_{01}x_{0}x_{1}+b_{02}x_{0}x_{2}+b_{11}x_{1}^{2}+b_{12}x_{1}x_{2}+x_{2}^{2}, $$ 
$$ f'_{3} = c_{00}x_{0}^{2}+c_{01}x_{0}x_{1}+c_{02}x_{0}x_{2}+c_{11}x_{1}^{2}+c_{12}x_{1}x_{2}+x_{2}^{2}; $$
we immediately get that 
$$ g_{1} = (c_{00}-b_{00})x_{0}+(c_{01}-b_{01}-1)x_{1}+(c_{02}-b_{02})x_{2} $$
$$ g_{2} = x_{0}+(c_{11}-b_{11})x_{1}+(c_{12}-b_{12})x_{2} $$
solve (\ref{eq:equation}), which concludes the proof.
\end{proof}

\begin{remark}
The previous result asserts that the vector bundle $ \Omega_{\mathbf{P}^2}^{1}(\log \mathcal{D}) $ lives in a space of dimension $ 2 $, while the number of parameters associated to a line and a conic with normal crossings is $ 7 $. So in this case the map in (\ref{eq:Torellimap}) can't be injective.
\end{remark}

With the same notations of the beginning of this section we have the following:

\begin{proposition}
Let $ \mathcal{D} = \{r,C\} $ be an arrangement with normal crossings. Then the pole of the line $ r $ with respect to the conic $ C $ describes the isomorphism class of $ \Omega_{\mathbf{P}^2}^{1}(\log \mathcal{D}) $, that is $ \pi_{2}(\Omega_{\mathbf{P}^2}^{1}(\log \mathcal{D})) $ is the pole of $ r $ with respect to $ C $.
\end{proposition}
\begin{proof}
Applying Cramer's rule we get that the point in $ \mathbf{P}^2 $ satisfying
$$ \displaystyle{ \sum_{j = 0}^{2}a_{1 j} x_{j} } = \displaystyle{ \sum_{j = 0}^{2}a_{2 j} x_{j} } = 0 $$
is $ P = [ a_{12}^{2}-a_{11}a_{22}, a_{22}a_{01}-a_{02}a_{12}, a_{02}a_{11}-a_{12}a_{01} ] $. The polar line of $ P $ with respect to $ C $ is given by
$$ ( a_{12}^{2}-a_{11}a_{22}, a_{22}a_{01}-a_{02}a_{12}, a_{02}a_{11}-a_{12}a_{01} ) \pmatrix{ a_{00} & a_{01} & a_{02} \cr a_{01} & a_{11} & a_{12} \cr a_{02} & a_{12} & a_{22} \cr} \pmatrix{ x_{0} \cr x_{1} \cr x_{2} \cr } = 0  $$
which reduces to $ x_{0} = 0 $, that is to $ r $, as desired.
\end{proof}

\begin{figure}[h]
    \centering
    \includegraphics[width=55mm]{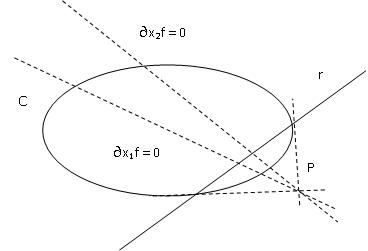}
    \caption{$ r $ is the polar line of $ P $ with respect to $ C $}
\end{figure}

We immediately get the following:

\begin{corollary}\label{C:line-conic}
Let $ \mathcal{D} = \{r, C\} $ and $ \mathcal{D'} = \{r', C'\} $ be two arrangements with normal crossings in $ \mathbf{P}^2 $ given by a line and a smooth conic. Then 
$$ \Omega_{\mathbf{P}^2}^{1}(\log \mathcal{D}) \cong \Omega_{\mathbf{P}^2}^{1}(\log \mathcal{D'}) $$
if and only if the pole of $ r $ with respect to $ C $ coincides with the pole of $ r' $ with respect to $ C' $ (see figure 8.3).
\end{corollary}

\begin{figure}[h]
    \centering
    \includegraphics[width=85mm]{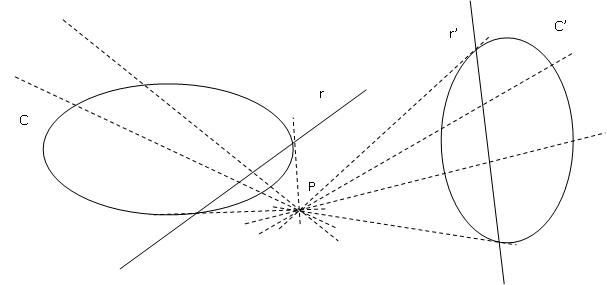}
    \caption{Line-conic arrangements with isomorphic logarithmic bundles}
\end{figure}

The previous results can be extended to $ \mathbf{P}^n $, for all $ n \geq 3 $.\\ 

Let $ \mathcal{D} = \{H, Q\} $ be an arrangement with normal crossings in $ \mathbf{P}^n $ made of a hyperlane $ H = \{\alpha_0x_{0}+ \cdots + \alpha_nx_{n} = 0 \} $ and a smooth quadric $ Q = \{f = 0\} $ with $ f = \displaystyle{ \sum_{i, j = 0}^{n}a_{i j} x_{i} x_{j}} $. \\
From theorem~\ref{T:Ancona} we get that the corresponding logarithmic bundle $ \Omega_{\mathbf{P}^n}^{1}(\log \mathcal{D}) $ admits the short exact sequence
$$ 0 \longrightarrow \mathcal{O}_{\mathbf{P}^n}(-1) \oplus \mathcal{O}_{\mathbf{P}^n}(-2)  \buildrel \rm M \over \longrightarrow \mathcal{O}_{\mathbf{P}^n}(-1)^{n+1} \oplus \mathcal{O}_{\mathbf{P}^n} \longrightarrow \Omega_{\mathbf{P}^n}^{1}(\log \mathcal{D}) \longrightarrow 0 $$
where $ M $ is the $ (n+2) \times 2 $ matrix given by
$$ M = \pmatrix{ 
\alpha_{0} & 2 \displaystyle{ \sum_{j = 0}^{n}a_{0 j} x_{j} }  \cr 
\vdots & \vdots \cr 
\alpha_{n} & 2 \displaystyle{ \sum_{j = 0}^{n}a_{n j} x_{j} }  \cr 
\alpha_0x_{0}+ \cdots + \alpha_nx_{n} & 0 \cr
}. $$
As in the case of $ n=2 $, we can assume that  $ H = \{x_{0} = 0 \} $ so that the minimal resolution of $ \Omega_{\mathbf{P}^n}^{1}(\log \mathcal{D}) $ takes the form 
\begin{equation}\label{eq:minreshyp-quad}
0 \longrightarrow \mathcal{O}_{\mathbf{P}^n}(-2)  \buildrel \rm \overline{M} \over \longrightarrow \mathcal{O}_{\mathbf{P}^n}(-1)^{n} \oplus \mathcal{O}_{\mathbf{P}^n} \longrightarrow \Omega_{\mathbf{P}^n}^{1}(\log \mathcal{D}) \longrightarrow 0
\end{equation}
with
$$ \overline{M} = \pmatrix{ 
2 \displaystyle{ \sum_{j = 0}^{n}a_{1 j} x_{j} }  \cr 
\vdots \cr 
2 \displaystyle{ \sum_{j = 0}^{n}a_{n j} x_{j} }  \cr 
-2 \displaystyle{ \sum_{j = 0}^{n}a_{0 j} x_{j} x_{0} } \cr
}. $$
Bohnhorst-Spindler criterion implies that $ \Omega_{\mathbf{P}^n}^{1}(\log \mathcal{D}) $ is not a semistable bundle over $ \mathbf{P}^n $ unless $ n=2 $. Nevertheless we can extend in a natural way theorem~\ref{T:parM201} also to $ n \geq 3 $:

\begin{theorem}
All the vector bundles $ E $ over $ \mathbf{P}^n $ having a minimal resolution
$$ 0 \longrightarrow \mathcal{O}_{\mathbf{P}^n}(-2)  \buildrel \rm ^{t}\pmatrix{ f_{1} & \cdots & f_{n+1} \cr} \over \longrightarrow \mathcal{O}_{\mathbf{P}^n}(-1)^{n} \oplus \mathcal{O}_{\mathbf{P}^n} \longrightarrow E \longrightarrow 0  $$
where $ f_{1}, \ldots, f_{n} \in H^{0}(\mathbf{P}^n, \mathcal{O}_{\mathbf{P}^n}(1)) $ and $ f_{n+1} \in H^{0}(\mathbf{P}^n, \mathcal{O}_{\mathbf{P}^n}(2)) $ are parametrized by $ \mathbf{P}^n $. In particular the correspondence is defined by
$$ E \buildrel \rm \pi_{n} \over \longmapsto \displaystyle{\bigcap_{i=1}^{n}\{f_{i}=0\}}. $$
\end{theorem}

\begin{remark}
Since $ \Omega_{\mathbf{P}^n}^{1}(\log \mathcal{D}) $ belongs to a $ n $-dimensional family and $ \mathcal{D} = \{H, Q\} $ is described by $ {n+2} \choose {2}$$ + n $ parameters, then $ \mathcal{D} $ can't be a Torelli type arrangement.
\end{remark}

\begin{proposition}
Let $ \mathcal{D}= \{H,Q\} $ be an arrangement with normal crossings. Then the pole of the hyperplane $ H $ with respect to the quadric $ Q $ describes the isomorphism class of $ \Omega_{\mathbf{P}^n}^{1}(\log \mathcal{D}) $, that is $ \pi_{n}(\Omega_{\mathbf{P}^n}^{1}(\log \mathcal{D})) $ is the pole of $ H $ with respect to $ Q $.
\end{proposition}

\begin{proof}
The solution of the linear system 
$$ a_{01}x_{0}+ \cdots + a_{1n}x_{n} = 0 $$
$$ \vdots $$
$$ a_{0n}x_{0}+ \cdots + a_{nn}x_{n} = 0 $$
is 
$$ P = [det A_{00}, - det A_{01}, det A_{02}, \ldots, (-1)^{n} det A_{0n}] $$
where $ A_{0j} $ denotes the submatrix of $ A = \{a_{ij}\} $ that we get by canceling the $0$-th row and the $j$-th column, for all $ j \in \{0, \ldots, n \} $. So the polar hyperplane of $ P $ with respect to $ Q $ is 
$$ (det A_{00}, - det A_{01}, det A_{02}, \ldots, (-1)^{n} det A_{0n}) \pmatrix{ a_{00} & \cdots & a_{0n} \cr \vdots & {} & \vdots \cr a_{0n} & \cdots & a_{nn} \cr} \pmatrix{ x_{0} \cr \vdots \cr x_{n} \cr } = 0   $$
that is $ x_{0} = 0 $.
\end{proof}

\begin{corollary}
Two arrangements with normal crossings $ \mathcal{D} = \{H, Q\} $ and $ \mathcal{D}' = \{H', Q'\} $ made of a hyperplane and a smooth quadric in $ \mathbf{P}^n $ have isomorphic logarithmic bundles if and only if the pole of $ H $ with respect to $ Q $ coincides the pole of $ H' $ with respect to $ Q' $.
\end{corollary}

\vfill\eject

\section{The case of a conic and two lines}

Let $ \mathcal{D} = \{r_{1}, r_{2}, C\} $ be an arrangement with normal crossings in $ \mathbf{P}^2 $, where $ r_{1} $ and $ r_{2} $ are lines and $ C $ is a smooth conic. We suppose that $ r_{1} = \{\alpha_{0}x_{0}+\alpha_{1}x_{1}+\alpha_{2}x_{2} = 0\} $, $ r_{2} = \{\beta_{0}x_{0}+\beta_{1}x_{1}+\beta_{2}x_{2} = 0\} $ and $ C = \{f = 0\} $ where $ f =  \displaystyle{ \sum_{i, j = 0}^{2}a_{i j} x_{i} x_{j}} $. \\

\begin{figure}[h]
    \centering
    \includegraphics[width=50mm]{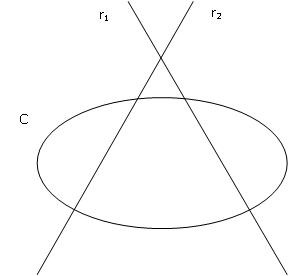}
    \caption{A smooth conic and two lines with normal crossings}
\end{figure}

In this case theorem~\ref{T:Ancona} says that $ \Omega_{\mathbf{P}^2}^{1}(\log \mathcal{D}) $ verifies
$$ 0 \longrightarrow \mathcal{O}_{\mathbf{P}^2}(-1)^{2} \oplus \mathcal{O}_{\mathbf{P}^2}(-2)  \buildrel \rm M \over \longrightarrow \mathcal{O}_{\mathbf{P}^2}(-1)^{3} \oplus \mathcal{O}_{\mathbf{P}^2}^{2} \longrightarrow \Omega_{\mathbf{P}^2}^{1}(\log \mathcal{D}) \longrightarrow 0 $$
and 
$$ M = \pmatrix{ 
\alpha_{0} & \beta_{0} & 2 \displaystyle{ \sum_{j = 0}^{2}a_{0 j} x_{j} }  \cr 
\alpha_{1} & \beta_{1} & 2 \displaystyle{ \sum_{j = 0}^{2}a_{1 j} x_{j} } \cr 
\alpha_{2} & \beta_{2} & 2 \displaystyle{ \sum_{j = 0}^{2}a_{2 j} x_{j} }  \cr 
\alpha_{0}x_{0}+\alpha_{1}x_{1}+\alpha_{2}x_{2} & 0 & 0 \cr
0 & \beta_{0}x_{0}+\beta_{1}x_{1}+\beta_{2}x_{2} & 0 \cr
}. $$
In order to simplify our computations, we can assume that $ r_{1} = \{ x_{0} = 0 \} $ and $ r_{2} = \{ x_{1} = 0 \} $; we denote with $ a_{ij} $ the coefficients of a quadratic form describing $ C $ also in this new frame. Thus the logarithmic bundle admits as minimal resolution
\begin{equation}\label{eq:rescon2lines}
0 \longrightarrow \mathcal{O}_{\mathbf{P}^2}(-2) \buildrel \rm \overline{M} \over \longrightarrow \mathcal{O}_{\mathbf{P}^2}(-1) \oplus \mathcal{O}_{\mathbf{P}^2}^{2} \longrightarrow \Omega_{\mathbf{P}^2}^{1}(\log \mathcal{D}) \longrightarrow 0
\end{equation}
where
$$  \overline{M} = \pmatrix{ 
2 \displaystyle{ \sum_{j = 0}^{2}a_{2 j} x_{j} }  \cr 
-2 \displaystyle{ \sum_{j = 0}^{2}a_{0 j} x_{j} x_{0} } \cr 
-2 \displaystyle{ \sum_{j = 0}^{2}a_{1 j} x_{j} x_{1} }  \cr 
}. $$
In particular $ c_{1}(\Omega_{\mathbf{P}^2}^{1}(\log \mathcal{D})) = 1 $ and $ c_{2}(\Omega_{\mathbf{P}^2}^{1}(\log \mathcal{D})) = 2 $. So its normalized bundle, that is $ \Omega_{\mathbf{P}^2}^{1}(\log \mathcal{D})(-1) $, belongs to $ \mathbf{M}_{\mathbf{P}^2}(-1,2) $, the moduli space of rank-$ 2 $ stable vector bundles over $ \mathbf{P}^2 $ with $ c_{1} = -1 $ and $ c_{2} = 2 $. An interesting description of this moduli space is given in the following result, of which we will see a sketch of the proof (see also~\cite{Wykno}). In order to state it we recall that $ \sigma_{2}(\nu_{2}(\mathbf{P}^2)) $ is the notation for the $2$-secant variety of the image of the quadratic Veronese map $ \nu_{2}(\mathbf{P}^2) $,~\cite{H}.

\begin{theorem}
$ \mathbf{M}_{\mathbf{P}^2}(-1,2) $ is isomorphic to $ \sigma_{2}(\nu_{2}(\mathbf{P}^2)) - \nu_{2}(\mathbf{P}^2) $, the projective space of symmetric $ 3 \times 3 $ rank-$ 2 $ matrices.
\end{theorem}
\begin{proof}
A vector bundle $ E $ lives in $ \mathbf{M}_{\mathbf{P}^2}(-1,2) $ if and only if it is endowed with a short exact sequence like
$$ 0 \longrightarrow \mathcal{O}_{\mathbf{P}^2}(-3) \buildrel \rm ^{t}\pmatrix{ f_{1} & f_{2} & f_{3} \cr} \over \longrightarrow \mathcal{O}_{\mathbf{P}^2}(-2) \oplus \mathcal{O}_{\mathbf{P}^2}^{2}(-1) \longrightarrow E \longrightarrow 0 $$ 
where $ f_{1} $ is a linear form and $ f_{2}, f_{3} $ are quadratic forms. \\
On the unique \emph{jumping line} of $ E $, which is $ \{f_{1} = 0\} $, the linear series given by $ f_{2} $ and $ f_{3} $ has two distinct double points, which we denote by $ P_{1} $ and $ P_{2} $. Then the map given by 
$$ \mathbf{M}_{\mathbf{P}^2}(-1,2) \longrightarrow \sigma_{2}(\nu_{2}(\mathbf{P}^2)) - \nu_{2}(\mathbf{P}^2) $$
$$ E \longmapsto \{P_{1},P_{2}\} $$
is an isomorphism, which concludes the proof.
\end{proof}

\begin{remark}
The previous theorem implies that an element in $ E \in \mathbf{M}_{\mathbf{P}^2}(-1,2) $ is characterized by $ 4 $ parameters, while a conic and two lines in the projective plane need $ 9 $ parameters to be described. So $ \mathcal{D} = \{r_{1},r_{2},C\} $ as above is not a Torelli arrangement. 
\end{remark}

\begin{remark}
The jumping line of $ \Omega_{\mathbf{P}^2}^{1}(\log \mathcal{D}) $ is $ \{\partial_{2}f = 0\} $ and it is the polar line with respect to $ C $ of $  r_{1} \cap r_{2} = [0,0,1] $: indeed, the equation 
$$ (0,0,1)  \pmatrix{ a_{00} & a_{01} & a_{02} \cr a_{01} & a_{11} & a_{12} \cr a_{02} & a_{12} & a_{22} \cr} \pmatrix{ x_{0} \cr x_{1} \cr x_{2} \cr } = 0 $$
reduces to $ \{\partial_{2}f = 0\} $.
Moreover, the linear series on this line is given by $ r_{1} \cup s_{2} $ and $ r_{2} \cup s_{1} $, where $ s_{2} $ is the polar line with respect to $ C $ of $ \{\partial_{2}f = 0\} \cap r_{2} = [a_{22}, 0, -a_{02}]$ and $ s_{1} $ is the polar line with respect to $ C $ of $ \{\partial_{2}f = 0\} \cap r_{1} = [0, a_{22}, -a_{2}] $, that is $  s_{2} = \{a_{22} \partial_{0}f - a_{02} \partial_{2}f = 0\} $ and $ s_{1} = \{a_{22} \partial_{1}f - a_{12} \partial_{2}f = 0\} $. The logarithmic bundle $ \Omega_{\mathbf{P}^2}^{1}(\log \mathcal{D}) $ corresponds to the two intersection points $ \{P_{1},P_{2}\} $ of $ C $ and $ \{\partial_{2}f = 0\} $ (see the figure below). 
\end{remark}

\vspace{1.0cm}

\begin{figure}[h]
    \centering
    \includegraphics[width=95mm]{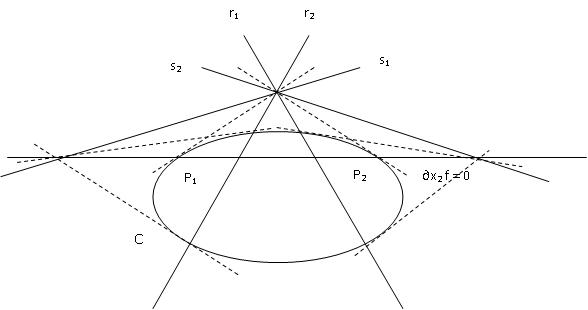}
    \caption{$ \Omega_{\mathbf{P}^2}^{1}(\log \mathcal{D}) $ corresponds to $ \{P_{1}, P_{2}\} $}
\end{figure}

\vfill\eject

\begin{corollary}\label{C:2lines-conic}
Let $ \mathcal{D} = \{r_{1}, r_{2}, C\} $ and $ \mathcal{D'} = \{r'_{1}, r'_{2}, C'\} $ be arrangements with normal crossings in $ \mathbf{P}^2 $ made of a smooth conic and two lines. \\
Then $ \Omega_{\mathbf{P}^2}^{1}(\log \mathcal{D}) \cong \Omega_{\mathbf{P}^2}^{1}(\log \mathcal{D'}) $ if and only if $ \{P_{1},P_{2}\} = \{P'_{1},P'_{2}\} $.
\end{corollary}

\begin{figure}[h]
    \centering
    \includegraphics[width=80mm]{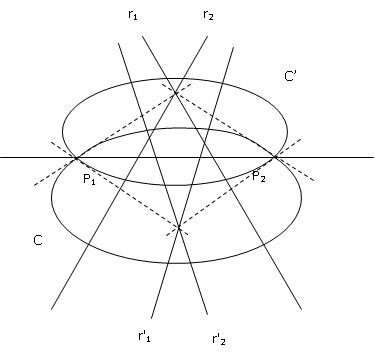}
    \caption{$ \mathcal{D} $ and $ \mathcal{D'} $ with isomorphic logarithmic bundles}
\end{figure}

\vfill\eject

\section{Some remarks about the case of a conic and three lines}

Let $ \mathcal{D} = \{r_{1},r_{2},r_{3}, C\} $ be an arrangement with normal crossings in $ \mathbf{P}^2 $ made of three lines and a smooth conic, let say $ r_{1} = \{\alpha_{0}x_{0}+\alpha_{1}x_{1}+\alpha_{2}x_{2} = 0\} $, $ r_{2} = \{\beta_{0}x_{0}+\beta_{1}x_{1}+\beta_{2}x_{2} = 0\} $,  $ r_{3} = \{\gamma_{0}x_{0}+\gamma_{1}x_{1}+\gamma_{2}x_{2} = 0\} $ and $ C = \{f = 0\} $ where $ f =  \displaystyle{ \sum_{i, j = 0}^{2}d_{i j} x_{i} x_{j}} $.

\begin{figure}[h]
    \centering
    \includegraphics[width=70mm]{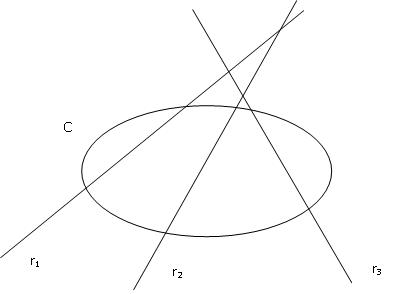}
    \caption{A smooth conic and three lines with normal crossings}
\end{figure}

Ancona's theorem implies that the corresponding logarithmic bunlde $ \Omega_{\mathbf{P}^2}^{1}(\log \mathcal{D}) $ has the following exact sequence:

$$ 0 \longrightarrow \mathcal{O}_{\mathbf{P}^2}(-1)^{3} \oplus \mathcal{O}_{\mathbf{P}^2}(-2)  \buildrel \rm M \over \longrightarrow \mathcal{O}_{\mathbf{P}^2}(-1)^{3} \oplus \mathcal{O}_{\mathbf{P}^2}^{3} \longrightarrow \Omega_{\mathbf{P}^2}^{1}(\log \mathcal{D}) \longrightarrow 0 $$
with
$$ M = \pmatrix{ 
\alpha_{0} & \beta_{0} & \gamma_{0} & 2 \displaystyle{ \sum_{j = 0}^{2}d_{0 j} x_{j} }  \cr 
\alpha_{1} & \beta_{1} & \gamma_{1} & 2 \displaystyle{ \sum_{j = 0}^{2}d_{1 j} x_{j} } \cr 
\alpha_{2} & \beta_{2} & \gamma_{2} & 2 \displaystyle{ \sum_{j = 0}^{2}d_{2 j} x_{j} }  \cr 
\alpha_{0}x_{0}+\alpha_{1}x_{1}+\alpha_{2}x_{2} & 0 & 0 & 0 \cr
0 & \beta_{0}x_{0}+\beta_{1}x_{1}+\beta_{2}x_{2} & 0 & 0 \cr
0 & 0 & \gamma_{0}x_{0}+\gamma_{1}x_{1}+\gamma_{2}x_{2} & 0 \cr
}. $$

In order to get a minimal resolution for $ \Omega_{\mathbf{P}^2}^{1}(\log \mathcal{D}) $, we can make a change of coordinates such that $ r_{1} = \{x_{0} = 0\} $, $ r_{2} = \{x_{1} = 0\} $, $ r_{3} = \{x_{2} = 0\} $ (we still denote with $ f $ a quadratic form defining $ C $). By using linear algebra computations we get the minimal resolution

\begin{equation}\label{eq:3lines-conic}
0 \longrightarrow \mathcal{O}_{\mathbf{P}^2}(-2) \buildrel \rm \overline{M} \over \longrightarrow \mathcal{O}_{\mathbf{P}^2}^{3} \longrightarrow \Omega_{\mathbf{P}^2}^{1}(\log \mathcal{D}) \longrightarrow 0
\end{equation}
where
$$ \overline{M} = \pmatrix{ 
-x_{0}\partial_{0}f  \cr 
-x_{1}\partial_{1}f \cr 
-x_{2}\partial_{2}f\cr 
}. $$

From (\ref{eq:3lines-conic}) it's not hard to prove that $ \Omega_{\mathbf{P}^2}^{1}(\log \mathcal{D}) $ is a stable bundle with Chern classes $ c_{1} = 2 $ and $ c_{2} = 4 $. In particular, its normalized bundle $ \Omega_{\mathbf{P}^2}^{1}(\log \mathcal{D})(-1) $ belongs to the moduli space $ \mathbf{M}_{\mathbf{P}^2}(0,3) $, which, according to (\ref{eq:dimmoduli}), has dimension $ 9 $. Thus, since the number of parameters associated to three lines and a conic is $ 11 $, such arrangement can't be of Torelli type. 

\begin{remark}
From a geometrical point of view, each quadratic form in $ \overline{M} $ represent the union of $ r_{i} $ with the polar line with respect to $ C $ of the intersection point of $ r_{j} $ and $ r_{k} $, for different $ i,j,k \in \{1,2,3\} $.
\end{remark}

\begin{remark}
If we do the tensor product of (\ref{eq:3lines-conic}) with $ \mathcal{O}_{\mathbf{P}^2}(-1) $ we get that $ \Omega_{\mathbf{P}^2}^{1}(\log \mathcal{D})(-1) $ has the following short exact sequence:
$$ 0 \longrightarrow \mathcal{O}_{\mathbf{P}^2}(-3) \buildrel \rm \overline{M} \over \longrightarrow \mathcal{O}_{\mathbf{P}^2}(-1)^{3} \longrightarrow \Omega_{\mathbf{P}^2}^{1}(\log \mathcal{D})(-1) \longrightarrow 0 $$
where $ \overline{M} $ is the matrix introduced above. Theorem \ref{T:Ancona} asserts that, if $ \mathcal{D'} = \{D\} $ is an arrangement made of one smooth cubic curve in $ \mathbf{P}^2 $, then the associated logarithmic bundle $ \Omega_{\mathbf{P}^2}^{1}(\log \mathcal{D'}) $ admits an exact sequence like the one for $ \Omega_{\mathbf{P}^2}^{1}(\log \mathcal{D})(-1) $, which is given by the three partial derivatives of a cubic polynomial defining $ D $. So an interesting problem is to understand if there is a smooth cubic curve in the projective plane corresponding to three lines and a smooth conic. 
\end{remark}

We have the following:

\begin{theorem}\label{T:3lines-conic}
Let $ \mathcal{D} $ be the arrangement with normal crossings given by $ \{x_{0}x_{1}x_{2}f = 0\} $, where $ f = \displaystyle{ \sum_{i, j = 0}^{2}d_{i j} x_{i} x_{j}}$. Then there exists $ \mathcal{D'} = \{D\} $, where $ D \subset \mathbf{P}^2 $ is a smooth cubic curve, such that 
$$ \Omega_{\mathbf{P}^2}^{1}(\log \mathcal{D}) \cong \Omega_{\mathbf{P}^2}^{1}(\log \mathcal{D'})(1). $$
\end{theorem}

\begin{proof}
We want to prove that there exists a homogeneous polynomial $ g $ of degree $ 3 $ in the variables $ x_{0}, x_{1}, x_{2} $ such that 
$$ <\partial_{0}g, \partial_{1}g, \partial_{2}g > = <-x_{0}\partial_{0}f, -x_{1}\partial_{1}f, -x_{2}\partial_{2}f> $$
holds. 
This is equivalent to say that the partial derivatives of $ g $ have to satisfy  
\begin{equation}\label{eq:cubica0}
\partial_{0}g = a_{0}(-x_{0}\partial_{0}f)+a_{1}(-x_{1}\partial_{1}f)+a_{2}(-x_{2}\partial_{2}f)
\end{equation}
\begin{equation}\label{eq:cubica1}
\partial_{1}g = b_{0}(-x_{0}\partial_{0}f)+b_{1}(-x_{1}\partial_{1}f)+b_{2}(-x_{2}\partial_{2}f) 
\end{equation}
\begin{equation}\label{eq:cubica2}
\partial_{2}g = c_{0}(-x_{0}\partial_{0}f)+c_{1}(-x_{1}\partial_{1}f)+c_{2}(-x_{2}\partial_{2}f)
\end{equation}
for certain complex coefficients $ a_{0},a_{1},a_{2},b_{0},b_{1},b_{2},c_{0},c_{1},c_{2} $. Since, by Schwarz's theorem, $ \partial_{i}\partial_{j}g = \partial_{j}\partial_{i}g $ for all $ i,j \in \{0,1,2\} $, the previous conditions become
$$ a_{0}\partial_{1}(x_{0}\partial_{0}f)+a_{1}\partial_{1}(x_{1}\partial_{1}f)+a_{2}\partial_{1}(x_{2}\partial_{2}f)= b_{0}\partial_{0}(x_{0}\partial_{0}f)+b_{1}\partial_{0}(x_{1}\partial_{1}f)+b_{2}\partial_{0}(x_{2}\partial_{2}f) $$
$$ a_{0}\partial_{2}(x_{0}\partial_{0}f)+a_{1}\partial_{2}(x_{1}\partial_{1}f)+a_{2}\partial_{2}(x_{2}\partial_{2}f)= c_{0}\partial_{0}(x_{0}\partial_{0}f)+c_{1}\partial_{0}(x_{1}\partial_{1}f)+c_{2}\partial_{0}(x_{2}\partial_{2}f) $$
$$ b_{0}\partial_{2}(x_{0}\partial_{0}f)+b_{1}\partial_{2}(x_{1}\partial_{1}f)+b_{2}\partial_{2}(x_{2}\partial_{2}f)= c_{0}\partial_{1}(x_{0}\partial_{0}f)+c_{1}\partial_{1}(x_{1}\partial_{1}f)+c_{2}\partial_{1}(x_{2}\partial_{2}f). $$
Let denote by $ \{a_{ij}\}, \{b_{ij}\}, \{c_{ij}\} $ the coefficients of $ x_{k}\partial_{k}f $ for $ k \in \{0,1,2\} $; by using the identity principle for polynomials we get the following system of nine equations with variables $ a_{0},a_{1},a_{2},b_{0},b_{1},b_{2},c_{0},c_{1},c_{2} $:

\vfill\eject

$$ a_{01}a_{0}+b_{01}a_{1}+c_{01}a_{2}=a_{00}b_{0}+b_{00}b_{1}+c_{00}b_{2} $$
$$ a_{11}a_{0}+b_{11}a_{1}+c_{11}a_{2}=a_{01}b_{0}+b_{01}b_{1}+c_{01}b_{2} $$
$$ a_{12}a_{0}+b_{12}a_{1}+c_{12}a_{2}=a_{02}b_{0}+b_{02}b_{1}+c_{02}b_{2} $$
$$ a_{02}a_{0}+b_{02}a_{1}+c_{02}a_{2}=a_{00}c_{0}+b_{00}c_{1}+c_{00}c_{2} $$
$$ a_{12}a_{0}+b_{12}a_{1}+c_{12}a_{2}=a_{01}c_{0}+b_{01}c_{1}+c_{01}c_{2} $$
$$ a_{22}a_{0}+b_{22}a_{1}+c_{22}a_{2}=a_{02}c_{0}+b_{02}c_{1}+c_{02}c_{2} $$
$$ a_{02}b_{0}+b_{02}b_{1}+c_{02}b_{2}=a_{01}c_{0}+b_{01}c_{1}+c_{01}c_{2} $$
$$ a_{12}b_{0}+b_{12}b_{1}+c_{12}b_{2}=a_{11}c_{0}+b_{11}c_{1}+c_{11}c_{2} $$
$$ a_{22}b_{0}+b_{22}b_{1}+c_{22}b_{2}=a_{12}c_{0}+b_{12}c_{1}+c_{12}c_{2}. $$
We remark that $ a_{ij}, b_{ij}, c_{ij} $ depend on the coefficients $ d_{ij} $ of the conic in our arrangement. In this sense the $ 9 $ by $ 9 $ matrix associated to the system turns to be the following:
$$ \pmatrix{ 
d_{01} & d_{01} & 0 & -2d_{00} & 0 & 0 & 0 & 0 & 0 \cr 
0 & 2d_{11} & 0 & -d_{01} & -d_{01} & 0 & 0 & 0 & 0 \cr 
0 & d_{12} & d_{12} & -d_{02} & 0 & -d_{02} & 0 & 0 & 0 \cr 
d_{02} & 0 & d_{02} & 0 & 0 & 0 & -2d_{00} & 0 & 0 \cr
0 & d_{12} & d_{12} & 0 & 0 & 0 & -d_{01} & -d_{01} & 0 \cr
0 & 0 & 2d_{22} & 0 & 0 & 0 & -d_{02} & 0 & -d_{02} \cr
0 & 0 & 0 & d_{02} & 0 & d_{02} & -d_{01} & -d_{01} & 0 \cr
0 & 0 & 0 & 0 & d_{12} & d_{12} & 0 & -2d_{11} & 0 \cr
0 & 0 & 0 & 0 & 0 & 2d_{22} & 0 & -d_{12} & -d_{12} \cr
}. $$
It's not hard to prove that the rank of this matrix is $ 8 $, that is the dimension of the space of solutions of the aforementioned system is $ \infty^{1} $. So assume that $ \overline{a}_{0},\overline{a}_{1},\overline{a}_{2},\overline{b}_{0},\overline{b}_{1},\overline{b}_{2},\overline{c}_{0},\overline{c}_{1},\overline{c}_{2} $ solve our system, we want to find a cubic polynomial $ g $ such that conditions (\ref{eq:cubica0}), (\ref{eq:cubica1}), (\ref{eq:cubica2}) are satisfied. Let integrate (\ref{eq:cubica0}) with respect to $ x_{0} $, we get
\begin{equation}\label{eq:cubicax0}
g(x_{0},x_{1},x_{2})=-2\overline{a}_{0}\left({x_{0}^3 \over 3}a_{00}+{x_{0}^2x_{1} \over 2}a_{01}+{x_{0}^2x_{2} \over 2}a_{02} \right)+
\end{equation}
$$ \quad\quad\quad\quad\quad -2\overline{a}_{1}\left({x_{0}^2x_{1} \over 2}a_{01}+x_{1}^2x_{0}a_{11}+x_{0}x_{1}x_{2}a_{12} \right)+ $$ 
$$ \quad -2\overline{a}_{2}\left({x_{0}^2x_{2} \over 2}a_{02}+x_{0}x_{1}x_{2}a_{12}+x_{0}x_{2}^2a_{22} \right)+h(x_{1},x_{2}) $$
where $ h $ is a function of $ x_{1},x_{2} $ to be determined. 
If we equate the expression of the partial derivative of $ g $ with respect to $ x_{1} $ coming from (\ref{eq:cubicax0}) with the one in (\ref{eq:cubica1}) and then we integrate with respect to $ x_{1} $ we get the following expression for $ h $:
$$ h(x_{1},x_{2}) = \overline{a}_{0}x_{0}^2x_{1}a_{01}+2\overline{a}_{1}\left({x_{0}^2x_{1} \over 2}a_{01}+x_{0}x_{1}^2a_{11}+x_{0}x_{1}x_{2}a_{12} \right) + 2\overline{a}_{2}x_{0}x_{1}x_{2}a_{12} + $$
$$ \,\,\, -\overline{b}_{0}\left(2x_{0}^2x_{1}a_{00}+x_{0}x_{1}^2a_{01}+2x_{0}x_{1}x_{2}a_{02} \right)  -\overline{b}_{1}\left(x_{0}x_{1}^2a_{01}+2{x_{1}^3 \over 3}a_{11}+x_{1}^2x_{2}a_{12} \right) + $$
\begin{equation}\label{eq:cubicax1}
-\overline{b}_{2}\left(2x_{0}x_{1}x_{2}a_{02}+x_{1}^2x_{2}a_{12}+2x_{1}x_{2}^{2}a_{22} \right) + i(x_{2}) \quad\quad\quad\quad\quad\quad\quad\quad\quad \,\,
\end{equation}
where we have to determine the function $ i(x_{2}) $.
Finally, if we compare the partial derivative of (\ref{eq:cubicax0}) with respect to $ x_{2} $ with (\ref{eq:cubica2}), using also  (\ref{eq:cubicax1}) and then we integrate with respect to $ x_{2} $, we can find explicitly $ i $, so that
$$ g(x_{0},x_{1},x_{2}) = - {2 \over 3}\overline{a}_{0}a_{00}x_{0}^3-2\overline{b}_{0}a_{00}x_{0}^2x_{1}-2\overline{a}_{1}a_{11}x_{0}x_{1}^2- {2 \over 3}\overline{b}_{1}a_{11}x_{1}^3-2\overline{a}_{2}a_{22}x_{0}x_{2}^2 + $$
$$ -2\overline{c}_{0}a_{00}x_{0}^2x_{2}-2(\overline{a}_{1}+\overline{a}_{2})a_{12}x_{0}x_{1}x_{2}-2\overline{c}_{1}a_{11}x_{1}^2x_{2}
-2\overline{b}_{2}a_{22}x_{1}x_{2}^2- {2 \over 3}\overline{c}_{2}a_{22}x_{2}^3 $$
is the required polynomial.
\end{proof}

\begin{remark}
From the proof of theorem \ref{T:3lines-conic} it follows also Hermite's theorem (1868), which asserts that a net of conics can be regarded as the net of the polar conics with respect to a given cubic curve, \cite{EC}.
\end{remark}

\begin{remark}
If we require that 
$$ \partial_{i}g = x_{i}\partial_{i}f $$
for all $ i \in \{0,1,2\} $, then necessarily it has to be 
\begin{equation}\label{eq:fermatcubic}
g(x_{0},x_{1},x_{2}) = {2 \over 3}a_{00}x_{0}^3+{2 \over 3}a_{11}x_{1}^3+{2 \over 3}a_{22}x_{2}^3, 
\end{equation}
provided that the conic is given in diagonal form by
$$ f(x_{0},x_{1},x_{2}) = a_{00}x_{0}^2+a_{11}x_{1}^2+a_{22}x_{2}^2. $$
So, let $ \{x_{0}x_{1}x_{2}f=0\} $ and $ \{x_{0}x_{1}x_{2}f'=0\} $ be two arrangements with normal crossings in $ \mathbf{P}^2 $ each of which made of three lines and a smooth conic defined by a diagonalized quadratic form. Each arrangement corresponds to a logarithmic bundle which is isomorphic to the logarithmic bundle of a smooth cubic like the one in (\ref{eq:fermatcubic}). As we can see in remark 4.8, two cubics which are both Fermat yield isomorphic logarithmic bundles. Thus our line-conic arrangements have isomorphic logarithmic bundles too.
\end{remark}

\end{document}